\documentclass{article}
\usepackage[utf8]{inputenc}
\usepackage[T1]{fontenc}
\usepackage[english]{babel}
\usepackage{enumitem}
\setlist{nosep} % or \setlist{noitemsep} to leave space around whole list

\usepackage{stackrel}

\let\OLDthebibliography\thebibliography
\renewcommand\thebibliography[1]{
  \OLDthebibliography{#1}
  \setlength{\parskip}{0pt}
  \setlength{\itemsep}{0pt plus 0.3ex}
}

\usepackage{comment}

\usepackage{mathtools}
\usepackage{amsfonts}
\usepackage{amssymb}
\usepackage{amsthm}
\usepackage[subnum]{cases}

\usepackage[colorlinks=true]{hyperref}
\hypersetup{urlcolor=blue, citecolor=red}

\numberwithin{equation}{section}
\newtheorem{theorem}{Theorem}[section]

\newtheorem{lemma}[theorem]{Lemma}
\newtheorem{proposition}[theorem]{Proposition}
\newtheorem{claim}[theorem]{Claim}
\newtheorem{definition}[theorem]{Definition}
\theoremstyle{remark}
\newtheorem{remark}[theorem]{Remark}
\newtheorem{conjecture}[theorem]{Conjecture}

\usepackage{fullpage}
\usepackage{graphicx}
\graphicspath{ {./Imagens_Final/} }
\usepackage{multicol} % Required for multiple columns
\usepackage{wrapfig}
\usepackage{pgf,tikz,pgfplots}
\usetikzlibrary{shapes,arrows} % Tikz libraries required for the flow chart in the template
\usetikzlibrary{automata,positioning}
\usetikzlibrary{decorations.markings}
\tikzset{->-/.style={decoration={
			markings,
			mark=at position #1 with {\arrow{>}}},postaction={decorate}}}
\usepackage{float}%figuras dispostas em colunas lado a lado

\newcommand{\N}{\mathbb{N}}

\newcommand{\R}{\mathbb{R}}

\newcommand{\G}{\mathcal{G}}
\newcommand{\norm}[3]{\|#1\|_{L^{#2}(#3)}}

\DeclareMathOperator{\sech}{sech}

\DeclareMathOperator{\supp}{supp}

\DeclareMathOperator{\sgn}{sgn}

\title{Classification and stability of positive solutions to the  NLS equation on the $\mathcal{T}$-metric graph}

\author{Francisco Agostinho, Sim\~ao Correia and
	Hugo Tavares}

\date{\today} %leave blank

\begin{document}
\maketitle

\begin{abstract}
Given $\lambda>0$ and $p>2$, we present a complete classification of the positive $H^1$-solutions of the equation $-u''+\lambda u=|u|^{p-2}u$ on the $\mathcal{T}$-metric graph (consisting of two unbounded edges and a terminal edge of length $\ell>0$, all joined together at a single vertex). This study implies, in particular, the uniqueness of action ground states. Moreover, for $p\sim 6^-$, the notions of action and energy ground states do not coincide and energy ground states are not unique. In the $L^2$-supercritical case $p>6$, we prove that, for $\lambda\sim 0^+$ and $\lambda\sim +\infty$, action ground states are orbitally unstable for the flow generated by the associated time-dependent NLS equation $i\partial_tu + \partial^2_{xx} u + |u|^{p-2}u=0$.  Finally, we provide numerical evidence of the uniqueness of energy ground states for $p\sim 2^+$ and of the existence of both stable and unstable action ground states for $p\sim6$.
 
\end{abstract}

\section{Introduction}

\subsection{Setting of the problem}
In this work, we are interested in the existence, multiplicity and characterization of $H^1$ positive solutions to
\begin{equation}\label{starionaryNLS}
-u'' + \lambda u = |u|^{p-2}u
\end{equation}
over a non-compact metric graph, for $\lambda>0$ and $p>2$. In the last decade, this problem has been the focus of active research. On the one hand, solutions to the elliptic problem relate to stationary solutions of the nonlinear Schr\"odinger equation
\begin{equation}\label{NLSE_intro}
i\partial_tu + \partial_{xx}^2
u + |u|^{p-2}u=0. 
\end{equation}
 When posed over $\R^d$, the initial-value problem for equation \eqref{NLSE_intro} has one of the most complete qualitative theories among nonlinear dispersive equations, ranging from a sharp local well-posedness theory to very precise descriptions of the asymptotic behavior of global solutions \cite{cazenave2003semilinear,tao2006nonlinear}. Many of the qualitative properties rely on the precise description of \textit{dispersion}, that is, the spatial decoupling of waves with different frequencies. However, in the case of other spatial domains, waves can reinteract at later times, due either to the presence of closed geodesics (as in the case of the torus $\mathbb{T}^d)$ or due to \textit{boundary reflections}, as is the case of bounded domains in $\R^d$ or of metric graphs. Consequently, despite the important applications in the propagation of beams in nonlinear optic fibers and Bose-Einstein Condensate theory (\cite{adami2016ground,berkolaiko2013introduction,burioni2001bose,noja2014nonlinear} and references therein), the qualitative theory for equation \eqref{NLSE_intro} posed over a metric graph is still largely underdeveloped.

On the other hand, the geometry and topology of the graph greatly influences the existence and qualitative properties of the solutions to \eqref{starionaryNLS}, see for instance \cite{adami2015nls,adami2016threshold}. As an example, the mere existence of $L^2$ normalized solutions, completely classified in the case of the real line since the '90s (see  \cite{cazenave2003semilinear}), is still an open problem for specific metric graphs.
As such, it is impossible to analyze the problem over all graphs. In order to provide a deeper and meaningful analysis on the properties of positive solutions, we have chosen to restrict our focus to a particular class of graphs, the $\mathcal{T}$-graphs. A $\mathcal{T}$-graph $\G_\ell$ is a metric graph consisting of two half-lines, $e_1$ and $e_2$, and a terminal edge of length $\ell>0$, all attached at the same vertex, which we denote by $\mathbf{0}$. The terminal vertex of the terminal edge we denote by $\boldsymbol{\ell}$ (see Figure \ref{fig:3}).
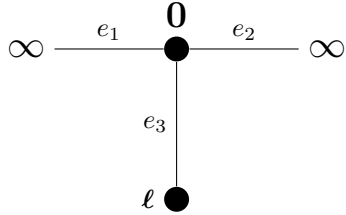
\begin{figure}[h]
\centering
\begin{tikzpicture}[node distance=2cm,  every loop/.style={}]
				\node(a)		{\Large $\infty$};
	\node[circle, fill=black, label=above:\Large $\mathbf{0}$](b)[right of=a]{};
	\node[circle, fill=black, label=left: $\boldsymbol{\ell}$](c)[below of=b]{};
	\node(d)[right of=b]{\Large $\infty$};
\path (a)edge node[above]{$e_1$}(b)
	   (b)edge node[left]{$e_3$} (c)
	      edge node[above]{$e_2$} (d)
;	   	
	\end{tikzpicture}
   \caption{The $\mathcal{T}$-graph $\mathcal{G}_\ell$.}
    \label{fig:3}
\end{figure}

In this work, we describe completely the positive solutions of \eqref{starionaryNLS} in $\mathcal{G}_\ell$. Using this characterization, we show that there exists a unique positive solution of \eqref{starionaryNLS} that achieves the action ground state level. Furthermore, in the $L^2-$subcritical case, this action ground state is not necessarily an energy ground state. We finish by proving orbital instability of action ground states through the flow generated by \eqref{NLSE_intro}, in the $L^2-$supercritical case, for small and large $\lambda$.

\subsection{Preliminaries and Main Results}

Throughout this work we identify, for $i=1,2$, $e_i$ with the interval $I_{i}:=[0,\infty)$, and $e_3$ with $I_{3}:=[0,\ell]$. A function $u$ in $\G_\ell$ is a triple $u=(u_{1},u_{2},u_{3})$, where $u_{i}:I_{i}\to\R$. The functional spaces we work with, namely $L^p(\G_\ell)$ and $H^1(\G_\ell)$, are defined in the standard way; we refer the reader to \cite{berkolaiko2013introduction}. With the previous identification, $u\in H^1(\mathcal{G}_\ell)$ is a positive (weak) solution of \eqref{starionaryNLS} over $\mathcal{G}_\ell$ if, and only if,
\begin{equation}\label{eqn.boundstateprob_intro}
\begin{cases}
-u_i''+\lambda u_i=|u_i|^{p-2}u_i,\ \text{in each edge}\ I_{i},\\
\sum_{e\prec \mathbf{v}}\frac{du}{dx_e}(\mathbf{v})=0,\quad \mathbf{v}\in\{\mathbf{0},\boldsymbol{\ell}\},\\
u>0,
\end{cases}
\end{equation}
where $e\prec \mathbf{v}$ means that the edge $e$ is incident to the vertex $\mathbf{v}$. In other words,
\[
u_{1}'(0)+ u_{2}'(0)+ u_{3}'(0)=0,\quad  u_{3}'(\ell)=0,
\]
which are the standard \emph{Neumann-Kirchoff boundary conditions}.

The relation between the parameters $\ell$ and $\lambda$ in \eqref{eqn.boundstateprob_intro} is crucial in our analysis.

\begin{lemma}\label{lema.scaling}
    Let  $\lambda=\ell^2>0$.
A function $u_\ell$ is a solution of the equation
\begin{equation}\label{eqnforL}
    -u'' +u=|u|^{p-2}u,\ \text{in}\ \G_\ell,
\end{equation}
if, and only if, the function $u^\lambda(x):=\ell^\frac{2}{p-2}u_\ell(\ell x)=\lambda^\frac{1}{p-2}u_\ell(\lambda^\frac{1}{2}x)$ is a solution of the equation
\begin{equation}
    -u''+\lambda u=|u|^{p-2}u,\ \text{in}\  \G_1. 
\end{equation}
\end{lemma}

In particular, this gives us the freedom of fixing a parameter, say $\lambda$, and work only with varying $\ell$ (or vice-versa). Also fundamental throughout this work is the real line soliton $\varphi^\lambda $, that is,
the unique positive, symmetric, decreasing solution to the equation
\begin{equation}\label{eq:soliton}
-\varphi''+ \lambda \varphi=|\varphi|^{p-2}\varphi\ \text{in}\ \R.
\end{equation}

It can easily be seen that any solution of \eqref{eqn.boundstateprob_intro} is, on $e_1$ and $e_2$, a portion of the soliton $\varphi^\lambda$, see Proposition \ref{prop:sol_halfline} below. There are, therefore, only three possible classes of positive solutions:
\begin{itemize}
	\item Type $\mathcal{A}$ solutions are symmetric and strictly decreasing on the half lines (with respect to the vertex, $\mathbf{0}$),
	\item  Type $\mathcal{B}$ solutions are also symmetric on the half lines, but are not monotone on each edge, attaining the maximum of the soliton $\varphi^\lambda$ in the interior of the half-lines. 
	\item Type $\mathcal{C}$ solutions are those that, after the identification of  $e_1\cup e_2$ with $\R$, are equal to a translation of the soliton $\varphi^\lambda$. 
\end{itemize}
%A representation of all three types of solutions can be seen in Figure \ref{fig:figure.5}.

\begin{figure}[!ht]
    \centering
\begin{minipage}{0.3\linewidth}
\centering
    \includegraphics[scale=0.25]{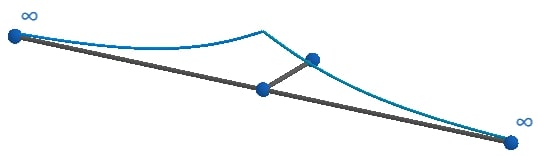}
\end{minipage}
\hspace{5pt}
\begin{minipage}{0.3\linewidth}
\centering
    \includegraphics[scale=0.25]{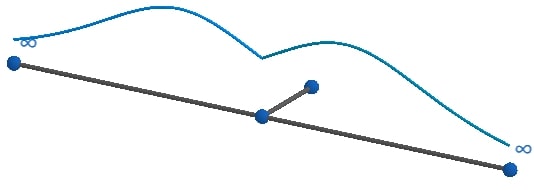}
\end{minipage}
\hspace{5pt}
\begin{minipage}{0.3\linewidth}
\centering
    \includegraphics[scale=0.25]{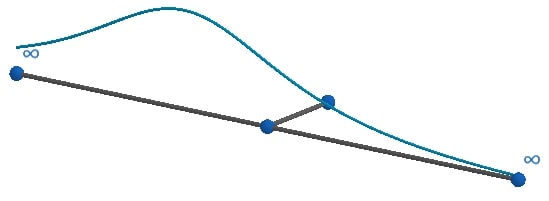}
\end{minipage}
    \caption{From left to right we have the graph of a type $\mathcal{A}$, type $\mathcal{B}$ and type $\mathcal{C}$ solutions restricted to the half lines.}
    \label{fig:figure.5}
\end{figure}

Using the Neumann-Kirchoff conditions, together with the fact that the solution on the half-lines is explicit, the study of positive solutions is reduced to the analysis of  particular overdetermined problems only on the terminal edge $e_3$, see Section \ref{sec:characterization} below for the details. To study these problems we apply techniques from ordinary differential equations. Given that the equation has an Hamiltonian structure, we analyze the orbits in its phase plane for each type of solution, $\mathcal{A}$, $\mathcal{B}$ and $\mathcal{C}$. Inspired by the works of Pierotti, Soave and Verzini \cite{pierotti2021local} and Kairzhan \textit{et al}. \cite{kairzhan2021standing}, we show that the characterization of solutions can be done by analyzing the monotonicity of what we call \textit{length function}. This function corresponds to a portion of the period function that is associated to the orbits of solutions to the NLS equation in the phase plane. 
%This lead to our first main result:

\begin{theorem}\label{Theo1:MAIN}
Fix $\lambda>0$, $p>2$ and let $\G_\ell$ be the $\mathcal{T}$-graph with a terminal edge of length $\ell$. Then:
\begin{enumerate}
    \item For each $\ell>0$, there exists a unique solution of type $\mathcal{A}$.
    \item There are no solutions of type $\mathcal{B}$.
    \item Set $\ell^*:=\frac{\pi}{\sqrt{\lambda(p-2)}}>0$. If  $\ell\in(n\ell^*,(n+1)\ell^*]$, with $n\in\N_0$, then there exist exactly $4n+2$ solutions of type $\mathcal{C}$.
\end{enumerate}
\end{theorem}
Among all solutions determined in Theorem \ref{Theo1:MAIN}, we are particularly interested on \textit{ground states}, that is, solutions that minimize certain functionals (see \cite{de2023notion} for a very complete discussion in the framework of metric graphs). Observing that solutions to \eqref{starionaryNLS} in $H^1\left(\mathcal{G}_\ell\right)$ are critical points of the action functional 
\begin{equation*}
   S_\lambda(u,\G_\ell)=\frac{1}{2}\int_{\G_\ell}(|{u'}|^2+\lambda u^2)dx-\frac{1}{p}\int_{\G_\ell}|u|^pdx,
\end{equation*}
it is natural to look for solutions that achieve the \textit{least action level}:
\begin{equation}\label{ActionProbI}   \mathcal{S}_{\G_\ell}(\lambda)=\inf\left\{S_\lambda(u,\G_\ell),\ u\in H^1(\G_\ell)\setminus\left\{0\right\}: -u''+ \lambda u=|u|^{p-2}u\right\}.
\end{equation}
In the second part of this work, we will be mostly focused on uniqueness and stability of solutions to  \eqref{starionaryNLS} achieving the level $\mathcal{S}_{\G_\ell}(\lambda)$, which we denote  by \textit{action ground states}.

Another common point of view is to interpret solutions to \eqref{starionaryNLS} as critical points of the \textit{energy functional},
\begin{equation*}
    E(u,\G_\ell)=\frac{1}{2}\int_{\G_\ell}|{u'}|^2dx-\frac{1}{p}\int_{\G_\ell}|u|^p dx,
\end{equation*}
constrained to an $L^2$-sphere. In this situation, the parameter $\lambda$ in the equation is not fixed \textit{a priori}, and is a Lagrange multiplier. Within this framework, given a \emph{mass} $\nu>0$, one is interested in \textit{energy ground states}, i.e., minimizers of the problem
 \begin{equation}\label{egsprobI}
      \mathcal{E}_{\G_\ell}(\nu)= \inf\left\{E(u,\G_\ell),\ u\in H^1(\G_\ell),\  \int_{\G_\ell}|u|^2dx=\nu\right\}.
    \end{equation}

For $p\in (2,6)$, the energy minimization problem \eqref{egsprobI} has been well-understood in several classes of graphs which include compact graphs \cite{cacciapuoti2018variational,pelinovsky2021edge}, non-compact graphs with half-lines  \cite{adami2016stable,adami2014constrained} and periodic graphs \cite{Simonegrid2018,SimoneHoneycomb2019}. Results for metric trees can be seen for example in \cite{SimoneTree2020}. For a more extensive reference list, we refer the reader to the introduction in \cite{dovetta2020uniqueness}. 

Concerning the $\mathcal{T}-$graph, the existence of energy ground states for every mass $\nu>0$ was proven in the $L^2$-subcritical case, $p\in(2,6)$, in \cite{adami2015nls,adami2016threshold}, via \textit{concentration-compactness} methods. Furthermore, therein it was shown that they are strictly decreasing on the half-lines, $e_1,e_2$ and monotone on the terminal edge, $e_3$, with a maximum at the tip, see \cite{adami2015nls}. In particular, they correspond to type $\mathcal{A}$ solutions. The critical case $p=6$ was later analyzed in \cite{adami2017negative}. In this case, energy ground states do\textit{ not} exist in $\mathcal{T}-$graphs for any $\nu>0$, due to the topology of the graph.

While in $\R^d$ and $p\in (2,6)$ the notions of action and energy ground states are equivalent, for domains which are not scale invariant the relation between both notions is not yet fully understood. In the classical case of an open subset of the Eucliden space, energy ground states are always action ground states (see \cite[Theorem 1.3]{dovetta2023action}). The converse is true if $\Omega$ is a ball \cite{NorisTavaresVerzini}, but may be false in general. 
%In the setting of metric graphs, the notions may or may not the equivalent, see \cite{de2023notion}. 
In our next two main results we provide a deep understanding of this topic in the special case of the $\mathcal{T}$-graph.

\begin{theorem}\label{thm:existence_uniqueness_action}
   Let $p>2$. Given $\ell,\lambda>0$,  there exists a unique action ground state solution. This is a type $\mathcal{A}$ solution. In particular, every energy ground state is an action ground state for the associated Lagrange multiplier.
\end{theorem}

The fact that energy ground states are action ground states was already known for the $\mathcal{T}$-graph, by combining  \cite[Theorem 1.3]{dovetta2023action} (where the result is stated for functions defined in Euclidean spaces, but extends naturally for graphs) with \cite{de2023notion}. In our paper, we provide an alternative proof, via uniqueness results.

As said before, a natural question is whether the converse also holds, that is, if action ground states are always energy ground state solutions. We show that in the $\mathcal{T}$-graph this is false, in general.
\begin{theorem}\label{Theo2.Main} 
 Suppose that $p\in(6-\varepsilon,6+\varepsilon)$, for some $\varepsilon>0$  sufficiently small. Then there exists a range of masses $[\nu_1,\nu_2]$ with the following property. Given $\nu\in[\nu_1,\nu_2]$, there exist at least two positive and distinct
$\lambda_1,\ \lambda_2$  such that, if $u^{\lambda_1},u^{\lambda_2}$ are the unique action ground state solutions of \eqref{starionaryNLS}, then
\begin{equation}\label{eqn.NormL2Igual}
\nu=\|u^{\lambda_1}\|_{L^2}^2= \| u^{\lambda_2}\|_{L^2}^2.
\end{equation}
In particular, if $p\in(6-\varepsilon,6)$, there exist action ground states which are \emph{not} energy ground states.
\end{theorem}

Observe that, once one proves \eqref{eqn.NormL2Igual}, the fact that action ground states are not always energy ground states is a direct consequence of a result from \cite{dovetta2020uniqueness} which, in our framework, reads as follows:
\begin{equation*}\label{quote_uniqueness}
\text{Let $p\in (2,6)$. Then, there exists a unique energy ground state  up to a countable set of masses $\nu>0$.}
\end{equation*}
The question of uniqueness of energy ground states, albeit natural, is a more delicate matter mainly due to the fact that, as opposed to the classical case $\G=\R$, the parameter $\lambda$ is a Lagrange multiplier that cannot be determined as a function of the parameter $\nu$. This is due to the lack of scaling invariance of the graph. Thanks to a suggestion of a referee to a previous version of this paper, we were able to prove that, indeed, energy ground states are not unique for some masses $\nu>0$ near the $L^2$- critical exponent, see the following result and Section \ref{sec:groundStates3} below.
\begin{theorem}\label{TheoMain5.NonUniquenessEGS}
    Let $\varepsilon>0$ be sufficiently small. Then, for all $p\in(6-\varepsilon,6)$, there exist two distinct energy ground states with the same mass.
\end{theorem}

A graph that admits more than one energy ground state was already known to exist, as one can see in \cite[Theorem 2.11]{dovetta2020uniqueness}. However, the metric graph constructed in the aforementioned reference exhibits a rather complicated topology. Our result shows that even graphs with very simple topologies, such as the $\mathcal{T}-$graph, can admit masses for which energy ground states are not unique.
On the other hand, we then provide sufficient conditions that ensure the uniqueness of energy ground states, and show numeric evidence that such conditions hold true for $p\sim 2^+$, which supports the following statement.
\vskip15pt
\begin{conjecture}\label{EGSUniquenessConjecture}
    For $\varepsilon>0$ sufficiently small and any $p\in(2,2+\varepsilon)$, energy ground states are unique for all masses.
\end{conjecture}
\vskip15pt
Finally, in order to discuss dynamical stability, let us recall that the initial value problem associated with \eqref{NLSE_intro} is locally well-posed in $H^1(\mathcal{G}_\ell)$ (see \cite{adami2014constrained}). In particular, given an initial condition $u\big|_{t=0}=u_0$, there exists a maximal time of existence $T=T(u_0)>0$ and a unique solution $u\in C([0,T),H^1(\mathcal{G}_\ell))$ to \eqref{NLSE_intro}. Moreover, if $T(u_0)<\infty$, then the $H^1$ norm must blow-up at $t=T(u_0)$. 

A direct computation shows that, if $u^\lambda$ is a solution to \eqref{eqn.boundstateprob_intro}, then $u(t)=\exp(i\lambda t)u^\lambda$ is the (unique) global solution to \eqref{NLSE_intro} with initial condition $u^\lambda$. Consequently, one may analyze the stability of action ground states through the flow generated by \eqref{NLSE_intro}:

\begin{definition}[Orbital Stability]\label{OrbitalStabilityDef}
\leavevmode
    The \textit{orbit} of a bound state $u^\lambda$ of \eqref{eqn.boundstateprob_intro} is the set $\mathcal{O}(u^\lambda):=\{e^{i\lambda\theta}u^\lambda:\ \theta\in\R\}$. We say that the action ground state $u^\lambda$ is \textit{orbitally stable} if, for every $\epsilon>0$, there exists $\delta>0$ such that
    $$ d\left(u_0,\mathcal{O}(u^\lambda)\right)<\delta\ \ \Rightarrow d\left( u(t),\mathcal{O}(u^\lambda)\right)<\epsilon,\ \forall t>0,$$
    where $u(t)$ is the solution of \eqref{NLSE_intro} with initial data $u_0$, and where, given $w\in H^1(\G_1)$
    $$ d\left(w,\mathcal{O}(u^\lambda)\right)=\inf_{\Psi\in\mathcal{O}(u^\lambda)}\|w-\Psi\|_{H^1(\G_\ell)}.$$
    If $u^\lambda$ is not (orbitally) stable, we say that it is (orbitally) unstable.
\end{definition}

Orbital stability in the case of the euclidean domain $\R^d$ has been studied in the seminal papers \cite{cazenave1982orbital,berestycki1981instability}, where it was shown that orbital stability holds if and only if $p<6$ (that is, in the $L^2$-subcritical case). While the proof of stability requires only conservation of mass and energy (which hold in more general domains), the methods employed to prove instability hinge on the scaling invariance of the domain, and thus are inapplicable in our setting. Indeed, in the case of metric graphs, this is only applicable in the case of star graphs \cite{adami2016stable,adami2014variational,pava2017orbital}.

Afterwards, the stability theory for ground states for general Hamiltonian equations was founded in the celebrated Grillakis-Shatah-Strauss papers \cite{grillakis1987stability, grillakis1990stability} (see the book \cite{kapitula2013spectral} for more developments). The starting point is the existence of a curve of bound states  $\lambda\mapsto u^\lambda$ such that, for each $\lambda>0$, $u^\lambda$ is a non-degenerate critical point of the action functional, with Morse index $1$. Then, the orbital stability of bound states is reduced to the monotonicity of the map $\lambda\mapsto \|u^\lambda\|_{L^2}$. In the context of metric graphs, this theory has been successfully employed in $p=6$ for tadpole graphs \cite{cacciapuoti2014topology,noja2015bifurcations,noja2020standing} and more general graphs, such as flower graphs \cite{kairzhan2021standing,pankov2018nonlinear}. However, in order to analyze the spectrum of the linearized operator and check the Morse index one condition, the authors use explicit formulae for the bound states, which are not available in our context, as we are dealing with a general $p>2$.

In order to bypass the lack of scaling invariance of the $\mathcal{T}$-graph and of spectral analytical tools, one can employ the arguments in \cite{gonccalves1991instability}, where similar difficulties arise. Therein, a delicate analysis of the Hamiltonian flow around the ground state, combined with its variational properties, allows to avoid the nondegeneracy assumption. As such, the instability of ground states can be reduced to the construction of an unstable direction for the flow tangent to the mass-constrained manifold (see also \cite{BahriIbrahimKikuchi2021, COS2022, EsfahaniLevandosky2013}). The application of these ideas in our context leads to the following result:
\begin{theorem}\label{Theo3.MAin} 
    Fix $p>6$ and $\ell>0$. For each $\lambda>0$, let $u^{\lambda}\in H^1(\G_\ell)$ be the corresponding action ground state. Then there exists $\varepsilon>0$ sufficiently small such that, for $\lambda\in (0,\varepsilon)\cup (1/\varepsilon,+\infty)$, there exist a neighborhood of $\mathcal{O}(u^\lambda)$ in $H^1(\G_\ell)$, $\mathcal{V}(\mathcal{O}(u^\lambda))$, and sequences $(u_n^{0,\lambda})_{n\in\N}$ with:
    \begin{itemize}
        \item $u_n^{0,\lambda}\to u^\lambda$ in $H^1(\G_\ell)$;
        \item Solutions of \eqref{NLSE_intro} with initial condition $u_n^{0,\lambda}$ are global and are denoted by $u_n^\lambda$;
        \item $T_*(u_n^{0,\lambda},\mathcal{V}(\mathcal{O}(u^\lambda)))=\sup\left\{t\in\R^+:u^\lambda_{n}(\tau)\in\mathcal{V}(\mathcal{O}(u^\lambda)),\ \text{for}\ \tau\in[0,t]\right\}<\infty$, for all $n\in\N$.
    \end{itemize}
    \vskip5pt
    In other words, the set $\mathcal{O}(u^\lambda)$ is unstable by the flow of \eqref{NLSE_intro}. In particular, $u^\lambda$ is orbitally unstable.
\end{theorem}
Let us briefly explain the idea behind Theorem \ref{Theo3.MAin}. Through scaling, we can fix $\lambda=1$, which leads to $\ell\sim 0^+$ and $\ell\sim +\infty$. On the one hand, if $\ell\sim 0^+$, the graph $\G_\ell$ is a small perturbation of the real line, and thus the action ground state should converge (suitably) to the soliton, which is orbitally unstable in the $L^2$-supercritical case, with unstable direction given by the generator of the scaling invariance in $\R^d$. On the other hand, if $\ell\to +\infty$, the soliton is "pulled" into the terminal edge and thus the action ground state on that edge should
converge to half of the soliton, which is also unstable in the supercritical case (with the analogous unstable direction).

\begin{remark}
    Assuming that the spectral conditions of the Grillakis-Shatah-Strauss theory are verified, we observe numerically (see Figure \ref{fig:figure5}) that, near $p=6$, the curve $\lambda\mapsto\|u^\lambda\|_{L^2}^{2}$ is not monotone. As a consequence, we conjecture the existence of both stable and unstable action ground states near the $L^2$-critical case. In this way, the restriction on $\lambda$ in Theorem \ref{Theo3.MAin} appears to be optimal.
\end{remark}

\vskip10pt
The paper is organized as follows. In Section \ref{sec:characterization}, we provide an exhaustive description of all positive solutions of \eqref{starionaryNLS}. In Section \ref{sec:groundStates1}, we discuss the existence and uniqueness of action ground states. The connection between energy and action ground states is carried out in Section \ref{sec:groundStates2}. In section \ref{sec:groundStates3} we prove the non uniqueness of energy ground states for $p\sim 6^-$ and present some numerical simulations which support the energy ground state uniqueness conjecture for $p\sim 2^+$. Finally, the orbital instability result is proven in Section \ref{sec:instability}.

\section{Characterization of Positive Solutions}\label{sec:characterization}

In this section, we describe all positive solutions to the equation \eqref{starionaryNLS}. By Lemma \ref{lema.scaling}, we may fix the parameter $\lambda=1$ and let the parameter $\ell$ to vary freely. We concern ourselves with the problem
\begin{equation}\label{eqn.boundstateprob}
\begin{cases}
-u''+u=|u|^{p-2}u,\ \text{in}\ \G_\ell,\\
\sum_{e\prec \mathbf{v}}\frac{du}{dx_e}(\mathbf{v})=0,\quad \mathbf{v}\in\{\mathbf{0},\boldsymbol{\ell}\},\\
u>0.
\end{cases}
\end{equation}
Our first result characterizes the solutions to \eqref{eqn.boundstateprob} on the half-lines. 
\begin{proposition}[Solutions are a portion of a soliton on unbounded edges]\label{prop:sol_halfline}
	Suppose $u\in H^1(\G_\ell)$ is a solution to the problem \eqref{eqn.boundstateprob}. Then there exist $y_1,y_2\in\R\setminus\left\{0\right\}$ with $|y_1|=|y_2|$ such that
	\begin{equation}\label{eq:shape_halflines} u_i(x)=\varphi(x+y_i),\ x\geq0,\ i=1,2, \end{equation}
 where $\varphi$ is the real line soliton for $\lambda=1$. We have $y_1=y_2>0$  for type $\mathcal{A}$ solutions, $y_1=y_2<0$ for type $\mathcal{B}$, and $y_1=-y_2$ for type $\mathcal{C}$.
\end{proposition}
\begin{proof}
    Since $u_1(+\infty)=u_2(+\infty)=0$, \eqref{eq:shape_halflines} follows by direct integration of the equation. The continuity of $u$ at the zero vertex implies $|y_1|=|y_2|$. If $y_1=0$, we would have $u_3(0)=\varphi(0)$ and $u_3'(0)=\varphi'(0)$ and thus $u_3\equiv \varphi$ in $[0,\ell]$. This would imply $u'_3(\ell)=\varphi'(\ell)\neq0$, contradicting the Neumann-Kirchoff condition $u_3'(\ell)=0$.
\end{proof}

By the previous result, the solution $u_3$ on the terminal edge is uniquely determined by existence and uniqueness theory of the ODE
\[
-u_3''+u_3=|u_3|^{p-2}u_3\ \text{in}\ [0,\ell],\\
\]
with initial data
\[
    u_3(0)=\varphi(y_1),\quad  u_3'(0)=-2\varphi'(y_1),\ \text{for type $\mathcal{A}$ and type $\mathcal{B}$ solutions},\\
\]
and
\[
u_3(0)=\varphi(y_1),\quad
    u_3'(0)=0,\ \text{for type $\mathcal{C}$ solutions}
\]
(recall the Neumann-Kirchoff conditions). This initial value problem will provide a solution of \eqref{eqn.boundstateprob} if and only if $u_3'(\ell)=0$.

In conclusion: solutions of \eqref{eqn.boundstateprob}, whenever they exist, can only be of type $\mathcal{A}$, $\mathcal{B}$ or $\mathcal{C}$; moreover, they can be completely classified once we understand for which $y>0$ one has solutions to the following superdertemined initial value problems:
\begin{equation}\label{eqn.probu_3}
    \begin{cases}
    -u_3''+u_3=|u_3|^{p-2}u_3,\ \text{in}\ [0,\ell];\\
    u_3(0)=\varphi(y);\\
    u_3'(0)=-2\varphi'(y);\\
    u_3'(\ell)=0,
    \end{cases}
\end{equation}
for  type $\mathcal{A}$ and $\mathcal{B}$ solutions and
\begin{equation}\label{eqn.probu_4}
    \begin{cases}
    -u_3''+u_3=|u_3|^{p-2}u_3,\ \text{in}\ [0,\ell];\\
    u_3(0)=\varphi(y);\\
    u_3'(0)=0=u_3'(\ell),
    \end{cases}
\end{equation} 
for type $\mathcal{C}$ solutions.
\paragraph{Phase plane analysis.} Given
$$ F(x,y)=-\frac{1}{2}y^2+\frac{1}{2}x^2-\frac{1}{p}|x|^p,$$
its contour plot  gives the $(u,u')$-phase plane diagram for the first order nonlinear ODE given by \eqref{eqn:uniq.3}, see Figure \ref{fig:my_label}.  As a phase plane diagram, we see that the level set corresponding to $C=0$ corresponds to two homoclinic orbits. The interior of the homoclinic orbits corresponds to positive values of $C$ and the exterior to negative ones. Furthermore, the phase plane is symmetric with respect to both axis $u=0$ and $u'=0$.

Given $u\in H^1(\G_\ell)$ solving \eqref{eqn.boundstateprob}, by Proposition \ref{prop:sol_halfline} and the decay of $\varphi$ at infinity, we know that in, each edge, we have 
\begin{equation}\label{eqn:uniq.3}
    F(u_i(t),u_i'(t))=-\frac{1}{2}u_i'^2+\frac{1}{2}u_i^2-\frac{1}{p}|u_i|^p=\begin{cases}
    0,\ \text{if}\ i=1,2,\\
    C,\ \text{if}\ i=3,
    \end{cases}
 \end{equation}
where $C\in\R$ can be computed depending on the type of solution we are considering.  Therefore, since $ F(u_i(t),u_i'(t))=0$
for every $t\geq0$ and $i=1,2$, we know that along the half-lines the solution goes along the homoclinc orbit on the phase-plane. At the common vertex $\mathbf{0}$, the solution jumps to another level curve of $F$; it jumps outwards for types $\mathcal{A}$  and $\mathcal{B}$ solutions (as in this case $C<0$, see Lemma \ref{lem:lemma7.3}), and inwards for type $\mathcal{C}$ solutions (since $C>0$, see Lemma \ref{lemma: C>0}). See also Figures \ref{fig:Figure6}, \ref{Fig:FIGURA5}, \ref{fig:Figure7} and \ref{fig:Figure8}.

\begin{figure}
    \centering
    \includegraphics[scale=0.35]{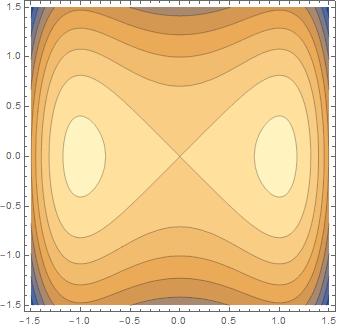}
    \caption{Contour Plot of $F$ for $p=6$.}
    \label{fig:my_label}
\end{figure}

\subsection{Type $\mathcal{A}$ and $\mathcal{B}$ solutions}\label{eq:typeAB}

\begin{lemma}\label{lem:lemma7.3}
Suppose $u\in H^1(\G_\ell)$ is a solution to \eqref{eqn.boundstateprob} of type $\mathcal{A}$ or $\mathcal{B}$. Then $$C=-\frac{3}{2}u_1'(0)^2=-3\left(\frac{1}{2}u_1(0)^2-\frac{1}{p}u_1^p(0)\right).$$
\end{lemma}
\begin{proof}
Recall that, in this graph, the Neumann-Kirchoff conditions together with the symmetry of the solution give, at the vertex, $u_3'(0)^2=4u_1'(0)^2$. Subtracting the equations in \eqref{eqn:uniq.3} the conclusion follows.\qedhere
\end{proof} 

The following result rules out the existence of signed solutions to the initial value problem \eqref{eqn.probu_3} in the type $\mathcal{B}$ case.

\begin{proposition}\label{Proposition:7.3}
For any $\ell>0$, any solution of the initial value problem \eqref{eqn.probu_3}, with $y<0$, is  sign-changing. In particular, there are no solutions to \eqref{eqn.boundstateprob} of type $\mathcal{B}$.
\end{proposition}

\begin{proof}
Without loss of generality, assume that $\ell$ is the first zero of $u'_3$. By the Neumann-Kirchoff conditions we have $u'_3(0)<0$, and thus $u_3$ is decreasing in $[0,\ell]$. By the previous lemma, $F(u_3(\ell),0)=C<0$. Since $F(z,0)>0$ iff $|z|<\varphi(0)$, we have
$|u_3(\ell)|>\varphi(0)>\varphi(y)>u_3(\ell),$ which implies $u_3(\ell)<0$.
\end{proof}

\begin{figure}[!ht]
    \centering
    \includegraphics[scale=0.08]{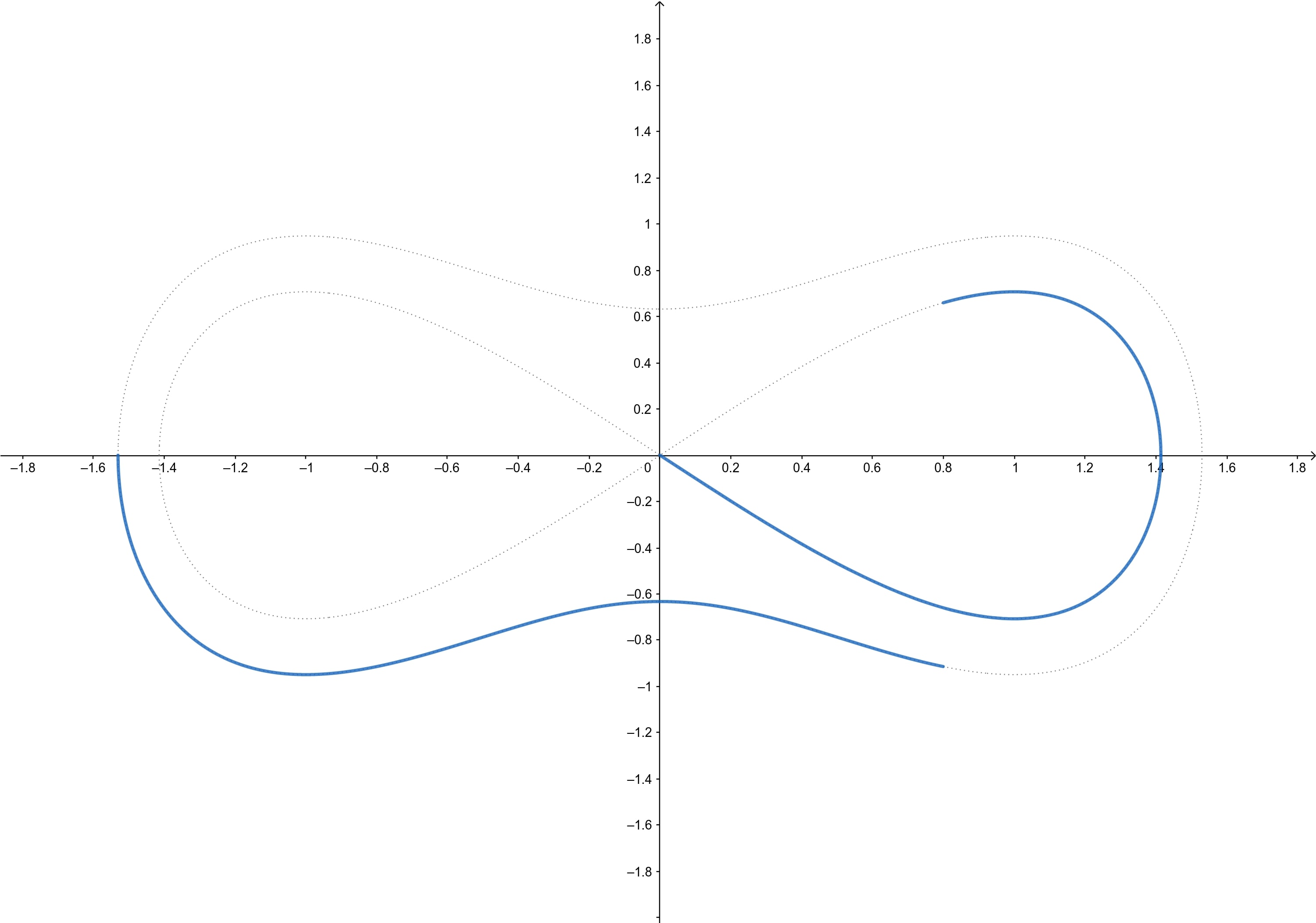}
    \caption{The $(u,u')$-phase plane implies the nonexistence of type $\mathcal{B}$ solutions.}
    \label{fig:Figure6}
\end{figure}
We now focus on type $\mathcal{A}$ solutions. Define the auxiliary function $f:\R^+\to \R$,

\begin{equation}\label{defif}
     f(z)=\frac{1}{2}z^2-\frac{1}{p}z^p.
\end{equation} 
For each $y>0$, let $u_3$ be the solution of the initial value problem
\begin{equation*}\label{eqn.probu_3.1}
 -u_3''+u_3=|u_3|^{p-2}u_3,\ \text{in}\ [0,\infty),\quad  u_3(0)=\varphi(y),\quad
    u_3'(0)=-2\varphi'(y),
\end{equation*}
and let $L>0$ be the first zero of $u_3'$. Since $u_3'(0)=-2\varphi'(y)>0$, we have $u_3(L)>u_3(0)>0$. Arguing as in Proposition \ref{Proposition:7.3}, if there was a second zero of $u_3'$ larger than $L>0$, then the solution would have to be sign-changing. This behaviour can be seen in Figure \ref{Fig:FIGURA5}.

\begin{figure}[!ht]
    \centering
    \includegraphics[scale=0.08]{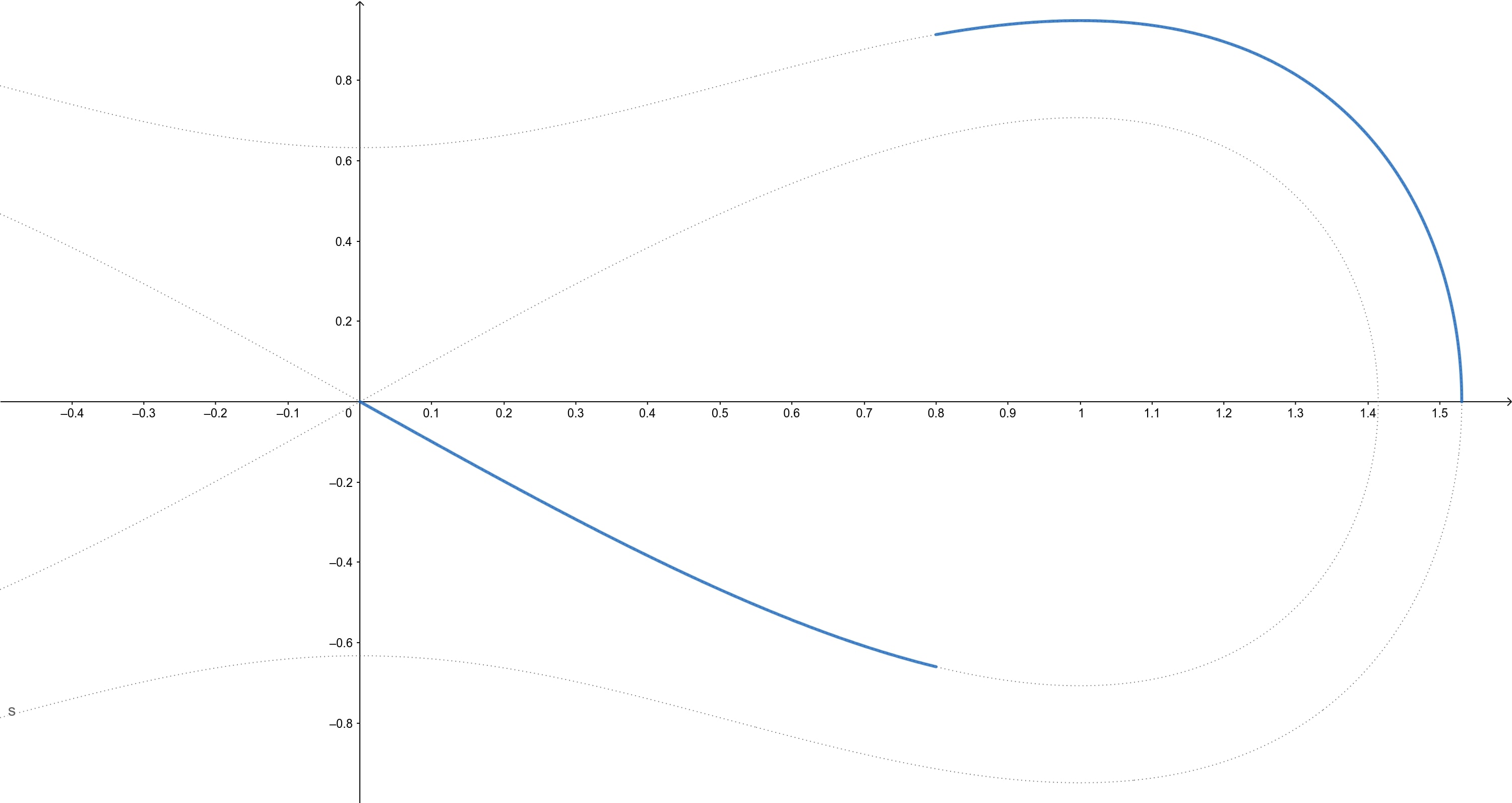}
    \caption{The $(u,u')$-phase plane argument for type $\mathcal{A}$ solutions.}
    \label{Fig:FIGURA5}
\end{figure}
We now show that there exists a relation between $u_3'(L)$  and   $u_3(0)=\varphi(y)$. Since $u_3'>0$, then
$$ L=\int_0^L dx=\int_{u_3(0)}^{u_3(L)}\frac{1}{u_3'(u_3^{-1}(t))}dt.$$
Given that $u_3'(L)=0$, by Lemma \ref{lem:lemma7.3}, we have
$ f(u_3(L))=-3f(u_3(0))$. Taking $f^{-1}$ to be the inverse of $f$ on $[1,\infty)$, since $-3f(u_3(0))<0$, it follows that
\begin{align*}
L=\int_{u_3(0)}^{u_3(L)}\frac{1}{u_3'(u_3^{-1}(t))}dt=\int_{u_3(0)}^{f^{-1}(-3f(u_3(0)))}\frac{1}{u_3'(u_3^{-1}(t))}dt.
\end{align*}
Finally, using again the equation \eqref{eqn:uniq.3} for $u_3$ and solving it for $u_3'>0$, 
$$L=\frac{1}{\sqrt{2}}\int_{u_3(0)}^{f^{-1}(-3f(u_3(0)))}\frac{1}{\sqrt{f(t)+3f(u_3(0))}}dt.$$

\begin{definition}\label{Def L} We define the \textit{length function} as
$$ \left(0,\left(\frac{p}{2}\right)^{\frac{1}{p-2}}\right]\ni z\mapsto L(z):=\frac{1}{\sqrt{2}}\int_z^{f^{-1}(-3f(z))}\frac{1}{\sqrt{f(t)+3f(z)}}dt.$$
\end{definition}

\begin{proposition}[Properties of the length function $L(z)$]\label{prop L(z)_resumo}
The function $L$ is positive, continuous and strictly decreasing in the interval $\left(0,\left(\frac{p}{2}\right)^{\frac{1}{p-2}}\right]$. Furthermore, we have $L(\left(\frac{p}{2}\right)^{\frac{1}{p-2}})=0$ and that $\lim\limits_{z\to 0^{+}}L(z)=\infty$.
\end{proposition}

Assuming Proposition \ref{prop L(z)_resumo} holds, we show the existence and uniqueness of type $\mathcal{A}$ solutions.

\begin{proof}[Proof of Theorem \ref{Theo1:MAIN}, parts 1 and 2.]
For item 1., given $\ell>0$, by Proposition \ref{prop L(z)_resumo}, there exists a unique $y>0$ such that $\varphi(y)=u_3(0)=L^{-1}(\ell)$. In particular, there exists a unique solution of type $\mathcal{A}$.
Item 2. follows from Proposition \ref{Proposition:7.3}.
\end{proof}

The remainder of this section is then devoted to the proof of Proposition \ref{prop L(z)_resumo}. 
We begin with the case $z\leq1$.

\begin{lemma}\label{prop 6.4}
For $z\leq1$, $L$ is continuous and decreasing.
\end{lemma}
\begin{proof}
 For $z\leq 1$, we split $L$ in the following way,

\begin{equation*}\label{eqn.zmenor1}
    L(z)=\frac{1}{\sqrt{2}}\int_z^1\frac{1}{\sqrt{f(t)+3f(z)}}dt+\frac{1}{\sqrt{2}}\int_1^{f^{-1}(-3f(z))}\frac{1}{\sqrt{f(t)+3f(z)}}dt.
\end{equation*}
Differentiating under the integral sign in the first integral, we obtain
\begin{equation*}
    -\frac{1}{\sqrt{8f(z)}}-\frac{3}{2\sqrt{2}}\int_z^1\frac{f'(z)}{\left(f(t)+3f(z)\right)^{3/2}}dt<0.
\end{equation*}
For the second term, we show that it is decreasing with $z$.
For fixed $z\in(0,1]$, take $ t=sf^{-1}(-3f(z))+(1-s)=:1+sI_z$,
where $I_z=f^{-1}(-3f(z))-1>0.$
With this change of variable,
$$\int_1^{f^{-1}(-3f(z))}\frac{1}{\sqrt{f(t)+3f(z)}}dt=\int_0^1\frac{I_z}{\sqrt{f(1+sI_z)+3f(z)}}ds=:\int_0^1K_1(z,s)ds. $$
Differentiating $K_1$ with respect to $z$ and using the fact that $I'_z=-3f'(z)/f'(1+I_z)$,
$$ \frac{\partial K_1}{\partial z}(z,s)=\frac{3f'(z)}{2\left(f(1+I_zs)+3f(z)\right)^\frac{3}{2}f'(1+I_z)}\left[-2\left(f(1+I_zs)+3f(z)\right)+I_zf'(1+I_zs)s-I_zf'(1+I_z)\right].
$$
For $z\leq 1$, the term outside the brackets is negative. Then it is enough to show that
$$ k_1(s)=-2\left(f(1+I_zs)+3f(z)\right)+I_zf'(1+I_zs)s-I_zf'(1+I_z)$$
is positive for fixed $z\leq 1$ and $s\in(0,1)$. For all $s\in(0,1)$,
$$k_1'(s)=-f'(1+I_zs)I_z+I_z^2f''(1+I_zs)s,\qquad k_1''(s)=I_z^3f'''(1+I_zs)s<0.$$ Since $k_1'$ is decreasing and $k_1'(0)=0$ we have that $k_1'<0$ and thus $k_1$ is also decreasing. Together with $k_1(1)=0$, we conclude that $k_1$ is positive in $(0,1)$.
The continuity follows from the Dominated Convergence Theorem.\qedhere
\end{proof}

The case $z\geq1$ is more involved. We introduce the following claim:

\begin{claim}\label{claim}
Let $\psi:(\R^+)^2\to\R$ be defined as
 $$\psi(z,x)= -f'(x)^2-3f'(z)\left(f'(x)+(z-x)f''(x)\right).$$
 Then, for all $z\geq1$, we have $\psi(z,x)\leq0$ for all $x\geq z$.
\end{claim}
 \begin{lemma}\label{lema6.5}
 Suppose that the Claim \ref{claim} holds. Then $z\mapsto L(z)$ is decreasing and continuous for $z\geq 1$.
 \end{lemma}
 \begin{proof}
 Suppose $z\geq 1$. Consider the change of variable $t=z+sI_z$, where $I_z=f^{-1}(-3f(z))-z>0$. Then
 \begin{equation*}\label{Lzmaiorque1}
      L(z)=\frac{1}{\sqrt{2}}\int_0^1\frac{I_z}{\sqrt{f(z+sI_z)+3f(z)}}ds:=\frac{1}{\sqrt{2}}\int_0^1H(z,s)ds.
 \end{equation*}
We have
$$\frac{\partial H}{\partial z}(z,s)=\frac{1}{2(f(z+sI_z)+3f(z))^{3/2}}\left[2I'_z(f(z+I_zs)+3f(z))-I_z(f'(z+I_zs)(1+I'_zs)+3f'(z))\right],$$ 
whose sign is determined by the expression in brackets. Fix now $z\geq1$ and, for $s\in[0,1]$, define 
$$ g(s):=2I'_z(f(z+I_zs)+3f(z))-I_z\left(f'(z+I_zs)(1+I'_zs)+3f'(z)\right).$$
Note that $I'_z=-1-\frac{3f'(z)}{f'(z+I_z)}$ and thus, together with the definition of $I_z$, $g(1)=0$. 
We claim that $g$ is an increasing function, which then yields $g<0$ for $s\in(0,1)$ and the conclusion follows. We have
$$ g'(s)=I_zI'_zf'(z+I_zs)-I_z^2f''(z+I_zs)(1+I'_zs),$$
$$ g''(s)=-I_z^3(1+I'_zs)f'''(z+I_zs).$$
Since $z+I_zs>z$, it follows that $I'_z,\  f'(z+I_z s),\  f''(z+I_zs)<0$ for $s\in[0,1]$. Let $s_0=-1/I'_z\in(0,1)$. Then $(1+I'_zs)>0$ for $s<s_0$ and thus $g'(s)>0$ for $s\in[0,s_0)$.  Now, since $z+I_zs>1$, we have $f'''(z+I_z s)<0$, which yields $g''<0$ in $(s_0,1)$. Thus, $g'\geq 0$ if and only if $g'(1)\geq 0$. The latter follows from Claim \ref{claim} since 
\begin{align*}
g'(1)&=I_zI'_zf'(z+I_z)-I_z^2f''(z+I_z)(1+I'_z)=I_z\left(I'_zf'(z+I_z)-I_zf''(z+I_z)(1+I'_z)\right)\\
&=\frac{I_z}{f'(z+I_z)}\psi(z,z+I_z)\geq 0.
\end{align*}
The continuity follows by the Dominated Convergence Theorem, since the integrand is monotone in $z$.
\end{proof}
 
In the next few lemmas, we prove that Claim \ref{claim} holds. We start with an auxiliary lemma.

\begin{lemma}\label{funcaoh}
Let $p>2$ and define $h:[0,\infty)\to\R$ by 
$h(x)=-2-(p-1)(p-6)x-p(p-1)x^2. $
Then, $h<0$ in $[0,\infty)$.
\end{lemma}

\begin{lemma}\label{Prop2.9}
Suppose that $p\in(2,3]$. Then  Claim \ref{claim} holds true.
\end{lemma}
\begin{proof}
Note that
\begin{equation}\label{eqn:prop6.7.1}
    f'(z)=f'(x)+f''(x)(z-x)+\frac{f'''(c)}{2}(z-x)^2,
\end{equation}
for some $c\in (z,x)$. Plugging \eqref{eqn:prop6.7.1} into the definition of $\psi$, 
\begin{equation*}
    \psi(z,x)=-f'(x)^2-3f'(z)^2+\frac{3}{2}f'(z)f'''(c)(z-x)^2.
\end{equation*}
Since, for $p\in(2,3]$,  $f^{(4)}$ is non-negative, we have that $f'''$ is nondecreasing. Furthermore,
$\displaystyle\frac{3}{2}f'(z)(z-x)^2<0 $
and thus
\begin{equation*}\label{eqn:prop6.7.3}
    \psi(z,x)\leq -f'(x)^2-3f'(z)^2+\frac{3}{2}f'(z)f'''(z)(z-x)^2=:w(x).
\end{equation*}

For some $\xi\in(x,z)$, 
$w(x)=w(z)+w'(z)(x-z)+\frac{w''(z)}{2}(x-z)^2+\frac{w'''(\xi)}{3!}(x-z)^3.$ Notice that 
$$w(z)=-4f'(z)^2<0,\quad w'(z)=-2f'(z)f''(z)<0$$ and 
 \begin{equation*} 
 w''(z)=-2f''(z)^2+f'(z)f'''(z)=-2-(p-1)(p-6)z^{p-2}-p(p-1)z^{2(p-2)}=h(z^{p-2})<0
 \end{equation*}
%and using the explicit expressions for $f',\ f''$ and $f'''$ we have
% \begin{equation*}
     %2''(z)=-2-(p-1)(p-6)z^{p-2}-p(p-1)z^{2(p-2)}=h(z^{p-2})<0
 %\end{equation*}
 by Lemma \ref{funcaoh}. Furthermore, 
$$w'''(\xi)=-2\left[3f''(\xi)f'''(\xi)+f'(\xi)f^{(4)}(\xi)\right]=-2\left[(p-1)(p-2)\xi^{p-3}\left(-p+2(2p-3)\xi^{p-2}\right)\right]\leq0 $$
since
 $\displaystyle \xi> 1$. Thus $w(x)\leq 0$ and Claim \ref{claim} follows.
\end{proof} 
The next results concerns the case $p>3$. 

\begin{lemma}\label{Prop2.10}
Let $n\in\mathbb{N}$ be such that $n\geq3$ and $p\in(n,n+1]$. Then, for each $k\in\left\{0,\cdots,n-1\right\}$,
$$\frac{\partial^k\psi}{\partial x^k}(z,z)<0,\ \text{for}\ z>1. $$
\end{lemma}
\begin{proof}
We have $\psi(z,z)=-4f'(z)^2<0$, $\frac{\partial\psi}{\partial x}(z,z)=-2f'(z)f''(z)<0$ and, by Lemma \ref{funcaoh},
$$\frac{\partial ^2\psi}{\partial x^2}(z,z)=-2f''(z)^2+f'(z)f'''(z)=h(z^{p-2})<0.$$
 Fix  $z>1$ and  $ k\geq 3$. Then
\begin{align*}
\frac{\partial^k\psi}{\partial x^k}(z,x)&=-\sum_{j=0}^{k}\left[{\binom{k}{j}}f^{(k+1-j)}(x)f^{(j+1)}(x)\right]-3f'(z)\left(f^{(k+1)}(x)+\sum_{j=0}^k{\binom{k}{j}}f^{(k+2-j)}(x)\frac{d^j}{dx^j}(z-x)\right)\\
&=-\sum_{j=0}^{k}\left[{\binom{k}{j}}f^{(k+1-j)}(x)f^{(j+1)}(x)\right]-3f'(z)\left(f^{(k+2)}(x)(z-x)-(k-1)f^{(k+1)}(x)\right).
\end{align*}
Set, for $z\geq 1$,
$$ g_k(z):=\frac{\partial ^k\psi}{\partial x^k}(z,z)=\sum_{j=0}^{k}-{\binom{k}{j}}f^{(k+1-j)}(z)f^{(j+1)}(z)+3(k-1)f'(z)f^{(k+1)}(z).$$
Since $p\in(n,n+1]$, $f^{(j)}(z)<0$ for $j=1,\cdots,n+1$ and $f^{(n+2)}\geq0$. Therefore, since $f'(1)=0$, 
$$g_k(1)=-\sum_{j=1}^{k-1}{\binom{k}{j}}f^{(k+1-j)}(1)f^{(j+1)}(1)<0. $$
We claim that $g_k$ is decreasing for $z\geq1$, from which the result follows immediately. We have
\begin{align*}
   g'_k(z)=&-\sum_{j=0}^k\left[{\binom{k}{j}}f^{(k+2-j)}(z)f^{(j+1)}(z)\right]-\sum_{j=0}^k\left[{\binom{k}{j}}f^{(k+1-j)}(z)f^{(j+2)}(z)\right]\\
    &+3(k-1)f^{(k+2)}(z)f'(z)+3(k-1)f''(z)f^{(k+1)}(z)\\
    =&-\sum_{j=3}^k
2\left[{\binom{k}{j}}f^{(k+2-j)}(z)f^{(j+1)}(z)\right]+(3k-5)f'(z)f^{(k+2)}(z)\\
&-k(k-1)f^{(k)}(z)f'''(z)+(k-3)f''(z)f^{(k+1)}(z).
\end{align*}

Note that
\begin{equation}\label{eqn.identidade1}
f^{(k)}(z)=\frac{f^{(k+1)}(z)z}{(p-k)}
=\frac{f^{(k+2)}(z)z^2}{(p-k)(p-(k+1))}
\end{equation}
and thus
\begin{align*}
    g_k'(z)=&-\sum_{j=3}^k\left[2{\binom{k}{j}}f^{(k+2-j)}(z)f^{(j+1)}(z)\right]+f^{(k+2)}(z)z\left[(3k-5)+\frac{k-3}{p-(k+1)}\right.\\
 &+\left.\left(\frac{k(k-1)}{(p-k)}(p-1)(p-2)-(p-1)(k-3)-(3k-5)(p-(k+1))\right)\frac{z^{p-2}}{p-(k+1)}\right].
\end{align*}
Since $f^{(j)}(z)<0$, for $2\leq j\leq k+2$, and 
\begin{align*}
  &k(k-1)(p-1)(p-2)-(p-1)(p-k)(k-3)-(3k-5)(p-k)(p-(k+1))\\
  &\geq k(k-1)(p-1)(p-2)-(p-1)(p-2)(k-3)-(3k-5)(p-2)(p-1)\\
  &\geq(p-1)(p-2)(k^2-5k+8)>0,
\end{align*}
we have $g'_k\leq0$ and the conclusion follows.
\end{proof}
 
\begin{lemma}\label{Prop2.11}
Let $n\in\mathbb{N}$ be such that $n\geq3$ and $p\in(n,n+1]$. Then, the Claim \ref{claim} holds.
\end{lemma} 
\begin{proof}
For fixed $z\geq 1$, 
$$\psi(z,x)=\sum_{k=0}^{n-1}\frac{d^k\psi}{dx^k}(z,z)\frac{(x-z)^k}{k!}+\frac{d^n\psi}{dx^n}(z,c)\frac{(x-z)^n}{n!}, $$
for some $c\in(z,x)$. The sum on the left hand side is negative by Lemma \ref{Prop2.10}. It is enough to show that $\frac{\partial ^n\psi}{\partial x^n}(z,x)<0$  for $x\geq z$. 

We have that
\begin{align*}
\frac{\partial ^n\psi}{\partial x^n}(z,x)&=-\sum_{j=0}^{n}\left[{ \binom{n}{j}}f^{(n+1-j)}(x)f^{(j+1)}(x)\right]-3f'(z)\left(f^{(n+1)}(x)+\sum_{j=0}^n{\binom{n}{j}}\frac{d^j(z-x)}{dx^j}f^{(n+2-j)}(x)\right)\nonumber \\
&=-\left[\sum_{j=2}^{n-2}\left[{ \binom{n}{j}}f^{(n+1-j)}(x)f^{(j+1)}(x)\right]+2f'(x)f^{(n+1)}(x)+2nf''(x)f^{(n)}(x)\right] \nonumber \\
&-3f'(z)\left(f^{(n+2)}(x)(z-x)-(n-1)f^{(n+1)}(x)\right). \label{eq:n-derivative_aux}
\end{align*}
Using the identities \eqref{eqn.identidade1},
\begin{align*}
    \frac{\partial ^n\psi}{\partial x^n}(z,x)&=-\sum_{j=2}^{n-2}\left[{ \binom{n}{j}}f^{(n+1-j)}(x)f^{(j+1)}(x)\right]+f^{(n+1)}(x)\left[\frac{-2n}{(p-n)}f''(x)x-2f'(x)\right.\\
    &\left.-3f'(z)(z-x)\frac{p-(n+1)}{x}+3(n-1)f'(z)\right].
\end{align*}
Note now that 
$f^{(n+1)}(x)<0$, $-3f'(z)(z-x)(p-(n+1))/x>0$ and $f'$ is decreasing. Furthermore, $f^{(j)}<0$ for $3\leq j\leq n+1$. Thus
\begin{align*}
    \frac{\partial ^n\psi}{\partial x^n}(z,x)&\leq f^{(n+1)}(x)\left[\frac{-2n}{(p-n)}f''(x)x-2f'(x)+3(n-1)f'(x)\right]\\
    &=\frac{f^{(n+1)}(x)x}{(p-n)}\left[-2n(1-(p-1)x^{p-2})+(3n-5)(p-n)(1-x^{p-2})\right].
\end{align*}
To finish, we show that
$ -2n(1-(p-1)x^{p-2})+(3n-5)(p-n)(1-x^{p-2})\geq 0,$
or, equivalently,
\begin{equation*}\label{eqn:Prop6.9.1}
     \frac{2n}{(p-n)(3n-5)}\geq \frac{1-x^{p-2}}{1-(p-1)x^{p-2}},
\end{equation*}
which holds, since
\begin{equation*}
   \frac{2n}{(p-n)(3n-5)}\geq\frac{1}{p-1}\geq\frac{1-x^{p-2}}{1-(p-1)x^{p-2}}.\qedhere 
\end{equation*}
\end{proof} 

\begin{proof}[Conclusion of the proof of Proposition \ref{prop L(z)_resumo}]
By Lemmas \ref{prop 6.4} and  \ref{lema6.5}, the only thing left to show is the behaviour of $L$ at the endpoints of its domain.

By Lemma \ref{lema6.5} and the Dominated Convergence Theorem,
$$L(z_0)=\frac{1}{\sqrt{2}}\int_0^1 \lim_{z\to z_0}H(z,s)ds,\quad z_0:=\left(\frac{p}{2}\right)^{\frac{1}{p-2}}.$$
Let $s\in(0,1)$ be fixed and recall that $\lim_{z\to z_0}I_z=\lim_{z\to z_0}f^{-1}(-3f(z))-z=0$. Hence,
$\lim_{z\to z_0}I'_z=\lim_{z\to z_0}-1-3\frac{f'(z)}{f'(z+I_z)}=-4.$
Then
\begin{align*}
    \lim_{z\to z_0}H(z,s)^2&=\lim_{z\to z_0}\frac{I_z^2}{f(z+sI_z)+3f(z)}=\lim_{z\to z_0}\frac{2I_zI'_z}{f'(z+sI_z)(1+sI'_z)+3f'(z)}=0,
\end{align*}
and thus $L(z_0)=0$.

As $z\to0^+$ we have, from Fatou's Lemma and Definition \ref{Def L}, that
\[
\liminf_{z\to 0^+}L(z)\geq\int_0^{\left(\frac{p}{2}\right)^\frac{1}{p-2}}\frac{dt}{\sqrt{f(t)}}=\infty.\qedhere
\]
\end{proof}
\subsection{Type  $\mathcal{C}$  Solutions}
\begin{proposition}\label{constantsolution}
For every $\ell>0$, $u_3\equiv 1$ is the only constant solutions of  \eqref{eqn.probu_4}. In particular, there exist exactly two solutions of type $\mathcal{C}$ of \eqref{eqn.boundstateprob} that are constant (equal to one) on the terminal edge $e_3$.\qedhere
\end{proposition}
\begin{proof}
Note that $u_3$ is constant if and only if $u_3\equiv 1$.
Since $\varphi$ is even, there exists a unique $y>0$ such that $\varphi(y)=\varphi(-y)=1$. 
\end{proof}
We now focus on nonconstant solutions of \eqref{eqn.probu_4}.

\begin{lemma}\label{lemma: C>0}
Suppose $u\in H^1(\G_\ell)$ is a type $\mathcal{C}$ solution to \eqref{eqn.boundstateprob}. Then the value of $C$ in \eqref{eqn:uniq.3} is given by $$C=\frac{1}{2}{u_1'(0)}^2=\frac{1}{2}u_1^2(0)-\frac{1}{p} u_1^p(0)=f(u_1(0))>0.$$
\end{lemma} 
\begin{proof}
The proof is analogous to that of Lemma \ref{lem:lemma7.3}, with $u_3'(0)=0$ by the Neumann-Kirchoff conditions.
\end{proof}
Recall from \eqref{defif} the auxiliary function $f(z)=z^2/2-z^p/p$
and denote by $f_1$ and $f_2$ its restriction to the intervals  $[0,1]$ and $(1,(p/2)^{1/(p-2)}]$, respectively. Observe that $f_1$ is increasing and $f_2$ is decreasing.

For each $y\in \R$, let $u_3$ be the solution of the initial value problem
\begin{equation*}\label{eqn.probu_3.2}
 -u_3''+u_3=|u_3|^{p-2}u_3,\ \text{in}\ [0,\infty),\quad  u_3(0)=\varphi(y),\quad
    u_3'(0)=0,
\end{equation*}
and let $L>0$ be the first zero of $u_3'$. Since, $u_3'(0)=0=u_3'(L)$ we have, by Lemma \ref{lemma: C>0}, that
\[
f(u_3(0))=F(u_3(0),0)=C=F(u_3(L),0)=f(u_3(L)).
\]
Since $u_3$ is strictly monotone on  $(0,L)$, $u_3(L)$ and $u_3(0)$ are the two unique and distinct (positive) solutions of the equation $f(z)=C$. Recalling that $\varphi(y)\in \left(0,(p/2)^\frac{1}{p-2}\right]$, we have the following cases to consider:

\smallbreak

\begin{enumerate}
\item if $\varphi(y)\in (0,1)$, then $u_3(0)=\varphi(y)<1<u_3(L)$ and $u_3'>0$ in $(0,L)$, see Figure \ref{fig:Figure7};
\item if $\varphi(y)\in (1,(p/2)^{\frac{1}{p-2}})$, then $u_3(0)=\varphi(y)>1>u_3(L)$ and $u_3'<0$ in $(0,L)$, see Figure \ref{fig:Figure8};
\item if $\varphi(y)=1$, then $u_3\equiv 1$, recovering the solution in Proposition \ref{constantsolution};
\item if $\varphi(y)=(p/2)^{\frac{1}{p-2}}$, then $y=0$, contradicting Proposition \ref{prop:sol_halfline}.
\end{enumerate}

\smallbreak

We focus on the first two cases. As before, we wish to write an expression for $L$, the first zero of $u_3'$, in terms of $u_3(0)=\varphi(y)$. Since $u_3$ has no critical points in $(0,L)$, we have
$$L =\int_0^L \, dx=\int_{u_3(0)}^{u_3(L)}\frac{1}{u'_3(u_3^{-1}(t))}dt. $$

\begin{figure}[ht]
    \centering
\begin{minipage}{0.45\linewidth}
\centering
    \includegraphics[scale=0.07]{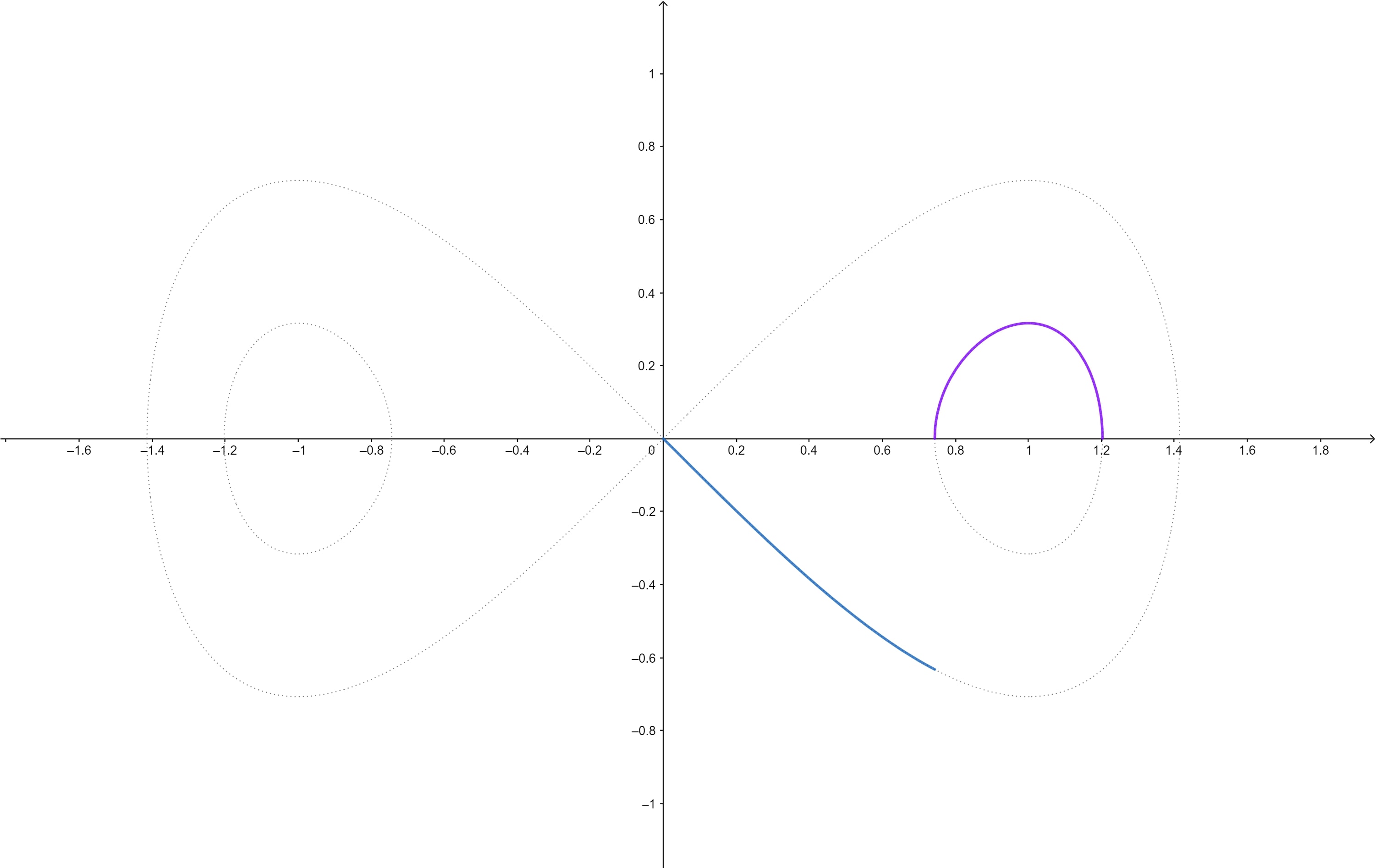}
\end{minipage}
\hspace{5pt}
\begin{minipage}{0.45\linewidth}
\centering
    \includegraphics[scale=0.07]{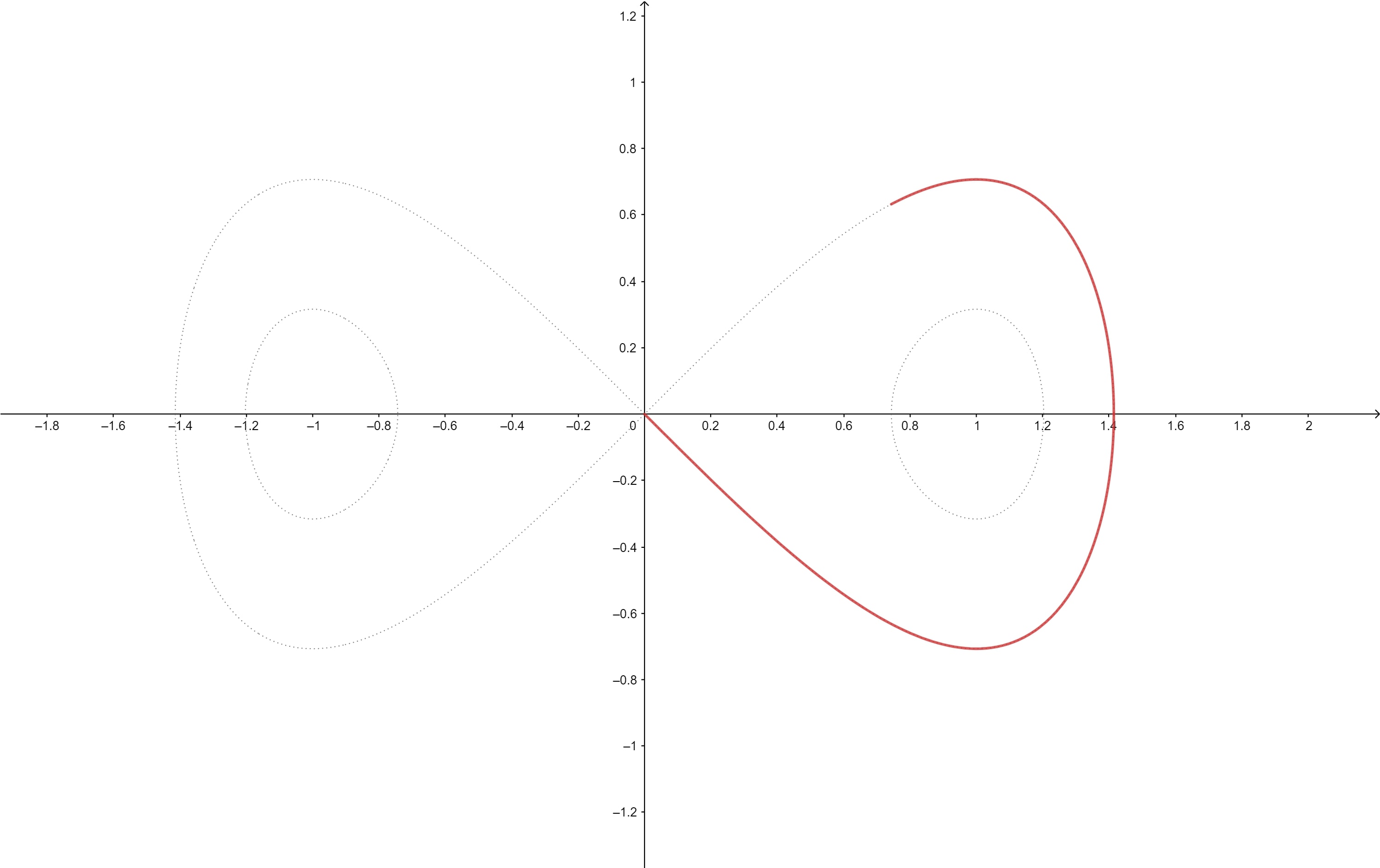}
\end{minipage}
    \caption{The $(u,u')-$phase plane argument for the case $\varphi(y)<1$ (type $\mathcal{C}$ solution). The blue portion represents the edge $e_1$, the purple portion the edge $e_3$ and the red portion the edge $e_2$.}
    \label{fig:Figure7}
\end{figure}
\begin{figure}[ht]
    \centering
\begin{minipage}{0.45\linewidth}
\centering
    \includegraphics[scale=0.07]{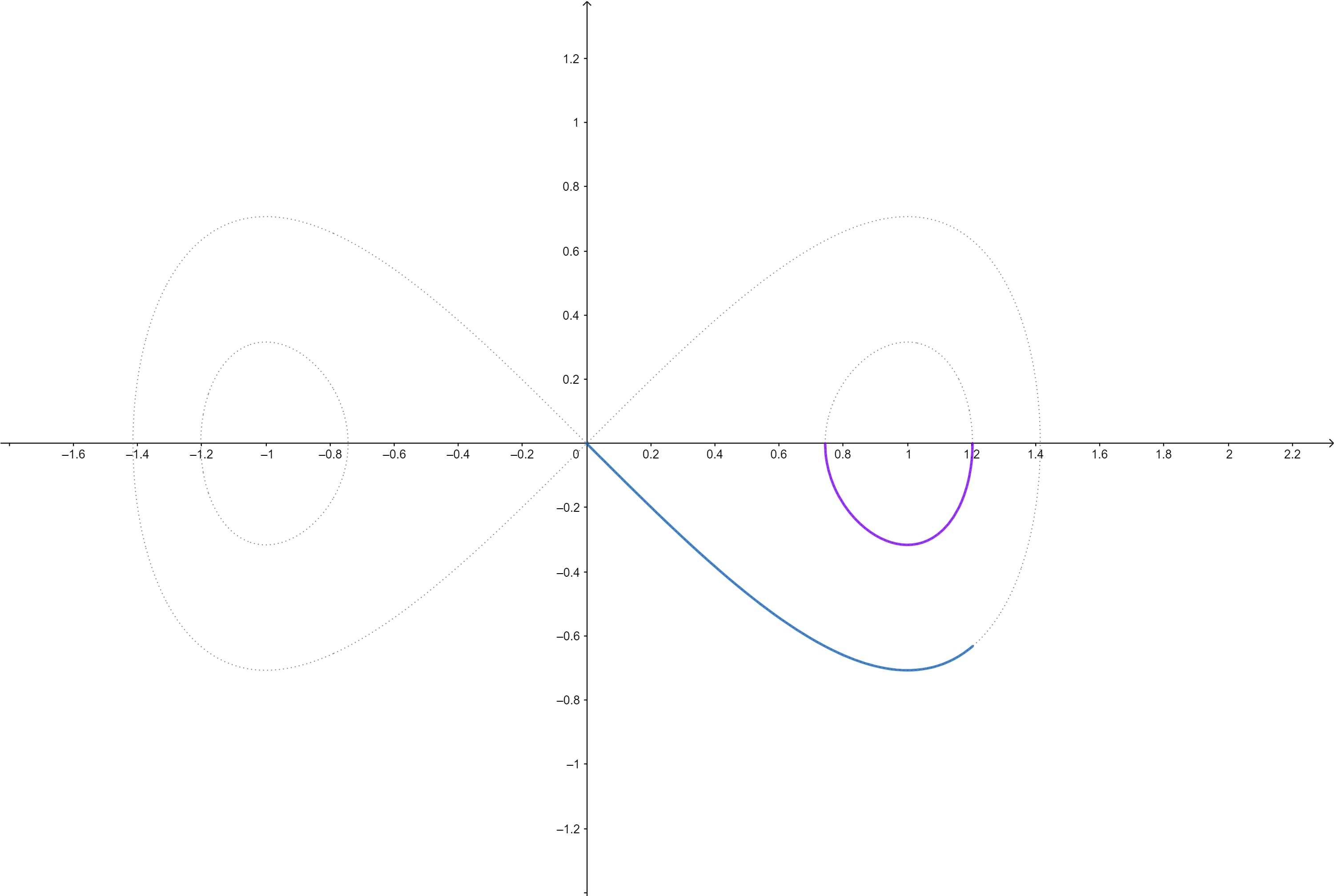}
\end{minipage}
\hspace{5pt}
\begin{minipage}{0.45\linewidth}
\centering
\includegraphics[scale=0.07]{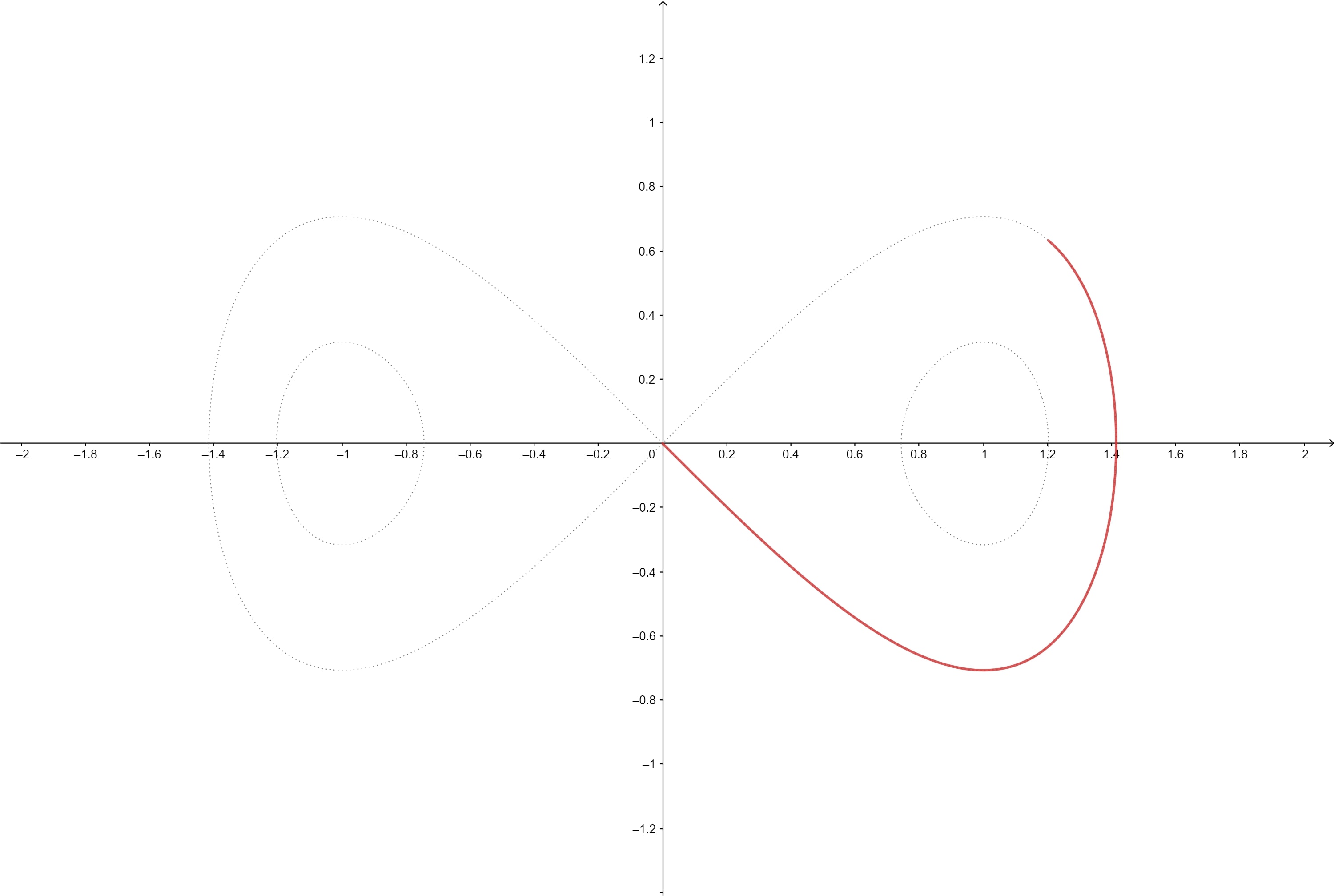}
\end{minipage}
    \caption{The $(u,u')-$phase plane argument for the case $\varphi(y)>1$ (type $\mathcal{C}$ solution). The blue portion represents the edge $e_1$, the purple portion the edge $e_3$ and the red portion the edge $e_2$.}
    \label{fig:Figure8}
\end{figure}

Assume that $u_3(0)=\varphi(y)<1<u_3(L)$, so that $u_3'>0$ and $u_3(L)=f_2^{-1}(C)=f_2^{-1}(f(u_3(0)))$. Then
$$
L =\frac{1}{\sqrt{2}}\int_{u_3(0)}^{f_2^{-1}(f(u_3(0)))}\frac{1}{\sqrt{f(t)-f(u_3(0))}}dt. 
$$
The other case is $u_3(0)=\varphi(y)>1>u_3(L)$, $u_3'<0$ and $u_3(L)=f_1^{-1}(C)=f_1^{-1}(f(u_3(0)))$. By similar arguments, we have
$$L =\frac{1}{\sqrt{2}}\int^{u_3(0)}_{f_1^{-1}(f(u_3(0)))}\frac{1}{\sqrt{f(t)-f(u_3(0))}}dt. $$
\begin{definition}\label{Definition L_1}
   For $z\in(0,(p/2)^{\frac{1}{p-2}})\setminus\{1\}$, we denote the \textit{length} function by
$$ L_1(z)=\frac{1}{\sqrt{2}}\int_{f_1^{-1}(f(z))}^{f_2^{-1}(f(z))}\frac{1}{\sqrt{f(t)-f(z)}}dt.$$ 
\end{definition}

\begin{proposition}[Properties of the length function $L_1(z)$]\label{prop L_1(z)_resumo}
The function $L_1$ is positive, differentiable, strictly decreasing in the interval $(0,1)$ and strictly increasing in $(1,(p/2)^{\frac{1}{p-2}})$. Furthermore, we have $$
\lim\limits_{z\to \left(\frac{p}{2}\right)^{\frac{1}{p-2}}}L_1(z)=\lim\limits_{z\to 0^{+}}L_1(z)=\infty
$$
and there exists $L_1(1):=\lim_{z\to 1^\pm} L_1(z)=\frac{\pi}{\sqrt{p-2}}>0$. 
\end{proposition}

Assuming Proposition \ref{prop L_1(z)_resumo} holds, we finish the proof of Theorem \ref{Theo1:MAIN}.

\begin{proof}[Proof of Theorem \ref{Theo1:MAIN} part 3.]
By Proposition \ref{constantsolution}, for each $\ell>0$ there exist two constant solutions of  \eqref{eqn.probu_4}. 

Define now $\ell^*=L_1(1)>0$. By Proposition \ref{prop L_1(z)_resumo} it follows that, for $\ell\leq\ell^*$, there exist no nonconstant solutions of  \eqref{eqn.probu_4}. If $\ell>\ell^{*}$, by Proposition \ref{prop L_1(z)_resumo}, there exist two distinct $z\in (0,(p/2)^{\frac{1}{p-2}})$, say $z_1$ and $z_2$, with $z_1<z_2$, such that $L_1(z_1)=\ell=L_1(z_2)$. Consequently, there exist $y_1>0$ and $y_2>0$ such that the function $u_3$ solves the problem \eqref{eqn.probu_4} with $y=\pm y_1$ and $y=\pm y_2$, and thus we obtain exactly four solutions of type $\mathcal{C}$ that are strictly monotone on the terminal edge $e_3$.

To construct solutions $u_3$ with at least $n-1$ critical points in the interval $(0,\ell)$, we can exploit the symmetry of the phase plane. Indeed, such solutions  go along the closed orbit passing through the points $(u_3(0),0)$ and $(u_3(\ell),0)$, covering exactly $n$ halves of the  orbit. In this case, the new length function becomes
$L_{n}(z)=nL_1(z)$ and the threshold length becomes $n\ell^*$. Thus, for each $\ell\in(n\ell^*,(n+1)\ell^*]$
we have a total of $4n$ nonconstant type $\mathcal{C}$ solutions.
\end{proof}

In the remaining part of this section, we prove Proposition \ref{prop L_1(z)_resumo}.

\begin{lemma}\label{identidades_L_derL}
    For any $z\in\left(0,\left(\frac{p}{2}\right)^\frac{1}{p-2}\right)\setminus\{1\}$ the following identities hold:
\begin{equation*}\label{id1_lemma2.22}
    \sqrt{2}L_1(z)(f(z)-f(1))=-\int_{f^{-1}_1(f(z))}^{{f^{-1}_2(f(z))}}\sqrt{f(t)-f(z)}\left(3-2\frac{(f(t)-f(1))f''(t)}{{f'}^2(t)}\right)dt
\end{equation*}
and
\begin{equation*}\label{id2_lemma2.22}
    \sqrt{2}L_1'(z)(f(z)-f(1))=-f'(z)\int_{f^{-1}_1(f(z))}^{f^{-1}_2(f(z))}\frac{\sqrt{f(t)-f(z)}}{f'(t)^4}g(t)dt,
\end{equation*}
where $g:\R^+\to\R$ is defined by
\begin{equation*}
    g(t)=-3f''(t)f'(t)^2+6(f(t)-f(1))f''(t)^2-2(f(t)-f(1))f'''(t)f'(t).
\end{equation*}
\end{lemma}
\begin{proof}
Here, we follow an idea from \cite{kairzhan2021standing} to get rid of the fact that, in the definition of $L_1(z)$, the boundary of integration corresponds to the zeros of the denominator in the integrand function.  Fix $z\in (0,(p/2)^\frac{1}{p-2})\setminus \{1\}$. 
Define $\gamma(t)= (x(t),y(t))$, where
\begin{equation*}
   \begin{cases}
        (x(t),y(t)):=(t,\sqrt{2}\sqrt{f(t)-f(z)}), &\ t\in [z,f^{-1}_2(f(z))],\ \text{if}\  z<1,\\
        (x(t),y(t)):=(f^{-1}_1(f(z))+z-t,-\sqrt{2}\sqrt{f(f^{-1}_1(f(z))+z-t)-f(z)}), &\ t\in [f^{-1}_1(f(z)),z],\ \text{if}\ z>1,
    \end{cases}
\end{equation*}
which parametrize the level curve $\Gamma$, $-\frac{1}{2}y^2+f(t)=f(z)$,  for $y\geq 0$ and $y\leq 0$, respectively. Let $W(t,y)=\frac{f(t)-f(1)}{f'(t)}y$, which is of class $C^1$ in $\R\setminus\{1\}\times\R$ and can be extended by continuity to $\R^2$. Then
\begin{align*}
  0&=W(f^{-1}_2(f(z)),0)-W(f^{-1}_1(f(z)),0)=  \int_{\Gamma}dW(t,y)= \int_{\Gamma} \frac{\partial W}{\partial t}(t,y)dt+\frac{\partial W}{\partial y}(t,y)dy \nonumber \\
  &=\int_\Gamma \left(1-\frac{(f(t)-f(1))f''(t)}{{f'}^2(t)}\right)ydt+\frac{f(t)-f(1)}{f'(t)}dy \nonumber \\
  &=\sqrt{2}\int_{f^{-1}_1(f(z))}^{f^{-1}_2(f(z))}\left(1-\frac{(f(t)-f(1))f''(t)}{{f'}^2(t)}\right)\sqrt{f(t)-f(z)}dt+\frac{1}{\sqrt{2}}\int_{f^{-1}_1(f(z))}^{f^{-1}_2(f(z))}\frac{f(t)-f(1)}{\sqrt{f(t)-f(z)}}dt\label{simplification}
\end{align*}
In particular,
\begin{align*}
\sqrt{2}L_1(z)(f(z)-f(1))&=-\int_{f^{-1}_1(f(z))}^{f^{-1}_2(f(z))}\sqrt{f(t)-f(z)}dt+\int_{f^{-1}_1(f(z))}^{f^{-1}_2(f(z))}\frac{f(t)-f(1)}{\sqrt{f(t)-f(z)}}dt\\
&=-\int_{f^{-1}_1(f(z))}^{f^{-1}_2(f(z))}\sqrt{f(t)-f(z)}\left(3-2\frac{(f(t)-f(1))f''(t)}{{f'}^2(t)}\right)dt.
\end{align*}
Differentiating under the integral sign,
\begin{align*}
    &\sqrt{2}L'_1(z)(f(z)-f(1))=-\sqrt{2}L_1(z)f'(z)-\int_{f^{-1}_1(f(z))}^{f^{-1}_2(f(z))}\frac{-f'(z)}{2\sqrt{f(t)-f(z)}}\left(3-\frac{2(f(t)-f(1))f''(t)}{{f'}^2(t)}\right)dt\\
&=f'(z)\int_{f^{-1}_1(f(z))}^{f^{-1}_2(f(z))}\left(\frac{1}{2}-\frac{(f(t)-f(1))f''(t)}{{f'}^2(t)}\right)\frac{dt}{\sqrt{f(t)-f(z)}}.
\end{align*}
Since
\begin{align*}
    \lim_{t\to 1^{\pm}}\frac{1}{f'(t)}\left(1-2\frac{(f(t)-f(1))f''(t)}{{f'}^2(t)}\right)=-\frac{1}{3f''(1)},
\end{align*}
we rewrite
\begin{align*}
&\int_{f^{-1}_1(f(z))}^{f^{-1}_2(f(z))}\left(\frac{1}{2}-\frac{(f(t)-f(1))f''(t)}{{f'}^2(t)}\right)\frac{dt}{\sqrt{f(t)-f(z)}}\\
&=\int_{f^{-1}_1(f(z))}^{f^{-1}_2(f(z))}\frac{d}{dt}(\sqrt{f(t)-f(z)})\left(\frac{1}{f'(t)}-2\frac{(f(t)-f(1))f''(t)}{{f'}^3(t)}\right)dt\\
=&-\int_{f^{-1}_1(f(z))}^{f^{-1}_2(f(z))}\sqrt{f(t)-f(z)}\frac{d}{dt}\left(\frac{1}{f'(t)}-2\frac{(f(t)-f(1))f''(t)}{{f'}^3(t)}\right)dt\\= &- \int_{f^{-1}_1(f(z))}^{f^{-1}_2(f(z))}\frac{\sqrt{f(t)-f(z)}}{f'(t)^4}g(t)dt,
\end{align*}
where we have used integration by parts and direct computations.
\end{proof}
We introduce the following claim:

\begin{claim}\label{claim1}
Let $g:\R^+\to\R$ be defined as in Lemma \ref{identidades_L_derL}.
Then $g(t)<0$ for $t\neq1$.
\end{claim}
\begin{lemma}\label{L_1_monotony}
Under Claim \ref{claim1}, $L_1$  is strictly decreasing for $z\in(0,1)$ and strictly increasing for $z\in(1,(p/2)^\frac{1}{p-2})$.
\end{lemma}
\begin{proof}
    From Lemma \ref{identidades_L_derL} we have that $L_1$ is differentiable and, moreover, since the Claim \ref{claim1} holds, then:
    $$\sgn L'_1(z)=-\sgn f'(z),\ \text{ for all }\ z\in(0,(p/2)^{\frac{1}{p-2}})\setminus\{1\},$$
    which finishes the proof.
\end{proof}

\begin{lemma}\label{g_inicio}
    Let $z^*=\left(\frac{1}{p-1}\right)^\frac{1}{p-2}\in(0,1)$, which is the zero of $f''$. Then, Claim \ref{claim1} holds for $t\in(0,z^*]$.
\end{lemma}
\begin{proof}
    For $t\in (0,z^*]$ we have that $f''(t)\geq0$, $f'(t)> 0$ and  $f(t)-f(1)<0$. Thus, all the terms in the definition of $g$ are negative.
\end{proof}

\begin{lemma}\label{g_fim}
    Let $\displaystyle t^*(p)=\left(\frac{8+p(2p-7)}{(p-1)(3p-4)}\right)^\frac{1}{p-2}\in (0,1)$. Then, Claim \ref{claim1} holds for $t\in (t^*(p),\infty)\setminus \{1\}.$
\end{lemma}
\begin{proof}
Firstly, it is immediate that $g(1)=0$. Differentiation in $t$ yields
\begin{equation*}\label{psi_derg}
    g'(t)=-5f'(t)^2f'''(t)+10(f(t)-f(1))f''(t)f'''(t)-2(f(t)-f(1))f'(t)f^{(iv)}(t).
\end{equation*}
Recalling that
$$f^{(iv)}(t)=f'''(t)\frac{p-3}{t},$$
we find
\begin{equation}\label{eq:derivative_of_g}
g'(t)=f'''(t)\left(-5f'(t)^2+10(f(t)-f(1))f''(t)-2(p-3)(f(t)-f(1))\frac{f'(t)}{t}\right)=:f'''(t)\psi(t).
\end{equation}
For $t>1$, it is enough to show that $\psi>0$,  as it implies that $g'<0$ and then $g<0$. Observe that $\psi(1)=0$ and
\begin{align*}
    \psi'(t)&=10(f(t)-f(1))f'''(t)-\frac{2(p-3)}{t}f'(t)^2-2(p-3)(f(t)-f(1))\left(\frac{f''(t)}{t}-\frac{f'(t)}{t^2}\right)\\
    &=-2t\left(2(p-2)(2p-1)(f(t)-f(1))t^{p-4}+(p-3)(1-t^{p-2})^2\right)=:-2t\xi(t),
\end{align*}
where the explicit expressions of $f$ and its derivatives were used.
% \begin{align*}
%     \psi'(t)&=10(f(t)-f(1))f'''(t)-\frac{2(p-3)}{t}f'(t)^2-2(p-3)(f(t)-f(1))\left(\frac{f''(t)}{t}-\frac{f'(t)}{t^2}\right)\\
%     &=10(f(t)-f(1))f'''(t)-\frac{2(p-3)}{t}f'(t)^2+2(p-3)(p-2)(f(t)-f(1))t^{p-3
%     }.
% \end{align*}
% Using the explicit expression for $f'''$ and $f'$ we obtain that
% \begin{align*}
%     \psi'(t)&=-4(p-2)(2p-1)(f(t)-f(1))t^{p-3}-2(p-3)(1-t^{p-2})^2t\\
%     &=-2t\left(2(p-2)(2p-1)(f(t)-f(1))t^{p-4}+(p-3)(1-t^{p-2})^2\right)=:-2t\xi(t).
% \end{align*}
We show that $\xi(t)<0$ for $t>1$. Note that $\xi(1)=0$ and that
\begin{align*}
    \xi'(t)&=2(p-2)(2p-1)\left((p-4)t^{p-5}(f(t)-f(1))+t^{p-4}f'(t)\right)-2(p-2)(p-3)t^{p-3}(1-t^{p-2})\\
    &=2(p-2)t^{p-5}\left((2p-1)(p-4)(f(t)-f(1))+t^2(p+2)(1-t^{p-2})\right)=:2(p-2)t^{p-5}\eta(t).
\end{align*}
To conclude that $\eta<0$ for $t>1$, observe that $\eta(1)=0$ and that
\begin{equation*}
    \eta'(t)=t(1-t^{p-2})\left(8+p(2p-7)\right)-(p-2)(p+2)t^{p-1}=\left(8+p(2p-7)-(p-1)(3p-4)t^{p-2}\right)t<0,
\end{equation*}
since $t>t^*(p)$.

Suppose now that $t\in(t^*(p),1)$. By the previous computation, it follows that $\eta>0$ in $t\in(t^*(p),1)$ and thus $\xi$ is increasing. Since $\xi(1)=0$, we have $\xi<0$ for $t\in(t^*(p),1)$. Therefore, $\psi$ is increasing and, given that $\psi(1)=0$ it follows that $\psi<0$. Thus, $g$ is increasing for $t\in(t^*(p),1)$ and, because $g(1)=0$, $g<0$ for $t\in(t^*(p),1)$.\qedhere
\end{proof}

\begin{lemma}
We have $t^*(p)\leq z^*$ if and only if $p\in(2,3]$. In particular, Claim \ref{claim1} holds for $p\in (2,3]$.
\end{lemma}
\begin{proof}
    We have $t^*(p)\leq z^*$ if and only if $\displaystyle \frac{8+p(2p-7)}{(3p-4)}\leq1$
or, equivalently,
$(p-2)(p-3)=p^2-5p+6\leq0.$
\end{proof}

\begin{lemma}\label{lemma.zerog}
    Let $p>3$ and suppose that $g(\tau)=0$ for some $\tau\in(z^*,t^*(p))$. Then we have that $g'(\tau)>0.$ In particular, Claim \ref{claim1} holds for $p>3$.
\end{lemma}
\begin{proof}
    Since $f''(\tau)\neq0$ we rewrite the identity $g(\tau)=0$ as
\begin{equation*}\label{eqn.keyidentity}
    f'(\tau)^2=2(f(\tau)-f(1))\left(f''(\tau)-\frac{f'(\tau)}{3f''(\tau)}f'''(\tau)\right).
\end{equation*}
Recalling \eqref{eq:derivative_of_g} and using the fact that $f'''<0$ and $t^*(p)<1$, we have
\begin{align*}
g'(\tau)&=f'''(\tau)\frac{(f(\tau)-f(1))f'(\tau)}{3f''(\tau)\tau}\left(10f'''(\tau)\tau-6(p-3)f''(\tau)\right)\\
&=f'''(\tau)\frac{(f(\tau)-f(1))f'(\tau)}{3f''(\tau)\tau}\left(-6(p-3)-2(p-1)(2p-1)\tau^{p-2}\right)>0.
\end{align*}

By contradiction, suppose that there exists $\tau\in(z^*,t^*(p))$ with $g(\tau)=0$. Since $g(t^*(p))<0$, $g$ would have to change sign from positive to negative in the interval $(\tau,t^*(p))$, a contradiction.
\end{proof}

The following result will prove useful to understand the behaviour of $L_1(z)$ for $z\sim 1$. 

\begin{lemma}\label{L_1rewriten1}
    The Length function $L_1$ can be rewritten as 
\[
   L_1(z) =\begin{cases}
        \frac{1}{\sqrt{2}}\int_0^1F(z,s)-F(f_2^{-1}(f(z)),s)ds\ \text{ for }\ z< 1\\
        \frac{1}{\sqrt{2}}\int_0^1G(z,s)-G(f_1^{-1}(f(z)),s)ds\ \text{ for }\ z>1,
    \end{cases}
\]
where $F(z,s)=\frac{1-z}{\sqrt{f(z+(1-z)s)-f(z)}}$ and $G(z,s)=\frac{z-1}{\sqrt{f(1+(z-1)s)-f(z)}}$.
\end{lemma}
\begin{proof}
    Suppose first that $z<1$. Note that
    $$\sqrt{2}L_1(z)=\int_z^1\frac{dt}{\sqrt{f(t)-f(z)}}+\int_1^{f_2^{-1}(f(z))}\frac{dt}{\sqrt{f(t)-f(z)}}.$$
Taking, in the second integral,
 $t=\frac{f_2^{-1}(f(z))-1}{z-1}\xi+\frac{z-f_2^{-1}(f(z))}{z-1}=a(z)\xi+b(z),$
it follows that
$$\sqrt{2}L_1(z)=\int_z^1\frac{1}{\sqrt{f(t)-f(z)}}-a(z)\frac{1}{\sqrt{f(a(z)t+b(z))-f(z)}}dt$$
Setting $t=z+(1-z)s$, the conclusion follows. For $z>1$, the proof is analogous.
\end{proof}

\begin{proof}[Conclusion of the proof of Proposition \ref{prop L_1(z)_resumo}]

The differentiability of $L_1$ follows from Lemma \ref{identidades_L_derL} and the monotonicity from Lemma \ref{L_1_monotony}. 
By Fatou's Lemma and Definition \ref{Definition L_1}, it follows that
$$\liminf_{z\to 0^+}L_1(z)\geq \int_0^{\left(\frac{p}{2}\right)^\frac{1}{p-2}}\frac{dt}{\sqrt{f(t)}}=\infty.$$
The asymptotic for $z\to (p/2)^{\frac{1}{p-2}}$ follows from an analogous computation.

The only thing left to check is the behaviour of $L_1$ near $z=1$. We start with the case $z<1$. By Lemma \ref{L_1rewriten1},  we focus on limit
$\lim\limits_{z\to 1^{-}}F(z,s)-F(f_2^{-1}(f(z)),s)$. Applying l'Hopital's rule twice,
\begin{equation*}\label{eqn2.23}
\lim_{z\to 1}F(z,s)^2=-2\lim_{z\to 1}\frac{1-z}{f'(z+(1-z)s)(1-s)-f'(z)}=\frac{2}{-f''(1)s(2-s)}=\frac{2}{(p-2)s(2-s)}.
\end{equation*}
Since $f_2^{-1}(f(z))\to 1^+$ as $z\to 1^{-}$, $F(z,s)>0$ for $z<1$ and $F(z,s)<0$ for $z>1$, 
\begin{equation*}\label{eqn.2.20}
    \lim_{z\to 1^-}F(z,s)-F(f_2^{-1}(f(z)),s)=\lim_{z\to 1^-}F(z,s)-\lim_{z\to 1^+}F(z,s)=\frac{2\sqrt{2}}{\sqrt{(p-2)s(2-s)}}.
\end{equation*}
By the Dominated Convergence Theorem,
$$\lim_{z\to1^-}L_1(z)=\int_0^1\frac{2}{\sqrt{(p-2)s(2-s)}}ds=-\frac{2}{\sqrt{p-2}}\arcsin\left.\left(1-s\right)\right|_{s=0}^{s=1}=\frac{\pi}{\sqrt{p-2}}.$$

 For the case $z>1$, we apply the same argument to $G$ given by Lemma \ref{L_1rewriten1}, obtaining  
$$
\lim_{z\to 1^+} G(z,s)=-\lim_{z\to 1^-} G(z,s) =\frac{\sqrt{2}}{\sqrt{(p-2)(1-s^2)}}
$$
and
$$\lim_{z\to1^+}L_1(z)=\int_0^1\frac{2}{\sqrt{(p-2)(1-s^2)}}ds=\frac{2}{\sqrt{p-2}}\arcsin{s}|_{s=0}^{s=1}=\frac{\pi}{\sqrt{p-2}}.$$
\end{proof}

\section{Ground State Solutions}\label{sec:groundStates}
In this section, we are concerned with existence and uniqueness of action ground states and in proving that, in general, action ground states are not energy ground states. We prove Theorem \ref{thm:existence_uniqueness_action} in Section \ref{sec:groundStates1} and Theorem \ref{Theo2.Main} in Section \ref{sec:groundStates2}. Another topic of this section is the uniqueness of energy ground states in the $L^2$ subcritical regime. In Section \ref{sec:groundStates3} we prove Theorem \ref{TheoMain5.NonUniquenessEGS} and we present numerical simulations supporting Conjecture \ref{EGSUniquenessConjecture}.

\subsection{Existence and uniqueness of Action ground states}\label{sec:groundStates1}

By Lemma \ref{lema.scaling}, we may assume throughout this section that $\ell=1$. Recall from the introduction that an action ground state is a solution $u\in H^1(\mathcal{G}_1)$ of \eqref{eqn.boundstateprob_intro} such that $S_\lambda(u,\mathcal{G}_1)=\mathcal{S}_{\G_1}(\lambda)$ (see \eqref{ActionProbI}).  

In order to show the uniqueness of action ground states, for each $\mu>0$ and a general metric graph $\mathcal{G}$, we introduce the following auxiliary 
quantity (for $p>2$):
\begin{equation}\label{eq:agscharact}
    \mathcal{I}_{\G,\lambda}(\mu):=\inf\left\{ I_\lambda (u,\G):=\frac{1}{2}\int_{\G}|u'|^2dx+\frac{\lambda}{2}\int_{\G}|u|^2dx,\  u\in H^1(\G),\ \norm{u}{p}{\G}^p=p\mu\right\}.
\end{equation}
For the $\mathcal{T}$-graph $\mathcal{G}_1$, in this section we show:
\begin{itemize}
\item that the level $\mathcal{I}_{\G_1,\lambda}(\mu)$ is achieved, and that minimizers are of type $\mathcal{A}$;
\item that there is a direct relation between the levels $\mathcal{I}_{\G_1,\lambda}(\mu)$ and $\mathcal{S}_{\G_1}(\lambda)$.
\end{itemize}
Then, the previous points combined with Theorem \ref{Theo1:MAIN}-1 show that action ground states are unique for each $\lambda>0$, thus proving Theorem \ref{thm:existence_uniqueness_action}.

\smallbreak

We start by collecting useful properties of $\mathcal{I}_{\mathcal{G}_1,\lambda}(\mu)$.

\begin{lemma}\label{lema3.1}
The following properties hold true.
\begin{enumerate} 
    \item Let $\mu_1,\mu_2>0$. Then, $\mathcal{I}_{\G_1,\lambda}(\mu_1)=\left(\mu_1/\mu_2\right)^{2/p}\mathcal{I}_{\G_1,\lambda}(\mu_2)$. Moreover,  $u_1$ minimizes $\mathcal{I}_{\G_1,\lambda}(\mu_1)$ if and only if $u_2=\left(\mu_2/\mu_1\right)^{1/p}u_1$ minimizes $\mathcal{I}_{\G_1,\lambda}(\mu_2)$. %and, in particular, $\mathcal{I}_{\G_1,\lambda}(\mu)=\mu^{\frac{2}{p}}\mathcal{I}_{\G_1,\lambda}(1)$. 
    \item Suppose $u\in H^1(\G_1)$ is a minimizer of $\mathcal{I}_{\G_1,\lambda}(\mu)$. Then $u$ satisfies 
\begin{equation*}\label{NLSMULTLAGRA}
   -u''+ \lambda u=\theta|u|^{p-2}u \text{ in } \mathcal{G}_1, \quad \text{ with } \quad \theta=\frac{2}{p}\mu^\frac{2-p}{p}\mathcal{I}_{\G_1,\lambda}(1).
\end{equation*}
    \end{enumerate}
\end{lemma}

\begin{proof}
    The proof of 1. is straightforward. As for 2., by the Lagrange's multipliers Lemma, there exists $\theta>0$ such that $-u''+ \lambda u=\theta |u|^{p-2}u$. Testing the equation with $u$ and integrating over $\G_1$, it follows that
$2\mathcal{I}_{\G_1,\lambda}(\mu)=p\mu\theta$. The conclusion now follows from point 1.
\end{proof}

Regarding the existence and characterization of minimizers associated with $\mathcal{I}_{\G_1,\lambda}(\mu)$, one can easily adapt the concentration compactness arguments by Adami, Serra and Tilli in  \cite[Theorem 2.6]{adami2015nls} (see also \cite[Theorem 3.3]{adami2016threshold}) to our variational setting. In the cited papers, the authors deal with \emph{energy} ground states for $p\in (2,6)$, but the same theory  gives a sufficient condition for existence of minimizers of \eqref{eq:agscharact} for general non-compact metric graphs and $p>2$.

\begin{theorem}[{{\cite[Theorems 2.2 \& 2.6]{adami2015nls}}}\label{thm:generalcriterium} - Uniform Bounds and Sufficient Condition for Existence of Minimizers] \label{th4.6}
Take $p>2$, $\mu>0$ and let $\G$ be a non-compact metric graph. Then
\begin{equation}\label{suffcond}
\mathcal{I}_{\R^+,\lambda}(\mu)\leq \mathcal{I}_{\G,\lambda}(\mu)\leq \mathcal{I}_{\R,\lambda}(\mu).%=T(\varphi,\R).
\end{equation}
Moreover, if there exists $u\in H^1(\G)$ such that $I_\lambda(u,\G)\leq\mathcal{I}_{\R}(\mu)$, then the infimum $\mathcal{I}_{\G,\lambda}(\mu)$ is achieved.
\end{theorem}

This sufficient condition together with the hybrid rearrangement introduced in \cite[Lemma 6.1]{adami2015nls} gives the following result, whose proof we sketch here:

\begin{theorem}[Existence and qualitative properties of minimizers of problem $\mathcal{I}_\G(\mu,\lambda)$]\label{propqualmin}
Let $p>2$. Given $\mu>0$, the infimum $\mathcal{I}_{\G_1,\lambda}(\mu)$ is achieved.
Moreover, each minimizer $u$ is strictly monotone on the terminal edge $e_3$ with a maximum at the tip,  it is strictly decreasing on each half-line, and symmetric with respect to the vertex. 
\end{theorem}
\begin{proof}[Proof: (Sketch)]
Let $\psi\in H^1(\R)$ be the unique positive and even minimizer of $\mathcal{I}_{\R,\lambda}(\mu)$ (which is a multiple of the soliton $\varphi^\lambda$, solution to \eqref{eq:soliton}). Consider the function $w\in H^1(\G_1)$ given by
$$w(x)=\begin{cases}
    \psi(x+\frac{\ell}{2}),\ x\in e_1,\ e_2,\\
    \psi(x-\frac{\ell}{2}),\ x\in e_3.
\end{cases}$$
Then we can consider $\Tilde{w}$, the hybrid rearrangement of $w$ (see \cite[Lemma 6.1]{adami2015nls}), which satisfies
$I_\lambda(\Tilde{w},\G_\ell)\leq I_\lambda(\psi,\R)=\mathcal{I}_{\R,\lambda}(\mu)$.  By Theorem \ref{thm:generalcriterium}, the infimum  $\mathcal{I}_{\G_1,\lambda}(\mu)$ is achieved.

From Lemma \ref{lema3.1}-2 and Proposition \ref{prop:sol_halfline},  any minimizer $u\in H^1(\G_1)$ of $\mathcal{I}_{\G_1,\lambda}(\mu)$ satisfies, for all $t>0$, $|\left\{u=t\right\}|=0$. From the hybrid rearrangement and  reasoning exactly as in \cite[Theorem 2.7]{adami2015nls}, we conclude the monotonicity and symmetry properties of $u$.\qedhere
\end{proof}

\begin{proposition}\label{minimizer equivalence}
    A function $u\in H^1(\G_1)$ minimizes  $\mathcal{I}_{\G_1,\lambda}(\mu)$ if and only if $$v=\left(\frac{\mu_0}{\mu}\right)^{\frac{1}{p}}u\quad\text{minimizes}\quad  S_{\G_1}(\lambda),\quad\text{for } \mu_0=\left(\frac{2}{p}\mathcal{I}_{\G_1,\lambda}(1)\right)^{\frac{p}{p-2}}.$$ Furthermore, $\mathcal{I}_{\G_1,\lambda}(\mu)=\frac{p-2}{p}\left(\frac{\mu_0}{\mu}\right)^{\frac{2}{p}}S_{\G_1}(\lambda)$.
\end{proposition}

\begin{proof}
    Let $u$ be a minimizer of problem $\mathcal{I}_{\G_1,\lambda}(\mu_0)$. We show that $S_\lambda(u,\G_1)=S_{\G_1}(\lambda)$. 
    Since $-u''+\lambda u=|u|^{p-2}u$, we have $S_{\G_1}(\lambda)\leq S_\lambda(u,\G_1)$. Suppose now that $v\in H^1(\G_1)$ solves the equation $-v''+\lambda v=|v|^{p-2}v$ in $\G_1$. Let $\mu>0$ be such that $\norm{v}{p}{\G_1}^p=p\mu$. Then
    $$
    \left(\frac{\mu}{\mu_0}\right)^{2/p} \mathcal{I}_{\mathcal{G}_1,\lambda}(\mu_0)=\mathcal{I}_{\G_1,\lambda}(\mu)\leq I_\lambda(v,\G_1).
    $$
    Moreover, since
    \[
    I_\lambda(v,\mathcal{G}_1)=\frac{p\mu}{2} \quad \text{ and } \quad \mathcal{I}_{\mathcal{G}_1,\lambda}(\mu_0)=
    \frac{p\mu_0}{2},
    \]
    we have $\mu_0\leq \mu$.    From here and Lemma \ref{lema3.1}, it follows that
    \begin{align}
       S_{\mathcal{G}_1}(\lambda)\leq  S_\lambda(u,\G_1)&=I_\lambda(u,\G_1)-\frac{1}{p}\norm{u}{p}{\G_1}^p
        =\left(1-\frac{2}{p}\right)I_\lambda(u,\G_1)
        =\left(1-\frac{2}{p}\right)\mathcal{I}_{\G_1,\lambda}(\mu_0)\nonumber\\
        &=\left(1-\frac{2}{p}\right)\left(\frac{\mu_0}{\mu}\right)^{\frac{2}{p}}\mathcal{I}_{\G_1,\lambda}(\mu)\leq\left(1-\frac{2}{p}\right)\left(\frac{\mu_0}{\mu}\right)^{\frac{2}{p}}I_\lambda(v,\G_1)\nonumber\\&=S_\lambda(v,\G_1)\left(\frac{\mu_0}{\mu}\right)^{\frac{2}{p}}\leq S_\lambda(v,\G_1).\label{eqn.inequalities}
    \end{align}
Taking the infimum in $v$, we conclude that $S_\lambda(u,\mathcal{G}_1)=S_{\mathcal{G}_1}(\lambda)$. Moreover, if $v\in H^1(\mathcal{G}_1)$ realizes $ S_{\mathcal{G}_1}(\lambda)$ and $\mu:=2I_\lambda(v,\mathcal{G}_1)$, then \eqref{eqn.inequalities}  becomes a chain of equalities, $\mu=\mu_0$ and $I_\lambda(v,\mathcal{G}_1)=\mathcal{I}_{\mathcal{G}_1,\lambda}(\mu_0)$. \qedhere
\end{proof}

\begin{proof}[Conclusion of the proof of Theorem \ref{thm:existence_uniqueness_action}]
Let $\mu_0=\left(\frac{2}{p}\mathcal{I}_{\G_1,\lambda}(1)\right)^{\frac{p}{p-2}}$ and let $u$ be a minimizer of $\mathcal{I}_{\G_1,\lambda}(\mu_0)$. Then, by Theorem \ref{propqualmin} and Proposition \ref{minimizer equivalence} , $u$ is an action ground state and a type $\mathcal{A}$ solution of $-v''+ \lambda v=|v|^{p-2}v$. The result now follows from part 1. of Theorem \ref{Theo1:MAIN}. \qedhere
\end{proof}

\subsection{Action vs. energy ground states}\label{sec:groundStates2}

In this section, we show Theorem \ref{Theo2.Main}, namely that the action ground states on the $\mathcal{T}$ graph are not necessarily energy ground states for $p\in (6-\varepsilon,6)$ for sufficiently small $\varepsilon>0$.
Let $u^\lambda$ be the unique type $\mathcal{A}$ solution of the equation
\begin{equation}\label{eqnforlambda}
    -u''+\lambda u=|u|^{p-2}u,\ \text{in}\  \G_1. 
\end{equation}

We focus on understanding the behaviour of the map 
\[
(p,\lambda)\in (2,\infty)\times (0,\infty)\mapsto \Theta(p,\lambda):=\norm{u^\lambda}{2}{\G_1}^2. 
\]
 We prove that $\Theta$ is continuous in each variable (Theorem \ref{ContinuityThetaTheo}), and then perform a detailed study of the critical case $p=6$ (see Proposition \ref{ThetaProperties} below). From this, the proof of Theorem \ref{Theo2.Main} follows. 

First, we focus on giving an alternative characterization of $\Theta(p,\lambda)$. Recall that $u_\ell=(u_{\ell,1},u_{\ell,2},u_{\ell,3})$ (cf. Lemma \ref{lema.scaling}), where $u_{\ell,i}$ represents the restriction of $u_\ell$ on the edge $e_i$. From Theorem \ref{Theo1:MAIN}, part 1., there exists a unique $y(p,\ell)>0$ such that $u_{\ell,1}(x)=\varphi(x+y(p,\ell))$; recall that $y(p,\ell)$ is uniquely characterized by the identity $\ell=L(\varphi(y(p,\ell)))$ (see Section \ref{eq:typeAB}, in particular Definition \ref{Def L} and Proposition \ref{prop L(z)_resumo}). Moreover, $u_{\ell,3}$ is the unique positive solution to the problem \eqref{eqn.probu_3}, that is,
\begin{equation*}
    -u_{\ell,3}''+u_{\ell,3}=u_{\ell,3}^{p-1}\ \text{in}\ [0,\ell],\quad
    u_{\ell,3}(0)=\varphi(y(p,\ell)),\quad
    u_{\ell,3}'(0)=-2\varphi'(y(p,\ell)),\quad
    u_{\ell,3}'(\ell)=0.    
\end{equation*}
\begin{lemma}\label{Theta}
    Let $p>2$ and $\lambda>0$. Then
    \begin{align}
\Theta(p,\lambda)&=\lambda^{\frac{6-p}{2(p-2)}}\left[2\int_0^\infty \varphi\left(x+y(p,\ell)\right)^2dx+\int_{0}^\ell u_{\ell,3}^2(x)dx\right]\label{caracterizacao1} \\
&=2\lambda^{\frac{6-p}{2(p-2)}}\int_0^\infty\varphi(x+y(p,\ell))^2dx+\lambda^{\frac{6-p}{2(p-2)}}\frac{1}{\sqrt{2}}\int_{\varphi(y(p,\ell))}^{f^{-1}(-3f(\varphi(y(p,\ell))))}\frac{t^2dt}{\sqrt{f(t)+3f(\varphi(y(p,\ell)))}}\label{caracterizacao2},
\end{align}
where $\ell=\ell(\lambda)=\lambda^{\frac{1}{2}}$, $f(z)=z^2/2-z^p/p$ and $f^{-1}$ denotes the inverse of $f$ on $[1,\infty)$.
\end{lemma}
\begin{proof} We have
\begin{align*}
   \norm{u^\lambda}{2}{\G_1}^2&=\lambda^{\frac{6-p}{2(p-2)}} \norm{u_\ell}{2}{\G_\ell}^2=\lambda^{\frac{6-p}{2(p-2)}}\left[2\int_0^\infty u_{\ell,1}^2(x)dx+\int_{0}^\ell u_{\ell,3}^2(x)dx\right]\\
   &=\lambda^{\frac{6-p}{2(p-2)}}\left[2\int_0^\infty \varphi\left(x+y(p,\ell)\right)^2dx+\int_{0}^\ell u_{\ell,3}^2(x)dx\right].
\end{align*}
By the arguments in Section \ref{sec:characterization}, taking $x=u_{\ell,3}^{-1}(t)$,
\begin{equation*}
    \int_0^\ell u_{\ell,3}^2(x)dx=\frac{1}{\sqrt{2}}\int_{u_{\ell,3}(0)}^{f^{-1}(-3f(u_{\ell,3}(0)))}\frac{t^2dt}{\sqrt{f(t)+3f(u_{\ell,3}(0))}}=\frac{1}{\sqrt{2}}\int_{\varphi(y(p,\ell))}^{f^{-1}(-3f(\varphi(y(p,\ell))))}\frac{t^2dt}{\sqrt{f(t)+3f(\varphi(y(p,\ell)))}}. 
\end{equation*}
\end{proof}

Figure \ref{fig:figure5} depicts numerical plots of the function $\lambda\mapsto\Theta(p,\lambda)$ for subcritical, critical and supercritical values of $p$, respectively.

\begin{figure}[ht]
    \centering
\begin{minipage}{0.3\linewidth}
\centering
    \includegraphics[scale=0.31]{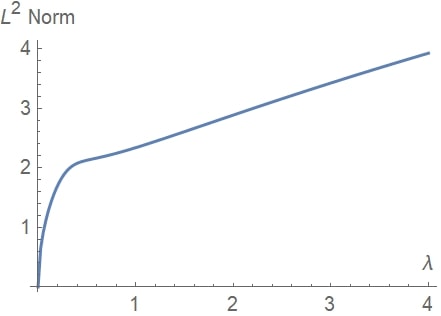}
\end{minipage}
\hspace{7pt}
\begin{minipage}{0.3\linewidth}
\centering
    \includegraphics[scale=0.28]{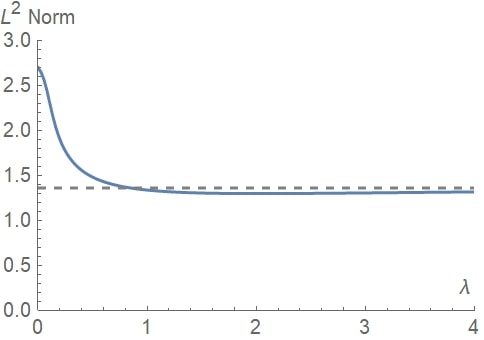}
\end{minipage}
\hspace{7pt}
\begin{minipage}{0.3\linewidth}
\centering
    \includegraphics[scale=0.29]{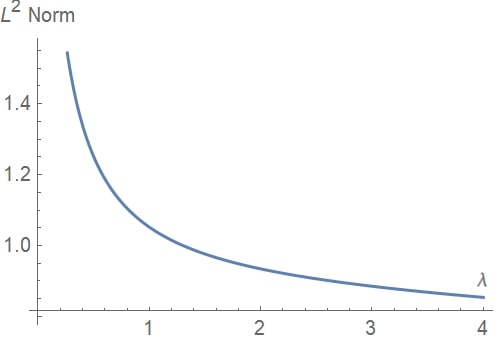}
\end{minipage}
    \caption{From left to right we have the plots of the functions $\Theta(4,\lambda),\ \Theta(6,\lambda),\ \text{and}\  \Theta(8,\lambda)$.}
    \label{fig:figure5}
\end{figure}

We now focus on the continuity of $\Theta(p,\lambda)$.

\begin{lemma}\label{ContinuityTheta_Lambda}
For each $p>2$, the function $\lambda\mapsto\Theta(p,\lambda)$ is continuous.
\end{lemma}

\begin{proof}
    Fix $p>2$. By the Dominated Convergence Theorem, for $z\in (0,(p/2)^{\frac{1}{p-2}})$, the functions
    $$z\mapsto\int_z^{f^{-1}(-3f(z))}\frac{t^2}{\sqrt{f(t)+3f(z)}}dt\quad \text{ and }\quad z\mapsto \int_0^\infty\varphi^2(x+z)dx$$
    are continuous. By the construction of $L$, $y(p,\ell)=\varphi^{-1}\circ L^{-1}(\ell)$. Since both $\varphi$ and $L$ are continuous and strictly decreasing, we have that  $\varphi^{-1}\circ L^{-1}$ is continuous and the conclusion follows from \eqref{caracterizacao2}.
\end{proof}
To show the continuity of $p\mapsto\Theta(p,\lambda)$, we start with a few auxiliary lemmas.
\begin{lemma}\label{nontrivsol}
    Let $p>2$. Then $u_{\ell,3}(\ell)\geq 1$.
\end{lemma}
\begin{proof}
Testing the equation with $u_\ell$,
    \begin{equation*}
\|u_\ell\|_{H^1(\mathcal{G}_\ell)}^2=\|u_\ell \|_{L^p(\mathcal{G}_\ell)}^p\leq\norm{u_\ell}{\infty}{\G_\ell}^{p-2}\norm{u_\ell}{2}{\G_\ell}^2=u_{\ell,3}(\ell)^{p-2}\norm{u_\ell}{2}{\G_\ell}^2\leq u_{\ell,3}(\ell)^{p-2}\|u_\ell\|^2_{H^1(\G_\ell)},
    \end{equation*}
which yields $1\leq u_{\ell,3}(\ell)$.    
\end{proof}

The following result will be used for the continuity in $p$ of the function $\Theta(p,\lambda)$ and in the last section.

\begin{lemma}\label{prop.NormBouund} 
For $p>2$, let $u_\ell$ be the unique positive solution of type $\mathcal{A}$ to \eqref{eqnforL}. Then
\begin{equation}\label{lowerupperboundu_l}
p^\frac{2}{2-p} \left(\mathcal{I}_{\R^+,1}(1)\right)^\frac{p}{p-2}\leq \|u_\ell\|_{H^1(\G_\ell)}^2=\|u_\ell\|_{L^p(\G_\ell)}^p\leq p^\frac{2}{2-p} \left(\mathcal{I}_{\R,1}(1)\right)^\frac{p}{p-2}.
\end{equation}
\end{lemma}

\begin{proof}
     Testing equation \eqref{eqnforL} with $u_\ell$ and combining the resulting identity with Proposition \ref{minimizer equivalence} and the second inequality in \eqref{suffcond}  yields 
     \begin{align*}
       \|u_\ell\|_{L^{p}(\mathcal{G}_\ell)}^p&=  \|u_\ell\|_{H^1(\G_\ell)}^2= \mathcal{I}_{\G_\ell,1}\left(\frac{\|u_\ell\|_{L^{p}(\mathcal{G}_\ell)}^p}{p}\right)\leq \mathcal{I}_{\R,1}\left(\frac{\|u_\ell\|_{L^{p}(\mathcal{G}_\ell)}^p}{p}\right)=\left(\frac{\|u_\ell\|_{L^{p}(\mathcal{G}_\ell)}^p}{p}\right)^\frac{2}{p}\mathcal{I}_{\R,1}(1).
     \end{align*}
This proves the upper bound in \eqref{lowerupperboundu_l}. The proof of the lower bound is analogous, using this time the first inequality in \eqref{suffcond}.
\end{proof}

\begin{lemma}\label{ContinuityTheta_p}
For each $\lambda>0$, the function $p\mapsto\Theta(p,\lambda)$ is continuous for $p>2$.
\end{lemma}
\begin{proof}
    Fix $\lambda>0$. Take $(p_n)_{n\in\N}\subset (2,\infty)$ such that $p_n\to p>2$ and, furthermore, for each $n\in\N$, let $u_\ell^n$ be the unique positive solution of type $\mathcal{A}$ to \eqref{eqnforL} with $p=p_n$. From Lemma \ref{prop.NormBouund}, it follows that $u_\ell^n$ is uniformly bounded in $H^1(\G_\ell)$ and thus there exists $u_\ell\in H^1(\G_\ell)$ such that $u_\ell^n\rightharpoonup u_\ell$ in $H^1(\G_\ell)$ and strongly in $L^\infty_{loc}(\G_\ell)$. By Lemma \ref{nontrivsol}, $u_\ell \not\equiv 0$.

    We check that $u_\ell$ solves the equation $-u''_\ell+u_\ell=u_\ell^{p-1}$ in $H^1(\G_\ell)$. From the weak convergence and the local uniform convergence, it follows immediately that, for every $v\in H^1(\G_\ell)$ with compact support on the half-lines,
    $$ \int_{\G_\ell}{u'}_\ell^nv'+u_\ell^nvdx\to\int_{\G_\ell}u'_\ell v+ u_\ell vdx\quad \text{ and }\quad \int_{\G_\ell}{(u^n_\ell)}^{p_n-1}vdx\to \int_{\G_\ell}u_\ell^{p-1}vdx.$$
Therefore,
\[
\int_{\G_\ell}u'_\ell v+ u_\ell vdx=\int_{\G_\ell}u_\ell^{p-1}vdx
\]
for every $v\in H^1(\G_\ell)$ with compact support on the half-lines. The fact that $u_\ell$ solves the equation $-u''_\ell+u_\ell=u_\ell^{p-1}$ in $\G_\ell$ follows by density arguments.

   Notice that $u_\ell(0)>0$. Indeed, if $u_\ell(0)=0$, then, by standard existence and uniqueness theory for ODE's, $u_{\ell,1}=u_{\ell,2}\equiv 0$. Then, by Neumann-Kirchoff conditions, at the vertex we would have $u'_{\ell,3}(0)=0$ and thus $u_{\ell,3}\equiv0$, a contradiction with Lemma \ref{nontrivsol}.
    
Therefore, on the one hand, there exists a unique $y(p,\ell)>0$ such that
    $u_\ell(0)=\varphi(y(p,\ell))$
    and, on the other hand, for each $n\in\N$ there exists a unique $y(p_n,\ell)$ such that    $u_\ell^n(0)=\varphi(y(p_n,\ell)). $
    Given that $u_{\ell,3}^n\to u_{\ell,3}$ uniformly in $[0,\ell]$, we have $u_\ell^n(0)\to u_\ell(0)$ and $y(p_n,\ell)\to y(p,\ell)$.
    
    The result now follows from the characterization \eqref{caracterizacao1}, together with the local uniform convergence of $u_{\ell,3}^n$, the convergence of $y(p_n,\ell)$ and the continuity of $\varphi$ in $p$.
\end{proof}

Lemmas \ref{ContinuityTheta_Lambda} and \ref{ContinuityTheta_p} imply the following. 
\begin{theorem}\label{ContinuityThetaTheo}
    The function $(p,\lambda)\mapsto\Theta(p,\lambda)$ is continuous in each variable.
\end{theorem}

We present a general sufficient criterion for the existence of action ground states having the same mass.

\begin{proposition}\label{Prop_Theta_N_Monotona}
Fix $p>2$ and suppose that the map $\lambda\mapsto\Theta(p,\lambda)$ is not monotone. Then, there exists a range $[\nu_1,\nu_2]$ with the following property: given $\nu\in[\nu_1,\nu_2]$, there exist at least two positive and distinct $\lambda_1,\lambda_2$ such that, if $u^{\lambda_1},u^{\lambda_2}$ are the action ground state solutions of \eqref{starionaryNLS}, then
\begin{equation}\label{eqn.normaL2igual}
\nu=\|u^{\lambda_1}\|_{L^2}^2= \| u^{\lambda_2}\|_{L^2}^2.
\end{equation}

In particular, if $p<6$, for $\nu\in[\nu_1,\nu_2]$ (up to a countable set), there exist action ground states that are not energy ground states.
\end{proposition}
\begin{proof}
    Since $\lambda\mapsto\Theta(p,\lambda)$ is not monotone, without loss of generality, there exist $\lambda_1<\lambda^*<\lambda_2$ such that $\Theta(p,\lambda^*)<\Theta(p,\lambda_1)<\Theta(p,\lambda_2)$. Therefore, since $\Theta(p,\lambda)$ is continuous in $\lambda$, for any $\nu\in\left(\Theta(p,\lambda^*),\Theta(p,\lambda_1)\right)$ there exist at least two distinct  $\lambda_1'$ and $\lambda_2'$, for which $\nu=\Theta(p,\lambda_1')=\Theta(p,\lambda_2')$ and \eqref{eqn.normaL2igual} follows.

    Suppose now that $p$ is subcritical. From \cite[Theorem 2.8]{dovetta2020uniqueness}, up to a countable set, given $\nu\in[\Theta(p,\lambda^*),\Theta(p,\lambda_1)]$, 
    there exists only one $\lambda$, satisfying $\nu=\Theta(p,\lambda)$, such that $u^\lambda$ is an energy ground state. This implies that the remaining solutions to $\nu=\Theta(p,\lambda)$ corresponds to action ground states which are not energy ground states.
\end{proof}
Using the continuity of $\Theta(p,\lambda)$ in $p$, we now verify that the conditions of Proposition \ref{Prop_Theta_N_Monotona} hold for $p\sim6$ (see also Figure \ref{fig:figure5}).
\begin{proposition}\label{ThetaProperties}
    For $p=6$, $\Theta(6,\lambda)\to \nu_\R$ as $\lambda\to 0^+$ and $\Theta(6,\lambda)\nearrow \nu_{\R^+}$ as $\lambda\to\infty$, where $\nu_\R$ and $\nu_{\R^+}$ are the $L^2$ norm of $\varphi$ in $\R$ and in $\R^+$, respectively. In particular, $\Theta(6,\lambda)$ is not monotone in $\lambda$.
\end{proposition}
\begin{proof}
    Note that, as $\lambda\to 0^+$, we have $\varphi(y(6,\ell(\lambda)))=L^{-1}(\ell(\lambda))=L^{-1}(\lambda^{\frac{1}{2}})\to 3^\frac{1}{4}$, or, equivalently, $y(6,\ell(\lambda))\to0$. Then, by the Dominated Convergence Theorem, since $t^2/\sqrt{f(t)}$ is integrable,
    it follows that
    $$ \lim_{\lambda\to 0^+}\Theta(6,\lambda)=2\int_0^\infty\varphi^2(x)dx=\int_{\R}\varphi^2(x)dx=\nu_{\R}.$$
   Similarly, one can see that $y(6,\ell(\lambda))\to\infty$ as $\lambda\to\infty$, and thus, by the Dominated Convergence Theorem, 
   \begin{equation*}
       \lim_{\lambda\to \infty}\Theta(6,\lambda)=\frac{1}{\sqrt{2}}\int_0^{3^{\frac{1}{4}}}\frac{t^2dt}{\sqrt{f(t)}}
       =\sqrt{3}\int_0^1\frac{s}{\sqrt{1-s^4}}ds=\frac{\sqrt{3}\pi}{4}=\sqrt{3}\int_0^\infty\sech(2x)dx=\int_0^\infty\varphi^2(x)dx=\nu_{\R^+}.
   \end{equation*}
We prove monotonicity as a corollary of \cite[Section A.2]{pierotti2021local}.
The authors consider the equation $-u''+\frac{\Lambda^2}{3}u=u^5$ on the graph $\G_1$ and define $y=y(\Lambda)>0$ as the solution to 
\begin{equation*}\label{eqn.problem}
    u_{1,3}(0)=\varphi_{\Lambda}(y),\quad u'_{1,3}(0)=-2|\varphi'_{\Lambda}(y)|,\quad u'_{1,3}(1)=0,
\end{equation*}
where $\varphi_\Lambda$ is the family of real line solitons  $\varphi_\Lambda(x)=\sqrt{\Lambda}\varphi(\Lambda x)$.  Then, the authors are able to construct a  family of type $\mathcal{A}$ solutions $z\mapsto \tilde u_z$, with the parameter $z\in (0,\infty)$ satisfying the relation $z=\Lambda y$.
By scaling, considering $\lambda=\frac{\Lambda^2}{3}$ in  equation \eqref{eqnforlambda} (with $p=6$), we have that 
$$u^\Lambda(x)=\left(\frac{\Lambda}{\sqrt{3}}\right)^\frac{1}{2}u_{\ell(\Lambda)}\left(\frac{\Lambda}{\sqrt{3}}x\right),$$
where $u_{\ell(\Lambda)}$ solves, for $p=6$, the equation \eqref{eqnforL}, with $\ell=\frac{\Lambda}{\sqrt{3}}$. By Theorem \ref{Theo1:MAIN}, part 1., given $\ell>0$ there exists a unique $y=y(\ell)$ such that $\varphi(y)=L^{-1}(\ell)$. Given that $\ell=\frac{\Lambda}{\sqrt{3}}$ and that $\lambda=\frac{\Lambda^2}{3}$ it follows that
$$z=\Lambda y=\Lambda y(\Lambda)=\Lambda\varphi^{-1}\left(L^{-1}(\ell(\Lambda))\right)=\sqrt{3\lambda}\varphi^{-1}\left(L^{-1}\left(\sqrt{\lambda}\right)\right)=:g(\lambda).$$
Moreover, since $\varphi^{-1}, L^{-1}$ are positive strictly decreasing functions, it follows that $g$ is strictly increasing and, as $\lambda\to0$,  we have $z=g(\lambda)\to 0$ and, as $\lambda\to\infty$, it follows that $z=g(\lambda)\to\infty$. Therefore,
for each $z>0$ we have $\lambda=g^{-1}(z)$. Moreover, if $\nu(z)=\norm{\Tilde{u}_z}{2}{\G_1}^2$, we have: 
\begin{equation}\label{Theta composition}
    \nu(z)=\Theta(6,g^{-1}(z))\ \text{ or, equivalently }\ \Theta(6,\lambda)=\nu(g(\lambda)).
\end{equation}
To conclude, from \cite[Section A.2]{pierotti2021local} it follows that $\nu(z)$ is increasing for $z$ sufficiently large. Since $g(\lambda)$ is strictly increasing, by \eqref{Theta composition}, we have that $\Theta(6,\lambda)$ is increasing, if $\lambda$ is sufficiently large.\qedhere

\end{proof}

\begin{proof}[Proof of Theorem \ref{Theo2.Main}]
By Theorem \ref{ContinuityThetaTheo} and Proposition \ref{ThetaProperties} there exists $\varepsilon$ sufficiently small such that for each $p\in(6-\varepsilon,6+\varepsilon)$ the map $\lambda\mapsto\Theta(p,\lambda)$ is not monotone.  The conclusion follows from Proposition \ref{Prop_Theta_N_Monotona}.
\end{proof}

\subsection{Non uniqueness of energy ground states}\label{sec:groundStates3}

We have shown that action ground states are unique and that, in the $L^2$ subcritical regime $p<6$, some of these are not energy ground states. Moreover, in \cite[Theorem 2.8]{dovetta2020uniqueness}, it is shown that, up to a countable set of masses $\nu>0$, there exists a unique energy ground state. It is natural to wonder if uniqueness holds for \emph{every} mass.  We show that, for $p\sim 6^-$, energy ground states are not unique in the $\mathcal{T}-$graph for at least one mass.

With the next lemma  we understand the behaviour, as $\lambda\to0^+$ and $\lambda\to +\infty$, of the function $\Theta(p,\lambda)$ in the subcritical regime.

\begin{lemma}\label{lower_upper_bound}
    For any $p\in(2,6)$. there exists $C>0$ such that 
    \begin{equation}\label{eqn.lema3.14}
        \frac{1}{C}\lambda^{\frac{6-p}{2(p-2)}}\leq \Theta(p,\lambda)\leq C \lambda^{\frac{6-p}{2(p-2)}}.
    \end{equation}
     In particular, for any $p\in (2,6)$, we have that $\Theta(p,\lambda)\to0^+$ as $\lambda\to 0^+$ and $\Theta(p,\lambda)\to+\infty$ as $\lambda\to+\infty$.
\end{lemma}
\begin{proof}
   On the one hand,  from \eqref{lowerupperboundu_l} and the scaling in Lemma \ref{lema.scaling}, there exists $\kappa=\kappa_p>0$ such that
   \begin{equation}\label{eqn.lema3.14.1}
\Theta(p,\lambda)\lambda^{\frac{p-6}{2(p-2)}}=\|u_\ell\|^2_{L^2(\G_\ell)}\leq\|u_\ell\|^2_{H^1(\G_\ell)}\leq \kappa.
   \end{equation}
   On the other hand, from Lemma \ref{lema.scaling}, the Gagliardo-Nirenberg inequality
   \[
   \|u_\ell\|_{L^p(\mathcal{G}_l)}^p\leq C \|u_\ell\|^\frac{p+2}{2}_{L^2(\mathcal{G}_l)}\|u_\ell'\|^\frac{p-2}{2}_{L^2(\mathcal{G}_l)}
   \]
   where $C=C_p$ (see \cite[Proposition 2.1]{adami2016threshold}), and the uniform bounds from \eqref{lowerupperboundu_l}, there exists $D=D_p>0$ such that
%    \begin{equation}
% \|u_\ell\|^2_{H^1(\G_\ell)}=\|u_\ell\|^p_{L^p(\G_\ell)}\leq C\|u_\ell\|^{\frac{p}{2}+1}_{L^2(\G_\ell)}\|u'_\ell\|^{\frac{p}{2}-1}_{L^2(\G_\ell)}\leq C\|u_\ell\|^{\frac{p}{2}+1}_{L^2(\G_\ell)}\|u_\ell\|^{\frac{p}{2}-1}_{H^1(\G_\ell)}.
%    \end{equation}
% From here we obtain, due to Lemma \ref{}, that there exists $D>0$ such that
\begin{equation}\label{eqn.lema.3.14.2}
    D\leq\frac{1}{C^{\frac{4}{p+2}}}\|u_\ell\|^{\frac{2(6-p)}{p+2}}_{H^1(\G_\ell)}\leq \|u_\ell\|^2_{L^2(\G_\ell)}=\lambda^{\frac{p-6}{2(p-2)}}\Theta(p,\lambda).
\end{equation}
The estimates in \eqref{eqn.lema3.14} now follow from equations \eqref{eqn.lema3.14.1} and \eqref{eqn.lema.3.14.2}. The asymptotic behavior as $\lambda\to 0^+$ and $\lambda\to +\infty$ follows from \eqref{eqn.lema3.14} and the fact that $p\in(2,6)$.
\end{proof}
We can now prove the last main result of this section.
\begin{proof}[Proof of Theorem \ref{TheoMain5.NonUniquenessEGS}]

Fix $\lambda_0,\lambda_1>0$, with $\lambda_0<\lambda_1$ such that $\Theta(6,\lambda_0)>\Theta(6,\lambda_1)$ (recall Proposition \ref{ThetaProperties}). By Lemmas \ref{ContinuityTheta_p} and \ref{lower_upper_bound}, there exists $\varepsilon_0>0$ such that 
  \begin{equation}\label{eq:chain_inequalities_aux}
  \Theta(6-\varepsilon,\lambda_0)>\Theta(6-\varepsilon,\lambda_1)>0=\Theta(6-\varepsilon,0^+),\ \text{for all}\ \varepsilon\leq\varepsilon_0.
  \end{equation}
 Fix any $\varepsilon\leq\varepsilon_0$. From Lemma \ref{ContinuityTheta_Lambda} and \eqref{eq:chain_inequalities_aux}, the maximum, $\Theta^*$, of the map $\lambda\mapsto\Theta(6-\varepsilon,\lambda)$ on the interval $(0,\lambda_1]$ is attained at an interior point  $\lambda^*\in(0,\lambda_1)$. Recalling that, for each $\lambda$, $u^\lambda$ is the unique action ground state, we define:
\begin{equation}\label{Teta_barra}
      \overline{\Theta}:=\sup\{\Theta(6-\varepsilon,\lambda):\ \lambda\in(0,\lambda^*]\ \text{and}\  u^\lambda\ \text{is an energy ground state}\}.
  \end{equation}
Since energy ground states exist for every mass $\nu>0$ and $\Theta(6-\varepsilon,\lambda)$ is continuous (Lemma \eqref{ContinuityTheta_Lambda}) and bounded away from zero if $\lambda$ is also bounded away from zero (Lemma \eqref{lower_upper_bound}), it follows that the set defined in \eqref{Teta_barra} is non-empty, and, moreover, $\overline{\Theta}>0$.

We now observe that there exists $\overline{\lambda}\leq\lambda^*$ such that $u^{\overline \lambda}$ is an energy ground state with  $\overline{\Theta}=\Theta(6-\varepsilon,\overline{\lambda})$. Indeed, one can take a sequence $(\lambda_n)_{n\in\N}$ such that, for each $n\in\N$, $u^{\lambda_n}$ is an energy ground state and $\Theta(6-\varepsilon,\lambda_n)\to\overline{\Theta}$. Up to a subsequence,  $\lambda_n\to\overline{\lambda}\in (0,\lambda^*]$. Therefore, by \cite[Proposition 3.3]{dovetta2020uniqueness}, $u^{\overline{\lambda}}$ is an energy ground state.

We split the rest of the proof in two cases. Suppose first that $\overline{\Theta}=\Theta(6-\varepsilon,\overline{\lambda})<\Theta(6-\varepsilon,\lambda^*)=\Theta^*$. Recall that energy ground states exist for all masses. In particular, for any $\mu\in(\overline{\Theta},\Theta^*)$ there exists $\lambda_\mu>0$ such that $u^{\lambda_\mu}$ is an energy ground state and, moreover, by definition of $\overline{\Theta}$, we have that $\lambda_\mu>\lambda^*$ for any $\mu\in(\overline{\Theta},\Theta^*)$. Take now $\mu_n\to\overline{\Theta}^+$ and a sequence $(\lambda_{\mu_n})_{n\in\N}$ such that $\mu_n=\Theta(6-\varepsilon,\lambda_{\mu_n})\to\overline{\Theta}^+$. Up to a subsequence, we have that $\lambda_{\mu_n}\to\tilde{\lambda}$ with $\tilde{\lambda}\geq \lambda^*>\overline{\lambda}$. By \cite[Proposition 3.3]{dovetta2020uniqueness}, we conclude that $u^{\tilde{\lambda}}$ is another energy ground state with mass $\overline{\Theta}$.

Finally, assume that $\overline{\Theta}=\Theta^*$. For any $\mu>\Theta^*$ there exists $\lambda_\mu>0$ such that $u^{\lambda_\mu}$ is an energy ground state and, moreover, by definition of $\overline{\Theta}$ and the construction of $\Theta^*$, we have that $\lambda_\mu>\lambda_1$ for any $\mu>\Theta^*$. Take $\mu_n\to\overline{\Theta}^+$ and a sequence $(\lambda_{\mu_n})_{n\in\N}$ such that $\mu_n=\Theta(6-\varepsilon,\lambda_{\mu_n})\to\overline{\Theta}^+$. Up to a subsequence we have that $\lambda_{\mu_n}\to\tilde{\lambda}$ with $\tilde{\lambda}\geq \lambda_1>\overline{\lambda}$. As before, \cite[Proposition 3.3]{dovetta2020uniqueness} implies that $u^{\tilde{\lambda}}$ is another energy ground state with mass $\overline{\Theta}$.

\end{proof}

It is natural to wonder if this lack of uniqueness holds only for $p$ close to 6. In the next result, we start by providing sufficient conditions to obtain uniqueness of energy ground states.

\begin{proposition}\label{SuffCondENERGYUNIQUENESS}
     Fix $p\in(2,6)$ and let $\lambda\mapsto u^\lambda$ be the action ground state curve in the graph $\G_1$.  
     \begin{enumerate}
         \item Suppose that $\lambda\mapsto\Theta(p,\lambda)$ is injective. Then, energy and action ground states coincide. In particular, energy ground states are unique for all masses;
         \item Suppose that $\lambda\mapsto E(u^\lambda,\G_1)$ is injective. Then energy ground states are unique.
     \end{enumerate}
\end{proposition}
\begin{proof}
Let $v\in H^1(\G_1)$ be an energy ground state with mass $\nu$ and energy $E(v,\G_1)$. Then, $v$ is a type $\mathcal{A}$ solution of $-u''+\lambda_0 u=u^{p-1}$, for some $\lambda_0$ and thus, $v=u^{\lambda_0}$ from Theorem \ref{Theo1:MAIN}, part 1. If $\lambda\mapsto\Theta(p,\lambda)$ is injective then $\lambda_0$ is uniquely determined by $\nu=\Theta(p,\lambda_0)$ and thus $v$ is the unique energy ground state. Moreover, given that there are no more action ground states with mass $\nu$, it follows that action and energy ground states coincide, thus proving item 1. For item 2., since $\lambda\mapsto E(u^\lambda,\G_1)$ is injective, $\lambda_0$ is uniquely determined by $E(v,\G_1)=E(u^{\lambda_0},\G_1)$ and the uniqueness follows.
\end{proof}

We now provide some numeric results that validate the assumptions made in the previous proposition. Figures \ref{fig:figure5} and \ref{fig:10} suggest that, for $p\sim2$ and $p\sim\infty$, the function $\lambda\mapsto\Theta(p,\lambda)$ is injective. Therefore, we give strong indications for the validity of Conjecture \ref{EGSUniquenessConjecture} by item $1.$ of Proposition \ref{SuffCondENERGYUNIQUENESS}.

\begin{figure}[h!]
    \centering
    \begin{minipage}{0.45\linewidth}
         \centering\includegraphics[scale=0.3]{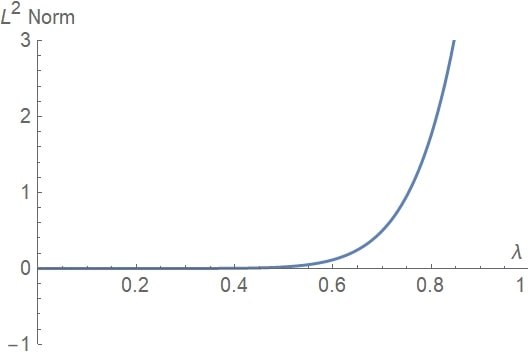}
    \end{minipage}
    \hspace{10pt}
    \begin{minipage}{0.45\linewidth}
         \centering\includegraphics[scale=0.25]{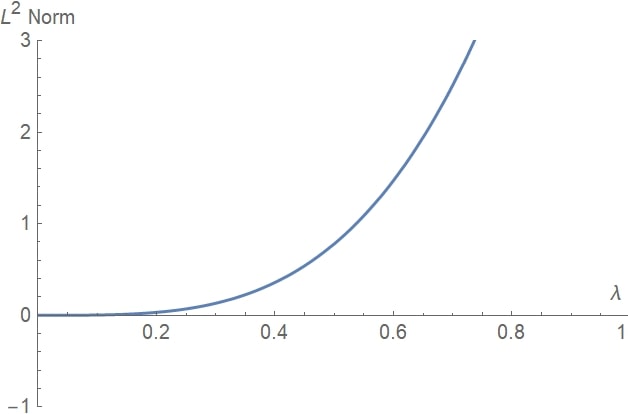}
    \end{minipage}
    \caption{The plot of the functions $\Theta(2.2,\lambda)$ and $\Theta(2.5,\lambda)$ as defined in Lemma \ref{Theta}.}
    \label{fig:10}
\end{figure}

Moreover, in Figure \ref{fig:figure5}, on the left, the map $\lambda\mapsto\Theta(p,\lambda)$ is not monotone for some subcritical values of $p$. However, the energy appears to be (strictly) decreasing up to $p=4$, and failing to be monotone at $p>4$. See Figure \ref{fig:figure6}.

\begin{figure}[h!]
    \centering
\begin{minipage}{0.3\linewidth}
\centering
    \includegraphics[scale=0.25]{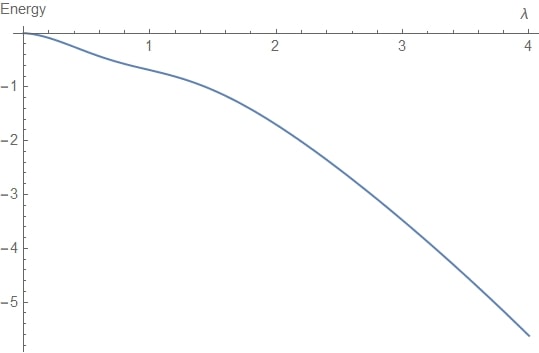}
\end{minipage}
\hspace{10pt}
\begin{minipage}{0.3\linewidth}
\centering
    \includegraphics[scale=0.25]{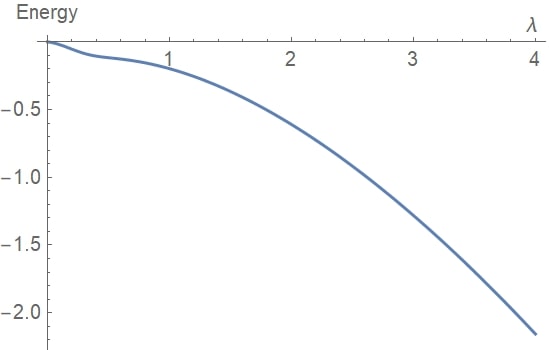}
\end{minipage}
\hspace{10pt}
\begin{minipage}{0.3\linewidth}
\centering
    \includegraphics[scale=0.25]{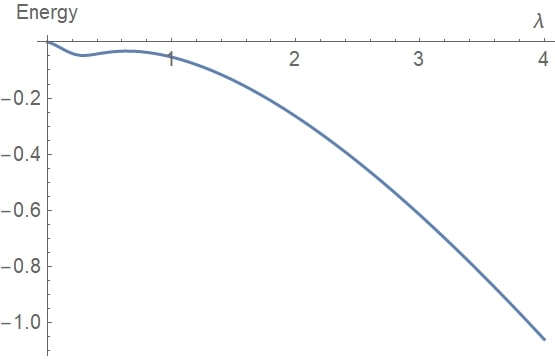}
\end{minipage}
    \caption{From left to right, the plot of the map $\lambda\mapsto E_\lambda(u^\lambda,\G_1)$ for $p=3.8;\ 4;\ 4.4$.}
    \label{fig:figure6}
\end{figure}

\begin{figure}
    \centering
\begin{minipage}{0.3\linewidth}
\centering
    \includegraphics[scale=0.22]{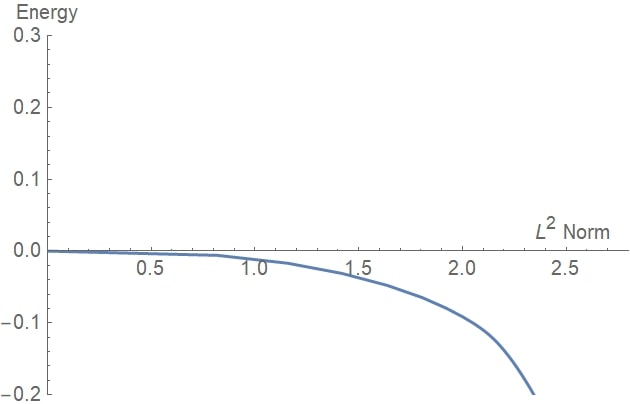}
\end{minipage}
\hspace{5pt}
\begin{minipage}{0.3\linewidth}
\centering
    \includegraphics[scale=0.25]{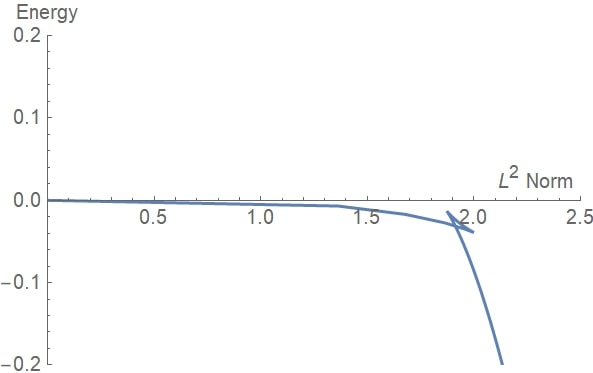}
\end{minipage}
\hspace{5pt}
\begin{minipage}{0.3\linewidth}
\centering
    \includegraphics[scale=0.25]{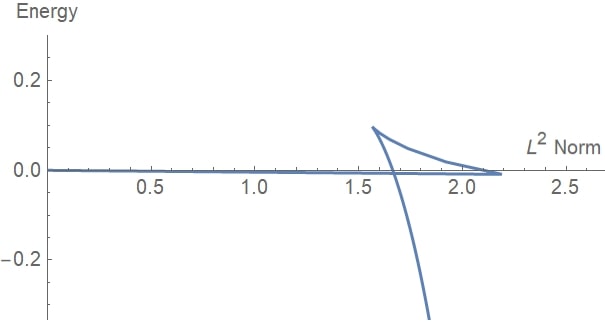}
\end{minipage}
    \caption{From left to right, the trace of the curve $\lambda\mapsto (\|u^\lambda\|^2_{L^2(\G_1)},E(u^\lambda,\G_1))$ for $p=4;\ 4.5;\ 5.2$.}
    \label{fig11}
\end{figure}

This suggests, according to Proposition \ref{SuffCondENERGYUNIQUENESS}, that uniqueness of energy ground states is expected until $p=4$ for all values of $\nu>0$.

For the remaining values of $p$, according to the results in Figures \ref{fig:figure5} and \ref{fig:figure6}, the assumptions of Proposition \ref{SuffCondENERGYUNIQUENESS} do not hold. However, note that, if energy ground states are not unique, then the curve $\lambda\mapsto (\Theta(p,\lambda),E(u^\lambda,\G_1))$ has at least one self intersection. With this in mind, the main purpose of Figure \ref{fig11} is to provide numerical evidence that indeed uniqueness of energy ground states is not expected for $p>4$. In fact, the value of the mass for which uniqueness of energy ground states fails (the value $\overline{\Theta}$ in the proof of Theorem \ref{TheoMain5.NonUniquenessEGS}) corresponds to the mass of the self intersection of the trace of the curve $\lambda\mapsto (\|u^\lambda\|^2_{L^2(\G_1)},E(u^\lambda,\G_1))$.

These simulations were made by inverting numerically the Length function $L(z)$, in order to obtain $y(p,\lambda)$. For instance, for the plot of the $L^2$-norm as a function of $\lambda$,  we used expression \eqref{Theta}. We point out that one can also adapt the software \texttt{GraFiDi} \cite{Grafidi, LeCoz1,LeCoz2}, which can be applied also to more general graphs.

\begin{remark}
The cusps in Figure \ref{fig11} correspond to critical values of the curve $\lambda\mapsto (\|u^\lambda\|^2_{L^2(\G_1)},E(u^\lambda,\G_1))$ (a similar phenomenon was observed for the cubic-quintic NLS in \cite{KillipOh} on $\R^3$). Assuming smoothness of the curve $\lambda\mapsto u^\lambda$ and using $\frac{d}{d\lambda} u^\lambda$ as a test function in 
\[
E'(u^\lambda,\G_1)+\lambda u^\lambda=0,
\]
we obtain
\[
\frac{d}{d\lambda} E(u^\lambda,\G_1) =-\frac{\lambda}{2} \frac{d}{d\lambda}\|u^\lambda\|^2_{L^2(\G_1)}.
\]
Therefore, in general, the energy has a critical point if, and only if, the mass has a critical point. This leads to the formation of cusps.
\end{remark}

\section{Orbital Instability of Action Ground States}\label{sec:instability}

Throughout this section, by scaling, we set $\lambda=1$ and discuss the dynamical stability of action ground states in terms of $\ell$. Let $u_\ell$ be the unique action ground state of
\begin{equation*}\label{NLSsec4}
    -u''+u=|u|^{p-2}u\ \text{in}\ \G_\ell.
\end{equation*}
For simplicity, from now on we drop the subscript $\ell$ from $u_\ell$. Set $\nu:=\norm{u}{2}{\G_\ell}^2$ and consider the manifold
$$\mathcal{Q}:=\left\{v\in H^1(\G_\ell): \norm{v}{2}{\G_\ell}^2=\nu\right\}, $$
whose tangent space at $u\in H^1(\G_\ell)$ we denote by $T_u\mathcal{Q}$.

In order to prove orbital instability (Theorem \ref{Theo3.MAin}), we follow closely Theorem 2.1 in  \cite[Section 2]{gonccalves1991instability}. The proof therein relies on two geometrical properties of the flow of \eqref{NLSE_intro}:
\begin{itemize}
    \item The geometric condition for instability:
    \begin{equation}\label{GCI}\tag{GCI}
   \text{There exists}\ \Psi\in H^1(\G_\ell)\ \text{such that}\  \Psi\in T_{u}\mathcal{Q}\ \text{and}\ \left<S''(u,\G_\ell)\Psi,\Psi\right><0.
\end{equation}
\item The geometric condition for the existence of an auxiliary dynamical system:
\begin{equation}\label{GCADS}\tag{GCADS}
   i\Psi \text{ is }L^2\text{ -orthogonal to the generators of the invariances defining the orbit }\mathcal{O}(u),
\end{equation}
\end{itemize}
where the orbit is given in Definition \ref{OrbitalStabilityDef}. In our framework, the only invariance that is present is the gauge invariance and thus \eqref{GCADS} reduces to $\langle i\Psi,iu\rangle_{L^2}=0$ or, equivalently, $\Psi\in T_u\mathcal{Q}$. Therefore, \eqref{GCI} implies \eqref{GCADS}.

Under \eqref{GCI}, Theorem \ref{Theo3.MAin} is shown exactly as in the case of \cite[Theorem 2.1]{gonccalves1991instability}. In what follows, we omit these details and focus on proving  the validity of \eqref{GCI} for  $\ell\sim 0,\infty$.

Following \cite{gonccalves1991instability}, we consider a differentiable, $L^2$ normalized, one-parameter family $t\mapsto\tilde\eta_t\in \mathcal{Q}$ that satisfies $\tilde \eta_1=u$, and compute
\begin{equation*}\label{ActiontoEnergyderivative}
    \left.\frac{d^2}{dt^2}E(\tilde\eta_t,\G_\ell)\right|_{t=1}= \left.\langle E''(\tilde\eta_t,\G_\ell)\frac{d}{dt}\tilde\eta_t,\frac{d}{dt}\tilde\eta_t\rangle\right|_{t=1}=\left.\langle S''(\tilde\eta_t,\G_\ell)\frac{d}{dt}\tilde\eta_t,\frac{d}{dt}\tilde\eta_t\rangle\right|_{t=1}.
\end{equation*}
Defining $\Psi$ as
$$\Psi:=\left.\frac{d}{dt}\tilde\eta_t\right|_{t=1}\in H^1(\mathcal{G}_\ell),$$
the geometric condition for instability will follow from
\begin{equation}\label{eqn.dernegativa}
     \left.\frac{d^2}{dt^2}E(\tilde\eta_t,\G_\ell)\right|_{t=1}<0.
\end{equation}

We construct $\tilde \eta_t$ in Sections \ref{sec:eta1} and \ref{sec:eta2} for  $\ell\sim \infty$ and $\ell\sim 0$ respectively. Before that, we prove some auxiliary results. Throughout this section, given $a,b>0$, we sometimes use the notation $a\lesssim b$, which means that there exists a universal constant $C>0$ such that $a\le Cb$. We also recall that $f(z)=z^2/2-z^p/p$ and that $f^{-1}$ denotes the inverse of $f|_{[1,\infty)}$.

\begin{lemma}\label{lemma4.12}
Let $u\in H^1(\G_\ell)$ be an action ground state and $y(\ell)>0$ be the (unique) translation parameter given by Proposition \ref{prop:sol_halfline}. Then, as $\ell\to\infty$, 
\begin{enumerate}
    \item$\lim\limits_{\ell\to\infty}y(\ell)=\infty$;
    \item $\lim\limits_{\ell\to\infty}\ell^s\varphi(y(\ell))=\lim_{\ell\to\infty}\ell^su (0)=0$, for any $s>0$.
    
\end{enumerate}
\end{lemma}
\begin{proof}
    By way of contradiction, suppose that $y(\ell)\to a<+\infty$ and thus $\lim_{\ell\to\infty}u_{1}(0)>0$. Moreover, since $u$ is a solution of type $\mathcal{A}$, for each $\ell>0$, $u_{3}(x)\geq u_{1}(0)$, for $x\in[0,\ell]$. Take now $\delta>0$ such that
$ u_{3}(x)\geq \delta$, for $\ell>0$. Then, using the upper bound in \eqref{lowerupperboundu_l}, 
$$ p^\frac{2}{2-p}\left(\mathcal{I}_{\R,1}(1)\right)^\frac{p}{p-2}\geq \int_0^\ell u_{3}^p(x)dx\geq\delta^p\ell.$$
Taking the limit as $\ell\to+\infty$ we arrive at a contradiction, thus proving item $1$.

For item 2., recall that $\ell=L(\varphi(y(\ell)))$ and that there exists $z^*>0$ such that $f^{-1}(-3f(z))\leq z^*$ for every $z\in(0,1)$.
We start by noticing that $L(z)\lesssim|\ln(z)|$ when $z\sim 0$.
Indeed,
$$ L(z)=\frac{1}{\sqrt{2}}\int_z^{f^{-1}(-3f(z))}\frac{dt}{\sqrt{f(t)+3f(z)}}\leq C\int_z^{z^*}\frac{dt}{\sqrt{f(t)}}$$
and, furthermore,
$$\frac{d}{dz}\int_z^{z^*}\frac{dt}{\sqrt{f(t)}}=-\frac{1}{\sqrt{f(z)}}\sim-\frac{1}{z},\ \text{for}\ z\sim0^+.$$
Thus, since $L((p/2)^{1/(p-2)})=0$,
$L(z)\lesssim |\ln(z)|$ for $z\sim 0$.
Since, by item 1, $\varphi(y(\ell))\to 0^+$ as $\ell\to \infty$,
\begin{align*}
0\leq\limsup_{\ell\to\infty}\ell^s\varphi(y(\ell))=\limsup_{z\to 0^+}L(z)^sz\lesssim\lim_{z\to 0^+}|\ln(z)|^sz=0
\end{align*}
and the conclusion follows.
\end{proof}

\begin{lemma}[Uniform estimates]\label{lem6.5}
Let $u\in H^1(\G_\ell)$ be the action ground state. Then there exist constants $C=C(p)>0$, independent of $\ell$, such that
\begin{equation*}
    \norm{u}{\infty}{\G_\ell},\norm{u'}{\infty}{\G_\ell}\leq C
\end{equation*}
and
\begin{equation}\label{derpointestiamte}
   0\leq  u'_3(x)\lesssim u_3(x)+o(1)\ \text{for all}\ x\in[0,\ell],\ \text{as}\ \ell\to\infty. 
\end{equation}

\end{lemma}
\begin{proof}
Since $u$ is a type $\mathcal{A}$ solution, Theorem \ref{Theo1:MAIN} and Lemma \ref{lem:lemma7.3} yield
$\norm{u}{\infty}{\G_\ell}=u_3(\ell)=f^{-1}\left(-3f(L^{-1}(\ell))\right)$ and the bound on $u$ follows from the uniform bound of the function $z\mapsto f^{-1}(-3f(z))$ in $(0,(p/2)^\frac{1}{p-2})=L^{-1}((0,\infty))$. 
Let $y(\ell)>0$ be the (unique) translation parameter given by Proposition \ref{prop:sol_halfline}. Since $\phi'$ is bounded,  $u'_1=u'_2=\varphi'(\cdot+y(\ell))$ are also uniformly bounded. Finally, 
recall that $u_3$ satisfies 
\begin{equation*}
    -\frac{1}{2}{u'_3}^2(x)+\frac{1}{2}u_3^2(x)-\frac{1}{p}u_3^p(x)=-3f(\varphi(y(\ell)))\ \text{ for }\ x\in[0,\ell].
 \end{equation*}
Therefore, as $\ell\to\infty$,
\begin{equation*}
    u'_3(x)=\sqrt{2}\sqrt{f(u_3(x))+3f(\varphi(y(\ell)))}\lesssim u_3(x)\sqrt{\frac{1}{2}-\frac{1}{p}u_3^{p-2}(x)}+o(1),\ \text{for all}\ x\in[0,\ell],
\end{equation*}
and thus we obtain  \eqref{derpointestiamte} and $\norm{u'_3}{\infty}{0,\ell}\leq C$.\qedhere
\end{proof}

%%%%%%%%%%%%%%%%%%%%%%% Regularização.
%%%%%%%%%%%%%%%%%%%%%%%%

\subsection{Geometric Condition for Large Terminal Edge}\label{sec:eta1}
In this section we construct, for $\ell\sim \infty$, the unstable path $t\mapsto \tilde \eta_t$ satisfying \eqref{eqn.dernegativa}.

\begin{lemma}\label{lem4.3}
Let $\psi:\R^+\to\R$ be given by
\begin{equation*}
    \psi(x)=\begin{cases}u_3(\ell-x),\ x\in[0,\ell],\\
u_2(x-\ell),\ x\in[\ell,\infty).\end{cases}
\end{equation*}
Then, for any $p\geq2$, $k\geq1$, $\gamma\geq k$ and $\varepsilon=\varepsilon(\ell)=\frac{1}{\ell^\gamma}$, $\lim_{\ell\to\infty} \frac{\ell^k}{\varepsilon}\psi^p(\ell-\varepsilon)=0$.
\end{lemma}
\begin{proof}
By Lemmas \ref{lemma4.12} and \ref{lem6.5}, for some $\eta\in(\ell-\varepsilon,\ell)$, we have
\begin{align*}\label{eqn.4.26}
    \frac{\ell^k}{\varepsilon}\psi^p(\ell-\varepsilon)&=\frac{\ell^k}{\varepsilon}(\psi^p(\ell-\varepsilon)-\psi^p(\ell))+\frac{\ell^k}{\varepsilon}\psi^p(\ell)=-p\ell^k \psi^{p-1}(\eta)\psi'(\eta)+o(1),\ \text{as}\ \ell\to\infty.
\end{align*}
Since $\psi^{p-1}$ is decreasing,
$$\psi^{p-1}(\eta)\leq\psi^{p-1}(\ell-\varepsilon)=(\psi^{p-1}(\ell-\varepsilon)-\psi^{p-1}(\ell))+\psi^{p-1}(\ell)=(p-1)\varepsilon\psi^{p-2}(\eta')\psi'(\eta')+o(1).$$
Therefore, from Lemma \ref{lem6.5}
\begin{equation*}
    \frac{\ell^k}{\varepsilon}\psi^p(\ell-\varepsilon)\lesssim
\ell^k\varepsilon \psi^{p-2}(\ell-\varepsilon) +o(1),\ \text{as}\ \ell\to\infty,
\end{equation*}
and the conclusion follows by taking $\gamma\geq k$ and $\varepsilon$ as in the statement.\qedhere
\end{proof}
Take $\phi\in C_c^\infty(\R)$ such that $\phi(0)=1$, $\phi'(0)=0$, $0\leq \phi\leq 1$ and $\supp\phi\subset[-1,1]$. For $\varepsilon>0$ small to be fixed later on, let
$$\phi_{\varepsilon,\ell}(x)=\phi\left(\frac{x-\ell}{\varepsilon}\right).$$
For $v\in H^1(\R)$ and $t>0$, consider the one-parameter family $v_t(x):=t^\frac{1}{2}v(tx)$ and define,
\begin{equation*}
   \xi^\varepsilon_t(x)=\psi\phi_{\varepsilon,\ell}+(\psi(1-\phi_{\varepsilon,\ell}))_t,\ \text{for}\  t,\varepsilon>0\ \text{and}\ x\geq0,
\end{equation*}
where $\psi$ is as defined in Lemma \ref{lem4.3}. We consider, on the graph $\G_\ell$, the following scaling
\begin{equation*}
\eta^\varepsilon_{t}:=\begin{cases}\eta^\varepsilon_{t,i}=\xi^\varepsilon_t(x+\ell),\ x\geq0,\ i=1,2,\\ \eta^\varepsilon_{t,3}=\xi^\varepsilon_t(\ell-x),\ x\in[0,\ell],\end{cases}
\end{equation*}
which, upon normalization, satisfies
\begin{equation*}
    \Tilde{\eta}_t^\varepsilon:=\frac{\sqrt{\nu}}{\norm{\eta^\varepsilon_t}{2}{\G_\ell}}\eta^\varepsilon_t\in\mathcal{Q},\ \text{for all}\ t,\varepsilon\geq0.
\end{equation*}
Together with the regularity properties of the ground state on the interior of the edges, this yields
\begin{lemma}\label{lem:curvasuave}
    The path $t\mapsto \tilde{\eta}_t^\varepsilon,$ for $t\sim 1$, is a $C^1$-curve on the manifold $\mathcal{Q}$.
\end{lemma}

\begin{remark}
    Let us explain the idea behind this particular choice of path. First, recall that, in the euclidean setting, instability in the $L^2$-supercritical case is caused by the presence of the scaling invariance. In the case of $\mathcal{T}$-graphs, one could consider the half-line composed of $e_3\cup e_1$, rescale, and define the values on $e_2$ by symmetry (this would correspond to $\eta^\varepsilon_t$ for $\phi_{\varepsilon,\ell}\equiv 1$). However, this procedure would move the lack of smoothness at the vertex to the interior of an edge and, in particular, the derivative of the path would not be in $H^1(\mathcal{G}_\ell).$ To bypass this difficulty, we fix the ground state around the vertex and rescale only in the interior of the edges, where the ground state in known to be $C^\infty$, with exponential decay on the unbounded edges.
\end{remark}

For any $t,\varepsilon>0$, by the definition of $\eta_t^\varepsilon$, we have
\begin{align}\label{enery along curve}
    E(\Tilde{\eta}_t^\varepsilon,\G_\ell)&=\frac{\nu}{2\norm{\eta^\varepsilon_t}{2}{\G_\ell}^2}\left(2\int_0^\infty|{\eta_{t,1}^\varepsilon}'|^2dx+\int_0^\ell|{\eta_{t,3}^\varepsilon}'|^2dx\right)-\frac{\nu^\frac{p}{2}}{p\norm{\eta^\varepsilon_t}{2}{\G_\ell}^p}\left(2\int_0^\infty|{\eta_{t,1}^\varepsilon}|^pdx+\int_0^\ell|{\eta_{t,3}^\varepsilon}|^pdx\right)\nonumber\\
    &=:\frac{\nu}{2\norm{\eta^\varepsilon_t}{2}{\G_\ell}^2}X_\varepsilon(t)-\frac{\nu^\frac{p}{2}}{p\norm{\eta^\varepsilon_t}{2}{\G_\ell}^p}\Tilde{X}_\varepsilon(t).
\end{align}

\begin{remark}
Note that, by construction, we have
\begin{equation*}
X_\varepsilon(1)=\int_{\G_\ell}|u'|^2dx,\qquad \ \Tilde{X}_\varepsilon(1)=\int_{\G_\ell}|u|^pdx\ \text{for all $\varepsilon>0$},
\end{equation*}
and, in particular,
\begin{equation}\label{solutionIdentitywithX}
    X_\varepsilon(1)+\nu=\Tilde{X}_\varepsilon(1)\ \text{for all}\ \varepsilon>0.
\end{equation}
\end{remark}

\begin{lemma}\label{lemmaFirstderivative}
    Let $\varepsilon>0$ be fixed. Then
    \begin{equation*}
         \left.\frac{d}{dt}E(\Tilde{\eta}_t^\varepsilon,\G_\ell)\right|_{t=1}=\frac{1}{2}X_\varepsilon'(1)-\frac{1}{p}\Tilde{X}_\varepsilon'(1)+\frac{1}{2}\left.\frac{d}{dt}\norm{\eta_t^\varepsilon}{2}{\G_\ell}^2\right|_{t=1}.
    \end{equation*}
Furthermore, the following Pohozaev type identity holds:
\begin{equation}\label{PohozaevType}
    -\frac{1}{2}\left.\frac{d}{dt}\norm{\eta_t^\varepsilon}{2}{\G_\ell}^2\right|_{t=1}=\frac{1}{2}X_\varepsilon'(1)-\frac{1}{p}\Tilde{X}_\varepsilon'(1).
\end{equation}
\end{lemma}
\begin{proof}
A direct computation yields
\begin{align*}
    \left.\frac{d}{dt}E(\Tilde{\eta}_t^\varepsilon,\G_\ell)\right|_{t=1}&=\frac{1}{2\nu}\left(X'_\varepsilon(1)\nu-X_\varepsilon(1)\left.\frac{d}{dt}\norm{\eta_t^\varepsilon}{2}{\G_\ell}^2\right|_{t=1}\right)\\
    &-\frac{1}{p\nu^{p-\frac{p}{2}}}\left(\Tilde{X}'_\varepsilon(1) \nu^\frac{p}{2}-\frac{p}{2}\Tilde{X}_\varepsilon(1)\nu^\frac{p-2}{2}\left.\frac{d}{dt}\norm{\eta_t^\varepsilon}{2}{\G_\ell}^2\right|_{t=1}\right)\\
    &=\frac{1}{2}X_\varepsilon'(1)-\frac{1}{p}\Tilde{X}'_\varepsilon(1)+\frac{1}{2\nu}\left(\Tilde{X}_\varepsilon(1)-X_\varepsilon(1)\right)\left.\frac{d}{dt}\norm{\eta_t^\varepsilon}{2}{\G_\ell}^2\right|_{t=1}\\
    &=\frac{1}{2}X_\varepsilon'(1)-\frac{1}{p}\Tilde{X}_\varepsilon'(1)+\frac{1}{2}\left.\frac{d}{dt}\norm{\eta_t^\varepsilon}{2}{\G_\ell}^2\right|_{t=1},
\end{align*}
where the last equality follows from \eqref{solutionIdentitywithX}. Furthermore, since $u$ is an action ground state, we know that 
$\frac{d}{dt}S(\Tilde{\eta}_t^\varepsilon,\G_\ell)=\frac{d}{dt}E(\Tilde{\eta}_t^\varepsilon,\G_\ell)=0$ at $t=1$, which yields \eqref{PohozaevType}.
\end{proof}

\begin{lemma}\label{lemma Secondderivative}
    Let $\varepsilon>0$ be fixed. Then 
\begin{align}\label{secondderivative}
    \left.\frac{d^2}{dt^2}E(\Tilde{\eta}_t^\varepsilon,\G_\ell)\right|_{t=1}&=\frac{X''_\varepsilon(1)}{2}-\frac{\Tilde{X}_\varepsilon''(1)}{p}+\frac{1}{2}\left.\frac{d^2}{dt^2}\norm{\eta_t^\varepsilon}{2}{\G_\ell}^2\right|_{t=1}+\frac{1}{\nu}\left(1-\frac{2}{p}\right)\Tilde{X}'_\varepsilon(1)\left.\frac{d}{dt}\norm{\eta_t^\varepsilon}{2}{\G_\ell}^2\right|_{t=1}\nonumber\\
    &+\frac{1}{\nu^2}\Tilde{X}_\varepsilon(1)\left(\frac{1}{2}-\frac{p}{4}\right)\left(\left.\frac{d}{dt}\norm{\eta_t^\varepsilon}{2}{\G_\ell}^2\right|_{t=1}\right)^2.
    \end{align}
\end{lemma}
\begin{proof}
In what follows, let $A:=\left.\frac{d}{dt}\norm{\eta_t^\varepsilon}{2}{\G_\ell}^2\right|_{t=1}$ and $B:=\left.\frac{d^2}{dt^2}\norm{\eta_t^\varepsilon}{2}{\G_\ell}^2\right|_{t=1}$. Direct computations yield
\begin{align}\label{SecondDerintermedstep}
   \left.\frac{d^2}{dt^2}E(\Tilde{\eta}_t^\varepsilon,\G_\ell)\right|_{t=1}&=\frac{1}{2\nu^3}\left[X''_\varepsilon(1)\nu^3-X_\varepsilon(1)B\nu^2-2X'_\varepsilon(1)\nu^2A+2X_\varepsilon(1)\nu A^2\right]\nonumber\\
   &-\frac{1}{p\nu^{2p-\frac{p}{2}}}\left[\Tilde{X}''_\varepsilon(1)\nu^{p+\frac{p}{2}}-    \Tilde{X}_\varepsilon(1)\left.\frac{d^2}{dt^2}\norm{\eta_t^\varepsilon}{2}{\G_\ell}^{p}\right|_{t=1}\nu^p-p\Tilde{X}'_\varepsilon(1)A\nu^{\frac{p}{2}+(p-1)}\right.\nonumber\\
   &\qquad \quad \qquad \left.+\ p\Tilde{X}_\varepsilon(1) A\nu^{p-1}\left.\frac{d}{dt}\norm{\eta_t^\varepsilon}{2}{\G_\ell}^{p}\right|_{t=1}\right].
   \end{align}  
Furthermore,
\begin{equation}\label{eqn.derivadanormapt=1*}
    \left.\frac{d}{dt}\norm{\eta_t^\varepsilon}{2}{\G_\ell}^p\right|_{t=1}=\frac{p}{2}\nu^{\frac{p-2}{2}}A\ \text{and}\ \left.\frac{d^2}{dt^2}\norm{\eta_t^\varepsilon}{2}{\G_\ell}^p\right|_{t=1}=\frac{p(p-2)}{4}\nu^{\frac{p-4}{2}}A^2+\frac{p}{2}\nu^{\frac{p-2}{2}}B.
\end{equation}
Plugging equations in \eqref{eqn.derivadanormapt=1*} into \eqref{SecondDerintermedstep}, it follows that
\begin{align*}
   \left.\frac{d^2}{dt^2}E(\Tilde{\eta}_t^\varepsilon,\G_\ell)\right|_{t=1}
   &=\frac{X''_\varepsilon(1)}{2}-\frac{\Tilde{X}''_\varepsilon(1)}{p}
   +\left(\Tilde{X}_\varepsilon(1)-X_\varepsilon(1)\right)\frac{B}{2\nu}
   +\left(\Tilde{X}'_\varepsilon(1)-X'_\varepsilon(1)\right)\frac{A}{\nu}\\
   &+\left(X_\varepsilon(1)-\left(\frac{p}{4}+\frac{1}{2}\right)\Tilde{X}_\varepsilon(1)\right)\frac{A^2}{\nu^2}.
\end{align*}
Using \eqref{solutionIdentitywithX} and \eqref{PohozaevType}, the expression in \eqref{secondderivative} follows from the previous computation.\qedhere
\end{proof}

\begin{lemma}\label{lemma4.13}
    For $\varepsilon=\varepsilon(\ell)=1/\ell^2$, the following identities hold:
\begin{align}
X'_\varepsilon(1)&=2\int_0^\ell|\psi'(x)|^2dx+o(1),\ \text{as}\ \ell\to\infty,\label{eqn.4.22}\\    \Tilde{X}_\varepsilon '(1)&=\left(\frac{p}{2}-1\right)\int_0^\ell|\psi(x)|^pdx+ o(1),\ \text{as}\ \ell\to\infty,\label{eqn.4.23}\\
\lim_{\ell\to\infty}&\left.\frac{d}{dt}\norm{\eta_t^\varepsilon}{2}{\G_\ell}^2\right|_{t=1}=0.\label{eqn.4.24}
\end{align}
Moreover,
\begin{equation}\label{PhozaevO(ell)}
\int_0^\ell|\psi'(x)|^2dx=\frac{1}{p}\left(\frac{p}{2}-1\right)\int_0^\ell|\psi(x)|^pdx+ o(1),\ \text{as}\ \ell\to\infty. 
\end{equation} 
\end{lemma}
\begin{proof} The identity \eqref{PhozaevO(ell)} follows from taking the limit in \eqref{PohozaevType} and using \eqref{eqn.4.22}, \eqref{eqn.4.23} and \eqref{eqn.4.24}.
For \eqref{eqn.4.23}, a direct computation shows that
 \begin{align*}
\Tilde{X}'_\varepsilon(1)&=\left(\int_0^\infty+\int_\ell^\infty\right)p\psi^{p-1}\left(\left(\frac{1}{2}\psi+x\psi'\right)\left(1-\phi\left(\frac{x-\ell}{\varepsilon}\right)\right)-\frac{x}{\varepsilon}\psi\phi'\left(\frac{x-\ell}{\varepsilon}\right)\right)dx\nonumber\\
&=2\int_\ell^\infty\left(\frac{p}{2}\psi^p(x)+x\frac{d}{dx}\psi^p(x)\right)\left(1-\phi\left(\frac{x-\ell}{\varepsilon}\right)\right)-\frac{p}{\varepsilon}x\psi^p(x)\phi'\left(\frac{x-\ell}{\varepsilon}\right)dx\ +\nonumber\\
&+\int_0^\ell\left(\frac{p}{2}\psi^p(x)+x\frac{d}{dx}\psi^p(x)\right)\left(1-\phi\left(\frac{x-\ell}{\varepsilon}\right)\right)-\frac{p}{\varepsilon}x\psi^p(x)\phi'\left(\frac{x-\ell}{\varepsilon}\right)dx.
\end{align*}
 Integrating by parts in the terms with $\frac{d}{dx}\psi^p(x)$, due to the exponential decay of $\varphi$ and the compact support of $\phi'$, we get
 \begin{align*}
 \Tilde{X}_\varepsilon'(1)&=(p-2)\int_\ell^\infty \psi^p(x)\left(1-\phi\left(\frac{x-\ell}{\varepsilon}\right)\right)-\frac{2(p-1)}{\varepsilon} x\psi^p(x)\phi'\left(\frac{x-\ell}{\varepsilon}\right)dx\nonumber\\
 &+\left(\frac{p}{2}-1\right)\int_0^\ell\psi^p(x)\left(1-\phi\left(\frac{x-\ell}{\varepsilon}\right)\right)-\frac{(p-1)}{\varepsilon} x\psi^p(x)\phi'\left(\frac{x-\ell}{\varepsilon}\right)dx=:I(\ell)+II(\ell).
 \end{align*}
 
We claim that $$I(\ell)\to 0\  \text{ and }\ 
 II(\ell)=\left(\frac{p}{2}-1\right)\int_0^\ell\psi^p(x)\left(1-\phi\left(\frac{x-\ell}{\varepsilon}\right)\right)dx+o(1)\ \text{as}\ \ell\to\infty.$$
 From the definitions of $\psi$ and $\phi$ and the Dominated Convergence Theorem, 
 $$(p-2)\int_\ell^\infty\psi^p(x)\left(1-\phi\left(\frac{x-\ell}{\varepsilon}\right)\right)dx\lesssim\int_0^\infty\varphi^p(x+y(\ell))dx\to 0.$$
Additionally, since $\psi^p$ is decreasing, it follows from Lemma \ref{lemma4.12} that
$$\left|-\frac{2(p-1)}{\varepsilon}\int_\ell^\infty x\psi^p(x)\phi'\left(\frac{x-\ell}{\varepsilon}\right)dx\right|\lesssim \frac{\ell}{\varepsilon}\int_\ell^{\ell+\varepsilon}\psi^p(x)dx\lesssim\ell\psi^p(\ell)\to0.$$
In particular, we have that $I(\ell)\to0$, as $\ell\to\infty$. Similar estimates yield 
$$\left|-\frac{p-1}{\varepsilon}\int_0^\ell x\psi^p(x)\phi'\left(\frac{x-\ell}{\varepsilon}\right)dx\right|\lesssim \frac{\ell}{\varepsilon}\int_{\ell-\varepsilon}^\ell\psi^p(x)dx\lesssim \ell\psi^p(\ell-\varepsilon).$$
The same exact argument as above now shows that
$$-\left(\frac{p}{2}-1\right)\int_0^\ell\psi^p(x)\phi\left(\frac{x-\ell}{\varepsilon}\right)=o(1),\ \text{as}\ \ell\to\infty,$$
which yields  the claim \eqref{eqn.4.23}. 

Direct computations show that, at $t=1$, the expression of $\norm{\eta_t^\varepsilon}{2}{\G_\ell}^2$ corresponds to the expression $\Tilde{X}_\varepsilon$ with $p=2$ and thus \eqref{eqn.4.24} follows by taking $p=2$ in the previous argument.

Finally, we focus on \eqref{eqn.4.22}. Computing $X'_\varepsilon(1)$, we have
\begin{align*}
X'_\varepsilon(1)&=\int_0^\ell\left(3{\psi'}^2+x\frac{d}{dx}{\psi'}^2\right)\left(1-\phi\left(\frac{x-\ell}{\varepsilon}\right)\right)-\frac{1}{\varepsilon}\left(\frac{3}{2}\frac{d}{dx}{\psi}^2+4x{\psi'}^2\right)\phi'\left(\frac{x-\ell}{\varepsilon}\right)-\frac{2x}{\varepsilon^2}\psi\psi'\phi''\left(\frac{x-\ell}{\varepsilon}\right)dx\nonumber\\
&+2\int_\ell^\infty\left(3{\psi'}^2+x\frac{d}{dx}{\psi'}^2\right)\left(1-\phi\left(\frac{x-\ell}{\varepsilon}\right)\right)-\frac{1}{\varepsilon}\left(\frac{3}{2}\frac{d}{dx}{\psi}^2+4x{\psi'}^2\right)\phi'\left(\frac{x-\ell}{\varepsilon}\right)-\frac{2x}{\varepsilon^2}\psi\psi'\phi''\left(\frac{x-\ell}{\varepsilon}\right)dx\nonumber\\
&=:I(\ell)+II(\ell).
\end{align*}
Once again, the monotonicity and decay of the soliton and the support of $\phi$ and its derivatives imply that $II(\ell)=o(1),\ \text{as}\ \ell\to\infty $. We then focus on $I(\ell)$. Integrating by parts the term with $\frac{d}{dx}{\psi'}^2$ and $\frac{d}{dx}{\psi}^2$,
\begin{align*}
X'_\varepsilon(1)&=2\int_0^\ell{\psi'}^2\left(1-\phi\left(\frac{x-\ell}{\varepsilon}\right)\right)dx-\frac{3}{\varepsilon}\int_0^\ell x{\psi'}^2\phi'\left(\frac{x-\ell}{\varepsilon}\right)dx-\frac{1}{\varepsilon^2}\int_0^\ell\left(x\frac{d}{dx}\psi^2-\frac{3}{2}\psi^2\right)\phi''\left(\frac{x-\ell}{\varepsilon}\right)dx\nonumber\\
&=:A(\ell)+B(\ell)+C(\ell).
\end{align*}
Since, by Lemma \ref{lem6.5}, $\psi'$ is bounded independently of $\ell$,
\begin{align*}
\left|A(\ell)-2\int_0^\ell {\psi'}^2dx\right|&\leq\left|2\int_0^\ell{\psi'}^2\phi\left(\frac{x-\ell}{\varepsilon}\right)dx\right|\lesssim\int_{\ell-\varepsilon}^\ell{\psi'}^2dx\lesssim \varepsilon.
\end{align*}
Similarly, using the estimate \eqref{derpointestiamte}, we have
$$|B(\ell)|\lesssim \frac{\ell}{\varepsilon}\int_{\ell-\varepsilon}^\ell{\psi'}^2dx\lesssim\ell {\psi'}^2(\ell-\varepsilon)\lesssim\ell\psi^2(\ell-\varepsilon)+o(1)$$
and, since $\psi$ is decreasing,
\begin{align*}
    |C(\ell)|&\lesssim\frac{\ell}{\varepsilon^2}\psi(\ell-\varepsilon)\int_{\ell-\varepsilon}^\ell|\psi'(x)|dx+\frac{1}{\varepsilon^2}\int_{\ell-\varepsilon}^\ell\psi^2(x)dx\\&\lesssim \frac{\ell}{\varepsilon^2}\psi(\ell-\varepsilon)\int_{\ell-\varepsilon}^\ell\psi(x)+o(1)dx+\frac{1}{\varepsilon^2}\int_{\ell-\varepsilon}^\ell\psi^2(x)dx\\&\lesssim\frac{\ell}{\varepsilon}\psi^2(\ell-\varepsilon)+\frac{\ell}{\varepsilon}\psi(\ell-\varepsilon)o(1)+\frac{1}{\varepsilon}\psi^2(\ell-\varepsilon).
\end{align*}
Choosing $\varepsilon=1/\ell^2$, Lemma \ref{lem4.3} implies that
$$
A(\ell)=2\int_0^\ell \psi'^2dx + o(1),\quad |B(\ell)|+|C(\ell)|=o(1)
$$
and the conclusion follows.
\end{proof}

In the next lemma we study the asymptotic behaviour of the remaining terms in \eqref{secondderivative}.

\begin{lemma}\label{lemma4.14}
     For $\varepsilon=\varepsilon(\ell)=\frac{1}{\ell^2}$, the following identities hold:
\begin{align}
X''_\varepsilon(1)&=2\int_0^\ell{\psi'}^2dx+o(1),\ \text{as}\ \ell\to\infty,\label{eqn.4.27}\\
    \Tilde{X}''_\varepsilon(1)&=\left(\frac{p}{2}-2\right)\left(\frac{p}{2}-1\right)\int_0^\ell \psi^pdx+o(1),\ \text{as}\ \ell\to\infty,\label{eqn.4.28}\\
    \frac{d^2}{dt^2}&\left.\norm{\eta_t^\varepsilon}{2}{\G_\ell}^2\right|_{t=1}=o(1),\ \text{as}\ \ell\to\infty.\label{eqn.4.29}
\end{align}
\end{lemma}
\begin{proof}
As before, \eqref{eqn.4.29} follows from \eqref{eqn.4.28} taking $p=2$. We start by focusing on \eqref{eqn.4.28}.

\begin{align}\label{Xtilde''def}
\Tilde{X}''_\varepsilon(1)=&\left(\int_0^\ell+2\int_\ell^\infty\right)p(p-1)\psi^{p-2}\left(\left(\frac{1}{2}\psi+x\psi'\right)\left(1-\phi\left(\frac{x-\ell}{\varepsilon}\right)\right)-\frac{x}{\varepsilon}\psi\phi'\left(\frac{x-\ell}{\varepsilon}\right)\right)^2dx\nonumber\\
&+\left(\int_0^\ell+2\int_\ell^\infty\right)p\psi^{p-1}\Bigg(\left(-\frac{1}{4}\psi+x\psi'+x^2\psi''\right)\left(1-\phi\left(\frac{x-\ell}{\varepsilon}\right)\right)\nonumber\\ 
&\qquad -\frac{1}{\varepsilon}(x\psi+2x^2\psi')\phi'\left(\frac{x-\ell}{\varepsilon}\right)-\frac{x^2}{\varepsilon^2}\psi\phi''\left(\frac{x-\ell}{\varepsilon}\right)\Bigg)dx.
\end{align}
We now show that the integrals over $[\ell,\infty)$ tend to $0$ as $\ell\to\infty$.
Firstly,
\begin{align*}
   &2\int_\ell^\infty p(p-1)\psi^{p-2}\left(\left(\frac{1}{2}\psi+x\psi'\right)\left(1-\phi\left(\frac{x-\ell}{\varepsilon}\right)\right)-\frac{x}{\varepsilon}\psi\phi'\left(\frac{x-\ell}{\varepsilon}\right)\right)^2dx\nonumber\\
   &\leq2\int_\ell^\infty p(p-1)\psi^{p-2}\left(\frac{1}{2}\psi+x\psi'\right)^2\left(1-\phi\left(\frac{x-\ell}{\varepsilon}\right)\right)+p(p-1)\psi^p\frac{x^2}{\varepsilon^2}{\phi'}^2\left(\frac{x-\ell}{\varepsilon}\right)dx.\nonumber
\end{align*}
The explicit form of $\psi$ and $\psi'$ yield:
\begin{align*}
    \int_\ell^\infty\psi^{p-2}\left(\frac{1}{2}\psi+x\psi'\right)^2\left(1-\phi\left(\frac{x-\ell}{\varepsilon}\right)\right)^2dx&\lesssim\int_\ell^\infty\psi^p(x)(1+x)^2dx\\
    				&=\int_{\ell+y(\ell)}^\infty\varphi^p(z)(1+z)^2dz=o(1),
\end{align*}
where the last equality holds due to the exponential decay of the real line soliton $\varphi$. 

Given that $\phi'$ is bounded, it follows from Lemma \ref{lemma4.12} that, for any $\varepsilon>0$,
\begin{align*}
    p(p-1)\int_\ell^\infty\psi^p(x)\frac{x^2}{\varepsilon^2}{\phi'}^2\left(\frac{x-\ell}{\varepsilon}\right)dx&\lesssim\frac{(\ell+\varepsilon)^2}{\varepsilon}\psi^p(\ell)\lesssim\frac{\ell^2}{\varepsilon}\psi^p(\ell)+\varepsilon\psi^p(\ell)=o(1),\ \text{as}\ \ell\to\infty.
\end{align*}
Using the fact that $\psi$ solves the equation $-\psi''+\psi=\psi^{p-1}$, one can use similar arguments to show that 
\begin{align*}
    &2\int_\ell^\infty p\psi^{p-1}\left[\left(-\frac{1}{4}\psi+x\psi'+x^2\psi''\right)\left(1-\phi\left(\frac{x-\ell}{\varepsilon}\right)\right)-\frac{1}{\varepsilon}(x\psi+2x^2\psi')\phi'\left(\frac{x-\ell}{\varepsilon}\right)\right.\nonumber\\
    &\left.
\qquad  \qquad \qquad -\frac{x^2}{\varepsilon^2}\psi\phi''\left(\frac{x-\ell}{\varepsilon}\right)\right]dx=o(1),\ \text{as}\ \ell\to\infty.
\end{align*}
As $\ell\to\infty$, write $\Tilde{X}''_\varepsilon(1)=I(\ell)+II(\ell)+o(1)$, where $I(\ell),\ II(\ell)$ are the integrals over $[0,\ell]$ of both terms in \eqref{Xtilde''def}, respectively. We rewrite $I(\ell)$ as
\begin{align*}
    I(\ell)&=p(p-1)\int_0^\ell \psi^{p-2}\left(\frac{1}{2}\psi+x\psi'\right)^2\left(1-2\phi\left(\frac{x-\ell}{\varepsilon}\right)+\phi^2\left(\frac{x-\ell}{\varepsilon}\right)\right)dx\nonumber\\
    &+p(p-1)\int_0^\ell\frac{2x}{\varepsilon}\psi^{p-1}\left(\frac{1}{2}\psi+x\psi'\right)\left(1-\phi\left(\frac{x-\ell}{\varepsilon}\right)\right)\phi'\left(\frac{x-\ell}{\varepsilon}\right)dx\nonumber\\
    &+p(p-1)\int_0^\ell\frac{x^2}{\varepsilon^2}\psi^p{\phi'}^2\left(\frac{x-\ell}{\varepsilon}\right)=:A(\ell)+B(\ell)+C(\ell).
\end{align*}
Since $\phi'$ is bounded, $\supp\phi'(\frac{\cdot-\ell}{\varepsilon})\subset[\ell-\varepsilon,\ell+\varepsilon]$ and $\psi$ is monotone decreasing, we have
$$|C(\ell)|\lesssim\frac{\ell^2}{\varepsilon^2}\int_{\ell-\varepsilon}^\ell\psi^p(x)dx\leq \frac{\ell^2}{\varepsilon}\psi^p(\ell-\varepsilon).$$
Similarly, from estimate \eqref{derpointestiamte},
\begin{equation*}
    |B(\ell)|\lesssim\frac{1}{\varepsilon}\int_{\ell-\varepsilon}^\ell x\psi^{p}+x^2\psi^{p-1}|\psi'|dx\leq\ell\psi^{p}(\ell-\varepsilon)+\ell^2\psi^p(\ell-\varepsilon) +o(1).
\end{equation*}
As for the term $A(\ell)$, it follows from \eqref{derpointestiamte} and the monotonicity of $\psi$ that
\begin{align*}
    \left|A(\ell)-p(p-1)\int_0^\ell \psi^{p-2}\left(\frac{1}{2}\psi+x\psi'\right)^2dx\right|&\lesssim\int_{\ell-\varepsilon}^\ell\psi^{p-2}\left(\frac{1}{2}\psi+x\psi'\right)^2dx\\&\lesssim \varepsilon\psi^{p}(\ell-\varepsilon)+\ell^2\varepsilon\psi^{2}(\ell-\varepsilon) +o(1).
\end{align*}
Thus, Lemma \ref{lem4.3} with $k=2$ yields, as $\ell\to\infty$,
$$A(\ell)=p(p-1)\int_0^\ell \psi^{p-2}\left(\frac{1}{2}\psi+x\psi'\right)^2dx+o(1)\ \text{and}\ |B(\ell)|+|C(\ell)|=o(1).$$
In particular,
$$I(\ell)=p(p-1)\int_0^\ell \psi^{p-2}\left(\frac{1}{2}\psi+x\psi'\right)^2dx+o(1).$$
For $II(\ell)$, using the estimate \eqref{derpointestiamte} and the fact that $\psi$ solves $-\psi''+\psi=\psi^{p-1}$,
\begin{align*}
    &\left|II(\ell)-\int_0^\ell p\psi^{p-1}\left(-\frac{1}{4}\psi+x\psi'+x^2\psi''\right)dx\right|\\&\leq \int_{\ell-\varepsilon}^\ell p\psi^{p-1}\left|-\frac{1}{4}\psi+x\psi'+x^2\psi''\right|dx\nonumber
    +\frac{1}{\varepsilon}\int_{\ell-\varepsilon}^\ell p\psi^{p-1}|x\psi+2x^2\psi'| dx+\frac{\ell^2}{\varepsilon^2}\int_{\ell-\varepsilon}^\ell p\psi^pdx\\
    &\lesssim \varepsilon\psi^p(\ell-\varepsilon)\left(1+\ell+\ell^2\right)+\ell\psi^p(\ell-\varepsilon)+\ell^2\psi^p(\ell-\varepsilon)+\frac{\ell^2}{\varepsilon}\psi^p(\ell-\varepsilon)+o(1),\ \text{as}\ \ell\to\infty
\end{align*}
Applying Lemma \ref{lem4.3} with $k=2$, 
$$II(\ell)=\int_0^\ell p\psi^{p-1}\left(-\frac{1}{4}\psi+x\psi'+x^2\psi''\right)dx+o(1),\ \text{as}\ \ell\to\infty,$$
whence
\begin{align*}
    \Tilde{X}''_\varepsilon(1)&=I(\ell)+II(\ell)+o(1)=p(p-1)\int_0^\ell \frac{1}{4}\psi^pdx+(p-1)\int_0^\ell x\frac{d}{dx}\psi^pdx+p\int_0^\ell x^2{\psi'}\frac{d}{dx}\psi^{p-1}dx\\
    &-\frac{p}{4}\int_0^\ell \psi^{p}+\int_0^\ell x\frac{d}{dx}\psi^pdx+p\int_0^\ell x^2\psi^{p-1}\psi''dx+o(1)\ \text{as}\ \ell\to\infty.
\end{align*}
Since the boundary terms are all of order $o(1)$ as $\ell\to\infty$, integrating by parts,
\begin{align*}
    \Tilde{X}''_\varepsilon(1)&=
    \frac{1}{4}(p-4)(p-2)\int_0^\ell\psi^pdx +o(1)=\left(\frac{p}{2}-2\right)\left(\frac{p}{2}-1\right)\int_0^\ell \psi^pdx+o(1)\ \text{as}\ \ell\to\infty.
\end{align*}
To finish, we prove \eqref{eqn.4.27}. Using the previous arguments and Lemma \ref{lem4.3} with $k=2$, we obtain
\begin{align*}
X''_\varepsilon(1)&=2\int_0^\ell\left(\frac{3}{2}\psi'+x\psi''\right)^2 dx+2\int_0^\ell \psi'\left(\frac{3}{4}\psi'+3x\psi''+x^2\psi'''\right)dx+o(1)\\&=2\int_0^\ell \frac{9}{4}{\psi'}^2+3x\psi'\psi''+x^2{\psi''}^2dx+2\int_0^\ell\frac{3}{4}{\psi'}^2+3x\psi'\psi''+x^2{\psi'}\psi'''dx+o(1)\\
&=6\int_0^\ell{\psi'}^2dx+6\left(\left.x{\psi'}^2\right|_0^\ell-\int_0^\ell{\psi'}^2dx\right)+2\int_0^\ell x^2(\psi'')^2\nonumber\\
&+2\left(\left.x^2{\psi'}\psi''\right|_0^\ell-\int_0^\ell (2x\psi'+x^2\psi'')\psi''dx\right)+o(1)\\
&=-4\int_0^\ell x\psi'\psi''dx+o(1)=-2\left(\left.x{\psi'}^2\right|_0^\ell-\int_0^\ell{\psi'}^2dx\right)+o(1)=2\int_0^\ell{\psi'}^2dx+o(1).\qedhere
\end{align*}
\end{proof}

\begin{proposition}\label{prop:lgrande}
    For $\varepsilon=\varepsilon(\ell)=\frac{1}{\ell^2}$ we have, as $\ell\to\infty$,
    \begin{equation*}
        \left.\frac{d^2}{dt^2}E(\Tilde{\eta}_t^\varepsilon,\G_\ell)\right|_{t=1}=-\frac{1}{p}\left(\frac{p}{2}-1\right)\left(\frac{p-6}{2}\right)\int_0^\ell \psi^pdx+o(1).
    \end{equation*}
In particular, for $p>6$ and $\ell$ sufficiently large, 
\begin{equation*}
        \left.\frac{d^2}{dt^2}E(\Tilde{\eta}_t^\varepsilon,\G_\ell)\right|_{t=1}<0
    \end{equation*}
and \eqref{GCI} holds.
\end{proposition}
\begin{proof}
    Plugging the expressions of Lemmas \ref{lemma4.13} and \ref{lemma4.14} into equation \eqref{secondderivative} we obtain
    \begin{align*}
         \left.\frac{d^2}{dt^2}E(\Tilde{\eta}_t^\varepsilon,\G_\ell)\right|_{t=1}&=\int_0^\ell{\psi'}^2dx-\frac{1}{p}\left(\frac{p}{2}-2\right)\left(\frac{p}{2}-1\right)\int_0^\ell \psi^pdx+o(1),\ \text{as}\ \ell\to\infty,
    \end{align*}
and the conclusion follows from \eqref{PohozaevType}. 
\end{proof}

\subsection{Geometric Condition for Small Terminal Edge}\label{sec:eta2}
In this section we construct, for $\ell\sim 0$, the unstable path $t\mapsto \tilde \eta_t$ satisfying \eqref{eqn.dernegativa}. Recall that $u\in H^1(\G_\ell)$ is an action ground state, $\psi$ is as in Lemma \ref{lem4.3} and $\varphi$ is the real line soliton with $\lambda=1$.
Let
\begin{equation*}
    \xi_t^\ell(x)=t^{\frac{1}{2}}\varphi(tx+(y(\ell)-\ell))+(\psi(x)-\varphi(x+(y(\ell)-\ell)),\ \text{for}\ x\geq0.
\end{equation*}
We define now the following scaling on the graph:
\begin{equation*}
\eta^\ell_{t}:=\begin{cases}\eta^\ell_{t,1}=\xi^\ell_t(x+\ell),\ x\geq0,\\ \eta^\ell_{t,2}=\xi^\ell_t(x+\ell),\ x\geq0,\\
\eta^\ell_{t,3}=\xi^\ell_t(\ell-x)\ x\in[0,\ell].\end{cases}
\end{equation*}
Normalizing in $L^2(\G_\ell)$ yields
\begin{equation*}
    \Tilde{\eta}_t^\ell(x)=\frac{\sqrt{\nu}}{\norm{\eta^\ell_t}{2}{\G_\ell}}\eta^\ell_t~\in \mathcal{Q}.
\end{equation*}

\begin{remark}
    Here, as $\ell\to 0$, the difference between $\psi$ and a portion of the soliton tends to zero. As such, we construct the unstable path just by rescaling the soliton, which is smooth everywhere (as opposed to $\psi$, which is not smooth at the vertex). In particular, Lemma \ref{lem:curvasuave} holds true for the curve $\tilde{\eta}^\ell_t$.
\end{remark}

For any $t,\ell>0$, by definition of $\eta_t^\ell$, 
\begin{align*}
    E(\Tilde{\eta}_t^\ell,\G_\ell)&=\frac{\nu}{2\norm{\eta^\ell_t}{2}{\G_\ell}^2}\left(2\int_0^\infty|{\eta_{t,1}^\ell}'|^2dx+\int_0^\ell|{\eta_{t,3}^\ell}'|^2dx\right)-\frac{\nu^\frac{p}{2}}{p\norm{\eta^\ell_t}{2}{\G_\ell}^p}\left(2\int_0^\infty|{\eta_{t,1}^\ell}|^pdx+\int_0^\ell|{\eta_{t,3}^\ell}|^pdx\right)\\
    &=:\frac{\nu}{2\norm{\eta^\ell_t}{2}{\G_\ell}^2}X_\ell(t)-\frac{\nu^\frac{p}{2}}{p\norm{\eta^\ell_t}{2}{\G_\ell}^p}\Tilde{X}_\ell(t).
\end{align*}
With this change of notation, the general computations in Lemmas \ref{lemmaFirstderivative} and \ref{lemma Secondderivative} hold simply by replacing $\varepsilon$ with $\ell$.
Furthermore, as before, the expression of
$\norm{\eta_t^\ell}{2}{\G_\ell}^2$ coincides exactly with that of $\Tilde{X}_\ell(t)$ for $p=2$.

Differentiating directly under the integral sign, one can easily derive the following lemma:
\begin{lemma}\label{lemma derivadas lpequeno}
    For $\ell>0$,
 \begin{align*}
X''_\ell(1)&=\left(\int_0^\ell+2\int_\ell^\infty\right)2\left(\frac{3}{2}\varphi'(x+(y(\ell)-\ell))+x\varphi''(x+(y(\ell)-\ell))\right)^2dx\nonumber\\
&+\left(\int_0^\ell+2\int_\ell^\infty\right)2\psi'\left(\frac{3}{4}\varphi'(x+(y(\ell)-\ell))+3x\varphi''(x+(y(\ell)-\ell))+x^2\varphi'''(x+(y(\ell)-\ell))\right)dx,\\
\Tilde{X}'_\ell(1)&=\left(\int_0^\ell+2\int_\ell^\infty\right)p\psi^{p-1}\left(\frac{1}{2}\varphi(x+(y(\ell)-\ell))+x\varphi'(x+(y(\ell)-\ell))\right)dx
\end{align*}
and
\begin{align*}
\Tilde{X}''_\ell(1)&=\left(\int_0^\ell+2\int_\ell^\infty\right)p(p-1)\psi^{p-2}\left(\frac{1}{2}\varphi(x+(y(\ell)-\ell))+x\varphi'(x+(y(\ell)-\ell))\right)^2dx\nonumber\\
&+\left(\int_0^\ell+2\int_\ell^\infty\right)p\psi^{p-1}\left(-\frac{1}{4}\varphi(x+(y(\ell)-\ell))+x\varphi'(x+(y(\ell)-\ell))+x^2\varphi''(x+(y(\ell)-\ell))\right)dx.
\end{align*}
\end{lemma}

\begin{lemma}\label{lemma.4.15}
As $\ell\to 0 $,
\begin{align}
    \lim_{\ell\to0} \Tilde{X}'_\ell(1)&=\left(\frac{p}{2}-1\right)\int_\R\varphi^pdx,\label{eqn.4.39}\\
    \lim_{\ell\to 0} \Tilde{X}''_\ell(1)&=\left(\frac{p}{2}-2\right)\left(\frac{p}{2}-1\right)\int_\R\varphi^pdx,\label{eqn.4.40}\\
    \lim_{\ell\to 0}X''_\ell(1)&=2\int_\R\varphi'^2dx.\label{eqn.4.41}
\end{align}
Moreover, we have that
\begin{equation}\label{eqn.4.42}
\lim_{\ell\to 0}\frac{d}{dt}\norm{\eta_t^\ell}{2}{\G_\ell}^2=\lim_{\ell\to 0}\frac{d^2}{dt^2}\norm{\eta_t^\ell}{2}{\G_\ell}^2=0.
\end{equation}

\end{lemma}
\begin{proof}
The limits in \eqref{eqn.4.42} follow from the limits in \eqref{eqn.4.39} and \eqref{eqn.4.40} and taking $p=2$. 

We focus now on \eqref{eqn.4.39}. From Lemma  \ref{lem6.5}, the integrals over the interval $[0,\ell]$ in the expressions of $\Tilde{X}'_\ell(1)$, $\Tilde{X}''_\ell(1)$ and $X''_\ell(1)$ (see Lemma \ref{lemma derivadas lpequeno}) all tend to $0$ as $\ell\to 0$, by the Dominated Convergence Theorem. 

Using the explicit expression of $\psi$ in $[\ell,\infty)$, together with Lemma \ref{lem4.3} and the exponential decay of $\varphi$, an integration by parts yields
\begin{align*}
\lim_{\ell\to0}\Tilde{X}'_\ell(1)
&=\lim_{\ell\to0}2\int_\ell^\infty\frac{p}{2}\varphi^p(x+(y(\ell)-\ell))+x\frac{d}{dx}\psi^p(x+(y(\ell)-\ell))dx=2\left(\frac{p}{2}-1\right)\lim_{\ell\to0}\int_0^\infty\varphi^p(x+y(\ell))dx\\
&=\left(\frac{p}{2}-1\right)\int_\R\varphi^p(z)dz,
\end{align*}
where the last equality holds by the Dominated Convergence Theorem. With the same arguments, we have
\begin{align*}
    \lim_{\ell\to 0}\Tilde{X}''_\ell(1)
&=\lim_{\ell\to0}2\int_\ell^\infty\Big[\left(\frac{p}{4}(p-1)-\frac{p}{4}\right)\varphi^p(x+(y(\ell)-\ell)) +px\frac{d}{dx}\varphi^p(x+(y(\ell)-\ell))\nonumber\\
    &\qquad\qquad+px^2\varphi^{p-1}(x+(y(\ell)-\ell))\frac{d}{dx}\varphi'(x+(y(\ell)-\ell))\\&\qquad\qquad+p(p-1)x^2\varphi^{p-2}(x+(y(\ell)-\ell)){\varphi'}^2(x+(y(\ell)-\ell))\Big]dx\nonumber\\
&=\frac{1}{4}(p-4)(p-2)2\int_0^\infty\varphi^p(z)dz=\left(\frac{p}{2}-2\right)\left(\frac{p}{2}-1\right)\int_\R\varphi^p(z)dz.
\end{align*}

Finally, recalling that $\varphi'$ and $\varphi''$ have exponential decay and that $\displaystyle \lim_{\ell\to0}\ell^2\varphi'(y(\ell))\varphi''(y(\ell))=\lim_{\ell\to0}\ell{\varphi'}^2(y(\ell))=0$,  integrating by parts,
\begin{align*}
\lim_{\ell\to0}X''_\ell(1)
&=\lim_{\ell\to 0}2\int_\ell^\infty 6{\varphi'}^2+6x\frac{d}{dx}{\varphi'}^2+2x^2{\varphi''}^2+2x^2\varphi'\frac{d}{dx}\varphi''dx\\
&=2\lim_{\ell\to0}\int_\ell^\infty6{\varphi'}^2+4x\frac{d}{dx}{\varphi'}^2dx=4\lim_{\ell\to 0}\int_\ell^\infty{\varphi'}^2(x+y(\ell))dx=2\int_\R{\varphi'}^2dx\qedhere
\end{align*}
\end{proof} 

\begin{proposition}\label{prop:lpequeno}
    One has
    \begin{equation*}
        \lim_{\ell\to0}\left.\frac{d^2}{dt^2}E(\Tilde{\eta}_t^\varepsilon,\G_\ell)\right|_{t=1}=-\frac{1}{p}\left(\frac{p}{2}-1\right)\left(\frac{p-6}{2}\right)\int_\R \varphi^pdx.
    \end{equation*}
Thus, for $p>6$ and $\ell\sim0$,
\begin{equation*}
        \left.\frac{d^2}{dt^2}E(\Tilde{\eta}_t^\ell,\G_\ell)\right|_{t=1}<0.
    \end{equation*}
In particular, \eqref{GCI} holds.
\end{proposition}
\begin{proof}
Taking the limit, as $\ell\to 0$, in \eqref{secondderivative} and using the results in Lemma \ref{lemma.4.15} one obtains

\begin{equation*}
\lim_{\ell\to0}\left.\frac{d^2}{dt^2}E(\Tilde{\eta}_t^\varepsilon,\G_\ell)\right|_{t=1}=\int_\R{\varphi'}^2dx-\frac{1}{p}\left(\frac{p}{2}-2\right)\left(\frac{p}{2}-1\right)\int_\R \varphi^pdx,
\end{equation*}
and the conclusion follows from the Pohozaev identity. 
\end{proof}

\begin{proof}[Conclusion of the proof of Theorem \ref{Theo3.MAin}] As discussed in the beginning of Section \ref{sec:instability}, the existence of a $C^1$ path on $\mathcal{Q}$ satisfying \eqref{eqn.dernegativa} implies orbital instability. The result then follows from Lemma \ref{lem:curvasuave} (which ensures the regularity of the curve), together with either Proposition \ref{prop:lgrande} (for $\ell\sim\infty$) or Proposition \ref{prop:lpequeno} (for $\ell\sim0$).
\end{proof}

\subsection*{Data availability statements}
No new data were created or analysed in this study.

\subsection*{Conflict of interest}
The authors declare that they have no competing interest.

\subsection*{Acknowledgements}
We would like to thank the anonymous referees for their useful remarks on a previous version of the paper, which led to  Theorem \ref{TheoMain5.NonUniquenessEGS} and its proof, as well as to other important improvements. 

The authors were partially supported by the Portuguese government through FCT - Funda\c{c}\~ao para a  Ci\^encia e Tecnologia, I.P, through CAMGSD, IST-ID
(grant UID/MAT/04459/2020) and through the project NoDES (PTDC/MAT-PUR/1788/2020). F. Agostinho was also partially supported by Funda\c{c}\~ao para a Ci\^encia e Tecnologia, through the PhD grant UI/BD/150776/2020.

%\bibliographystyle{plain}
%\bibliography{Mybib}

\medbreak

	{\noindent Francisco Agostinho, Sim\~ao Correia and Hugo Tavares}\\
{\footnotesize
	Center for Mathematical Analysis, Geometry and Dynamical Systems,\\
	Department of Mathematics,\\
	Instituto Superior T\'ecnico, Universidade de Lisboa\\
	Av. Rovisco Pais, 1049-001 Lisboa, Portugal\\
	\texttt{francisco.c.agostinho@tecnico.ulisboa.pt}\\ \texttt{simao.f.correia@tecnico.ulisboa.pt}\\ \texttt{hugo.n.tavares@tecnico.ulisboa.pt}
}

\end{document}